\documentclass[10pt]{article}

\topskip  \usepackage{amsmath,amsthm,amssymb,cite,graphics,color}
\usepackage{ytableau,pstricks}
\usepackage{makeidx}

\makeindex

\def\FF{\mathcal{F}}

\begin{document}
\newtheorem{theorem}{Theorem}
\newtheorem{definition}[theorem]{Definition}
\newtheorem{lemma}[theorem]{Lemma}
\newtheorem{proposition}[theorem]{Proposition}
\newtheorem{corollary}[theorem]{Corollary}
\newtheorem{example}[theorem]{Example}
\title{The components of the singular locus of a component of a Springer fiber over $x^2=0$}
\author{Ronit Mansour\footnote{Supported by ISF grant 797/14} and Anna Melnikov}
\date{}
\maketitle
\begin{center}
Department of Mathematics, University of Haifa,  3498838  Haifa, Israel\\[4pt]
{\it hadilatil@gmail.com, anna.melnikov@gmail.com}
\end{center}
\begin{abstract}
For $x\in End(\mathbb K^n)$ satisfying $x^2=0$ let $\FF_x$ be the variety of full
flags stable under the action of $x$ (Springer fiber over $x$). The full classification of the components of $\FF_x$ according to their smoothness was provided in \cite{FM2009} in terms of both Young tableaux and link patterns. Moreover in \cite{F2012} the purely combinatorial algorithm to compute the singular locus of a singular components of $\FF_x$ is provided. However this algorithm involves the computation of the graph of the component, and the complexity of computations grows very quickly, so that in practice it is impossible to use it. In this paper, we construct another algorithm, derived from the algorithm of Fresse \cite{F2012}, providing all the components of the singular locus of a singular component of $\FF_x$ in terms of link patterns constructed straightforwardly from its link pattern.

\end{abstract}

\section{Introduction}
Throughout this paper, we set $\mathbb{K}$\index{$\mathbb{K}$, Algebraically closed field} to be an algebraically closed field of arbitrary characteristic. We set $V$\index{$V$, Vector space of finite dimension} to be a vector space of finite dimension $n$. A {\bf complete flag}\index{Complete flag} of $V$ is a chain $V_0\subset V_1\subset\cdots\subset V_n$ of subspaces of $V$ with $dim V_i=i$ for all $i=0,1,\ldots,n$. We denote the set of all the complete flags by $\mathcal{F}$.\index{$\mathcal{F}$, Set of all the complete flags}

\subsection{Springer fibers and their components}
Let $x$ be a nilpotent endomorphism of $V$.\index{$x$, Nilpotent endomorphism} A {\bf Springer fiber}\index{Springer iber} $\mathcal{F}_x$\index{$\mathcal{F}_x$, Springer fiber} is a subset of $x$-stable complete flags, that is, flags $(V_0,V_1,\ldots,V_n)$ such that $x(V_i)\subset V_{i-1}$ for all $i=1,2,\ldots,n$. Clearly, $\mathcal{F}_x$ is a closed subvariety of $\mathcal{F}$ (see \cite{S1977}) and depends on the Jordan form of $x$ only.  Note that $\mathcal{F}_x$  is reducible unless $x$ is zero or regular.
Different aspects of Springer fibers were studied by many authors. In particular many aspects of the study of Springer fibers and its connection to Schubert varieties is described in  the survey \cite{Tym2017}.

The main objects of our interest are the irreducible components of a Springer fiber. We concentrate on the case of $x$ satisfying $x^2=0$. In this case the components are described in terms of link patterns. We provide the algorithm describing all components of the singular locus of a singular component in terms of admissible pairs of a link pattern.

\subsection{Parametrization of the irreducible components of $\mathcal{F}_x$ by Young tableaux}\label{parametr_of_comp}
A nilpotent  endomorphism $x:V\mapsto V$ has a unique eigenvalue 0, so its Jordan  form can be written as the list of lengths of its Jordan blocks and since Jordan form is unique up to the order of Jordan blocks this list can be viewed as a partition of $n$. Put $\lambda(x):=\lambda=(\lambda_1,\ldots,\lambda_r)\vdash n$\index{$\lambda$, Partition}\index{$\lambda\vdash n$, Partition of $n$} to be this partition. A {\bf Young diagram}\index{Young diagram} $Y(x)=Y_\lambda$\index{$Y(x)=Y_\lambda$, Young diagram} is an array of $r$ rows of boxes starting on the left, with the $i$-th row containing  $\lambda_i$ boxes. For $\lambda=(\lambda_1,\ldots,\lambda_l)\vdash n$ and $\mu=(\mu_1,\ldots,\mu_k)\vdash m$ where $m<n$ we say that $\mu$ is a subdiagram of $\lambda$ if $k\leq l$ and for each $i\ :\ 1\leq i\leq k$ one has $\mu_i\leq \lambda_i$.

Let $\lambda^*$\index{$\lambda$, Conjugate partition}  denote the {\bf conjugate partition}\index{Conjugate partition}  of $\lambda$ that is the list of the lengths of columns in $Y_\lambda$. By \cite[Sec. II, 5.5]{S1982}, the dimension of $\mathcal{F}_x$ is given by
$\dim\mathcal{F}_x=\sum_{\mu_i\in\lambda^*(x)}\binom{\mu_i}{2},$ and in particular for $\lambda=(2^k,1^{n-2k})$ one has $\dim\mathcal{F}=\binom{n-k}{2}+\binom{k}{2}.$

Fill the boxes of $Y(x)$ with numbers $1,2,\ldots,n$ in such a way that numbers increase from left to right in the rows and from top to bottom in the columns. Such an array is called a {\bf  standard Young  tableau of shape}\index{Standard Young  tableau} $Y(x)$. We call it a {\bf tableau} in what follows. We denote
the set of all standard tableaux of shape $Y_\lambda$ by $Tab_\lambda$.\index{$Tab_\lambda$, Set of all standard tableaux of shape $Y_\lambda$}  Given  $T\in Tab_\lambda$ where $\lambda\vdash n$, put $T_{\{i\}}$ to be a subtableau of $T$ containing the entries $1,2,\ldots,i$, and respectively put $Y_i(T)$ to be its shape.\index{$T_{\{i\}}$, Subtableau of $T$ containing the entries $1$, $2$, $\ldots$, $i$}\index{$Y_i(T)$, Shape of the subtableau of $T$ containing the entries $1$, $2$, $\ldots$, $i$}

The  components of $\mathcal{F}_x$ are parametrized by standard tableaux of shape $Y(x)$, and we provide one of the possible ways of parametrization, following \cite{S1982}. Let $F=(V_0,\ldots,V_n)\in\mathcal{F}_x$, where $V_i$ is $x$-stable subspace of $V$. Thus, the restriction $x\mid_{V_i}$\index{$x\mid_{V_i}$, Restriction of the nipotent endomorphosim $x$} is a nilpotent endomorphism  and its Jordan form is represented by a shape $Y(x|_{V_i})$, which is a subdiagram of $Y(x|_{V_{i+1}})$ (and of $Y(x)$). Define\index{$\mathcal{V}_x^T$, Parittion of $\mathcal{F}_x$}
$$\mathcal{V}_x^T=\{(V_0,\ldots,V_n)\in\mathcal{F}_x\mid Y(x|_{V_i})=Y_i(T),\,i=1,2,\ldots,n\},$$
which is a partition of $\mathcal{F}_x$, namely, $\mathcal{F}_x=\coprod_{T\in Tab_\lambda}\mathcal{V}_x^T$.
By \cite[Sec. II.5.4-5]{S1982}, the set $\mathcal{V}_x^T$ is a locally closed, irreducible subset of $\mathcal{F}_x$ and $\dim\mathcal{V}_x^T=\dim\mathcal{F}_x$. Define \index{$\mathcal{F}_T=\overline{\mathcal{V}_x^T}$, Closure in Zariski topology} $\mathcal{F}_T=\overline{\mathcal{V}_x^T}$ to be the closure in Zariski topology. Then $\{\mathcal{F}_T\}_{T\in Tab_\lambda}$ are all the irreducible components of $\mathcal{F}_x$.\index{$\{\mathcal{F}_T\}_{T\in Tab_\lambda}$, Irreducible components of $\mathcal{F}_x$} This parametrization carries a lot of information on the components.

\subsection{Smooth and singular components of $\mathcal F_x$ in general and in case $x^2=0$}\label{1.3} Each $\mathcal{F}_x$ has at least one smooth component. Moreover, as it is shown in \cite{FM2010} all the components of $\mathcal{F}_x$ are smooth if and only if
$\lambda=(\lambda_1,n-\lambda_1)$ or $\lambda=(\lambda_1,n-\lambda_1-1,1)$ or
$\lambda=(\lambda_1,1^{n-\lambda_1})$ (or $\lambda=(2^3)$).

In all other cases there is at least one singular component, but in general the classification of the components according to their smoothness is a very difficult problem which is beyond our means. However, the nice exception is the case $x^2=0$. Here a full  classification of smooth components exists.

In this case all the block in Jordan form of $x$ are of length at most 2. Moreover the number of blocks of length 2 is equal to ${\rm Rank}\, x$ so that the components of $\mathcal F_x$ are parameterized by Young tableaux of shape $(2^k,1^{n-2k})$ where $k={\rm Rank}\, x.$

For $T\in Tab_\lambda$ where $\lambda\vdash n$ and $i\ :\ 1\leq i\leq n$ let $col(i)$ be the number of column $i$ belongs to.\index{$col(i)$, Number of column $i$ belongs to}
Put
$$\tau^*(T):=\{i\ :\ col(i+1)>col(i)\}.$$\index{$\tau^*(T)$, Set $\{i\ :\ col(i+1)>col(i)\}$}

For $T\in Tab_{(2^k,1^{n-2k})}$ let $T_1$ denote the first column\index{$T_i$, $i$-th column}
of $T$ and $T_2$ denote the second column of $T$. In this case
$\tau^*(T)=\{i\in T_1 :\ i+1\in T_2\}$.\index{$\tau^*(T)$, Set $\{i\in T_1 :\ i+1\in T_2\}$}

For a (finite) set  $S$ put $|S|$ to be its cardinality.\index{$|S|$, Cardinality of $S$}

For $T\in Tab_{(2^k,1^{n-2k})}$ let $T_1=(1=a_1,\ldots, a_{n-k})$ and $T_2=(b_1,\ldots, b_k)$. Put
\begin{align}
\rho(T):=\left\{\begin{array}{ll} |\tau^*(T)|+2, &{\rm if}\ a_{n-k}=n\ {\rm and}\ b_i>2i\ \forall\ i\,:\, 1\leq i\leq k;\\
|\tau^*(T)|+1, &{\rm if\ either}\ (a_{n-k}=n\ {\rm and}\ \exists\ i\, :\,  b_i=2i)\ {\rm or}\\
&\quad\quad\quad (b_k=n\ {\rm and}\ b_i>2i\ \forall\  1\leq i\leq k-1 );\\
|\tau^*(T)|, &{\rm if}\ b_k=n\ {\rm and}\ \exists\ i\, :\, 1\leq
i\leq k-1\ {\rm s.t.}\ b_i=2i. \end{array}\right.\label{eqrho1}
\end{align}\index{$\rho(T)$}
For instance, if
$$T=\ytableausetup{mathmode}
 \begin{ytableau}
  1 & 4\\
  2 & 5\\
  3 &9\\
  6&10\\
  7\\
  8\\
 \end{ytableau},$$
then $\tau^*(T)=\{3,8\}$, so that $|\tau^*(T)|=2$. Note that
$b_4=10$, $b_1=4>2$, $b_2=5>4$ and $b_3=9>6$. Hence,
$\rho(T)=|\tau^*(T)|+1=3$.

For $T\in Tab_{(2^k,1^{n-2k})}$ the simple criteria for $\mathcal{F}_T$ to be singular  is provided in terms of $\rho(T)$  by \cite[Theorem 1.2]{FM2009}:

\begin{theorem}\label{thm1}
For $T\in Tab_{(2^k,1^{n-k})}$  $\mathcal{F}_T$ is smooth if and only if $\rho(T)\leq 3$.
\end{theorem}

\subsection{The orbits of centralizer of $x$ for $x^2=0$ and link patterns}\label{seclinkp}
In order to study the components of $\FF_x$ we have to introduce $Z_x$-orbits.

Let  $GL(V)$ be the general linear group of $V$. Define $Z_x=\{g\in GL(V)\mid gxg^{-1}=x\}$\index{$Z_x$, Centralizer of $x$} to be the centralizer of $x$. So, for $g\in Z_x$ and $F\in\mathcal{F}_x$ one has $xg(F)=gx(F)$ that is $g(F)\in\mathcal{F}_x$ so that $Z_x$ naturally acts on the complete flags and leaves the Springer fiber $\mathcal{F}_x$ stable. However, $\mathcal{F}_x$ in general contains an infinite number of $Z_x$-orbits. The case $x^2=0$ is one of a 3 general cases where $\mathcal{F}_x$ is a union of finite number of $Z_x$-orbits (the other two cases are $\lambda=(\lambda_1,1^{n-\lambda_1})$ where all the components are smooth and $\lambda=(3,2^k,1^{n-3-2k})$ in which classification of the components according to their smoothness is far beyond our means).
%

For $x$ satisfying $x^2=0$  and ${\rm Rank}\, x=k$\  $Z_x$-orbits in $\mathcal{F}_x$ are labeled by involutions of symmetric group\index{$S_n$, Symmetric group of order $n$}\index{${\rm Rank}\,x$, Rank of $x$}
$S_n$  with $k$ disjoint 2-cycles as follows:

Let $I_n$ denote the set of all involutions in $S_n$,\index{$I_n$, Involutions in $S_n$} that is, $I_n=\{\sigma\in S_n\mid \sigma^2=id\}$. Let $I_{n,k}$\index{$I_{n,k}$, Involutions in $S_n$ with $k$ disjoint 2-cycles} denote the set of all involutions in $S_n$ with $k$ disjoint 2-cycles, that is
$$I_{n,k}=\{(i_1,j_1)\ldots(i_k,j_k)\in I_n\ :\ 1\leq i_s<j_s\leq n,\, \{i_s,j_s\}\cap\{i_t,j_t\}=\emptyset,\,1\leq s\ne t\leq k\}.$$

For $\sigma\in I_{n,k}$, a {\bf $\sigma$-basis}\index{$\sigma$-basis, Restriction of basis of $V$} $(v_1,\ldots,v_n)$ of $V$ is a basis of $V$ such that
$$x(v_i)=\left\{\begin{array}{ll} v_{\sigma(i)}&{\rm if}\ \sigma(i)<i;\\
0&{\rm otherwise}.\\
\end{array}\right.$$
For example, let $\{e_1,\ldots,e_5\}$ be a standard basis of $V$ and assume that $x$ with $\lambda(x)=(2,2,1)$ is defined by\index{$x(e_i)$}
$$x(e_i)=\left\{\begin{array}{ll}e_{i-1}&{\rm if}\ i=2,4;\\
0&{\rm otherwise}.\end{array}\right.$$
Let $\sigma=(1,5)(2,3)$ then some $\sigma$-bases are $\{e_1,e_3,e_4,e_5,e_2\},$ $\{e_3,e_1,e_2,e_5+e_3,e_4\},$ $\{e_1+e_3,e_1-2e_3,e_2-2e_4,e_5+e_1,e_2+e_4\}$ etc.

A {\bf $\sigma$-flag}\index{$\sigma$-flag, Complete flag} in $\mathcal{F}$ is a complete flag of the form $F=(Span\{ v_1,\ldots,v_{i}\}_{i=0}^n)$, where $(v_1,\ldots,v_n)$ is a $\sigma$-basis. We denote the set of all $\sigma$-flags by $\mathcal{Z}_\sigma$. Clearly, $\mathcal{Z}_\sigma\subset\mathcal{F}_x$ and it is $Z_x$-orbit of some $\sigma$-flag.
By \cite{F2012} for $x$ of square zero and ${\rm Rank}\, x=k$ one has
$$\mathcal F_x=\coprod\limits_{\sigma\in I_{n,k}}\mathcal Z_\sigma.$$

Since by the result above each $\mathcal F_T$ is the union of finite number of $Z_x$-orbits, in particular there exists a unique $\sigma_T\in I_{n,k}$ such that\index{$\sigma_T$, Involution}
$\mathcal F_T=\overline{\mathcal Z}_{\sigma_T}$. We use the combinatorics of $I_{n,k}$ in order to describe singular locus of singular $\mathcal F_T$ in terms of $\mathcal Z_\sigma$.\index{$\mathcal Z_\sigma$} A general algorithm for computing the singular locus is provided in \cite{F2012}. To do this Fresse constructs a  graph $G_T$,\index{$G_T$, Graph of Fresse} its vertexes are  $\{\sigma\in I_{n,k}\ :\ \mathcal Z_\sigma\subset \mathcal F_T\}$ (the combinatorial description of vertexes of $G_T$ is known (cf. \ref{2.1} for details)) and an edge connects two vertexes if one of them is obtained from another by so called  elementary move (also defined combinatorially (cf. \ref{secGG}). By \cite{F2012}
$\mathcal Z_\omega$ is in the singular locus of $\mathcal F_T$ if and only if in $G_T$ the number of edges $\{\omega,\omega'\}\in G_T$  is bigger than the number of edges $\{\sigma_T,\upsilon\}\in G_T$. As one can see at Figure \ref{figgg1} already for the first singular component (case $\lambda=(2^2,1^2)$) graph $G_T$ is big and complex.

\subsection{The algorithm for the components of a singular locus of $\FF_T$} In the paper we construct another algorithm, based on admissible pairs of $\sigma_T$ on
an interval. We do not use $G_T$ and construct all the $\omega_i$ such that $\{\mathcal Z_{\omega_i}\}_{i=1}^s$ are the components of the singular locus of $\FF_T$ straightforwardly from $\sigma_T$. For each pair of minimal arcs $(i,i+1),(j,j+1)$ in $\sigma_T$ satisfying special conditions we construct a component of the singular locus. Since we need more notation in order to explain the algorithm, the exact formula is provided in  \S\ref{maintheorem}.
Using this algorithm we show in particular that singular locus of $\FF_T$ is irreducible iff $\rho(T)=4$. In the subsequent paper we construct a natural stratification of $\FF_T$ with $\rho(T)=4$ and show that if $\rho(T)>5$ there is no well defined stratification.
\medskip

The structure of the paper is as follows. In \S 2 we provide all the notation and facts needed in what follows in order to make the paper self contained. Then is \S 3
we formulate the main theorem and explain its proof. In \S 4 we provide all the technical details of the proof. Finally in \S 5 we deduce the first results following from our algorithm. The list of notation is provided at the end of the paper.

\section{Preliminaries}
\subsection{Link patterns}\label{2.1}
For $\sigma=(i_1,j_1)\ldots(i_k,j_k)\in I_{n,k}$, its {\bf link pattern}\index{Link pattern} $P_\sigma$\index{$P_\sigma$\, Link pattern of $\sigma$}  is a graph with vertices $1,2,\ldots,n$ on a horizontal line and $k$ edges (arcs) connecting
$i_s$ and $j_s$, for all $s=1,2,\ldots,k$. From now and on, we identify the involution $\sigma$ with its link pattern $P_\sigma$ and write $\sigma$ for both of them. We set $|\sigma|=k$ to be the number of arcs in $\sigma$.

For  $\sigma\in I_{n,k}$, the point $i$, $1\leq i\leq n$, is called a {\bf fixed point}  if $\sigma(i)=i$ (the degree of the vertex $i$ in $\sigma$ is zero), and an
{\bf end point}  if $\sigma(i)\neq i$ (the degree of the vertex $i$ in $\sigma$ is one). In the later case, if $\sigma(i)>i$, $i$ is called a {\bf left end point} and if $\sigma(i)<i$ it is called a {\bf right end point}. We denote the set of all fixed points, left end points and right end points of $\sigma$, by $\sigma^0$, $\sigma^\ell$ and $\sigma^r$,\index{$\sigma^0$, Set of the fixed points of $\sigma$}\index{$\sigma^\ell$, Set of the left end points of $\sigma$}\index{$\sigma^r$, Set of the right end points of $\sigma$} respectively, that is,
$\sigma^0=\{i:i=\sigma(i)\}$, $\sigma^\ell=\{i:i<\sigma(i)\}$ and  $\sigma^r=\{i:i>\sigma(i)\}$.
We say $(i,j)\in\sigma$ if $i<j$ and $\sigma(i)=j$. For example, Figure \ref{figlp1} presents the link pattern $\sigma=(2,6)(3,8)(4,5)\in I_{9,3}$. Note that $1,7,9\in\sigma^0$, $2,3,4\in\sigma^{\ell}$ and $5,6,8\in\sigma^r$.
\begin{figure}[htp]
\begin{center}
\begin{pspicture}(3.3,.7)
\pscircle*(0,0){0.05}\pscircle*(.4,0){0.05}\pscircle*(.8,0){0.05}
\pscircle*(1.2,0){0.05}\pscircle*(1.6,0){0.05}\pscircle*(2,0){0.05}\pscircle*(2.4,0){0.05}
\pscircle*(2.8,0){0.05}\pscircle*(3.2,0){0.05}
\put(-.07,-.27){\scriptsize$1$}\put(.33,-.27){\scriptsize$2$}\put(.73,-.27){\scriptsize$3$}
\put(1.13,-.27){\scriptsize$4$}\put(1.53,-.27){\scriptsize$5$}\put(1.93,-.27){\scriptsize$6$}
\put(2.33,-.27){\scriptsize$7$}\put(2.73,-.27){\scriptsize$8$}\put(3.13,-.27){\scriptsize$9$}
\pscurve(.4,0)(1.2,.6)(2,0)\pscurve(.8,0)(1.9,.6)(2.8,0)
\pscurve(1.2,0)(1.4,.3)(1.6,0)
\end{pspicture}
\caption{The graph of $(2,6)(3,8)(4,5)\in I_{9,3}$}\label{figlp1}
\end{center}
\end{figure}
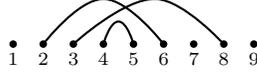

Let $\sigma$ be any link pattern. We say that $\sigma$  has a {\bf crossing}  if there is a pair of arcs $(i,j),(i',j')\in\sigma$ such that
$i<i'<j<j'$, we say that $(i',j')$ {\bf crosses} $(i,j)$ {\bf on the right} and $(i,j)$ {\bf crosses} $(i',j')$ {\bf on the left} in that case. For $(i,j)\in\sigma$,
let $c_{\sigma}^r((i,j)):=|\{(i',j')\in\sigma\ :\ i<i'<j<j'\}|$\index{$c_{\sigma}^r((i,j))$, Number of arcs crossing $(i,j)$ on the right} (that is  the number of arcs crossing $(i,j)$ on the right) and respectively
$c_{\sigma}^l((i,j)):=|\{(i',j')\in\sigma\ :\ i'<i<j'<j\}|$\index{$c_{\sigma}^l((i,j))$, Number of arcs crossing $(i,j)$ on the left} (that is the number of arcs crossing $(i,j)$ on the left). We put $c(\sigma)$ to be the total number of crossings in $\sigma$.\index{$c(\sigma))$, Number of total of crossings in $\sigma$} Obviously,
$$c(\sigma)=\sum\limits_{(i,j)\in\sigma}c_\sigma^r(i,j)=\sum\limits_{(i,j)\in\sigma}c_\sigma^l(i,j).$$
For instance, if $\sigma=(2,6)(3,8)(4,5)\in I_{9,3}$, see Figure \ref{figlp1}, then $c(\sigma)=1$, namely, the arc $(2,6)$ crosses the arc $(3,8)$.

For $a<b\in\mathbb N$ let $[a,b]:=\{i\}_{i\in \mathbb N,\ a\leq i\leq b}$ denote the set of integer points of interval $[a,b].$\index{$[a,b]$, Set of integer points of interval $[a,b]$}

We call the arc $(i,j)\in\sigma$ a {\bf bridge} over  $p\in\sigma^0$ if $i<p<j$. For $(i,j)\in\sigma$, put $b_\sigma((i,j)):=|\sigma^0\cap[i,j]|$ (that is number of $p\in\sigma^0$ such that $(i,j)$ is a bridge over them).\index{$b_\sigma((i,j))$, Number of fixed points below arc $(i,j)$}  Set $b(\sigma)$ to be the sum of number of bridges over all fixed points.\index{$b(\sigma)$, Sum of number of bridges over all fixed points} Obviously,
$$b(\sigma)=\sum\limits_{(i,j)\in\sigma}b_\sigma((i,j)).$$
For $f\in\sigma^0$ we also introduce for the future use notation $b_\sigma(f)$ -- the number of bridges over $f$.\index{$b_\sigma(f)$, Number of bridges over $f$} Obviously $b(\sigma)=\sum\limits_{f\in\sigma^0}b_\sigma(f)$. For instance,
if $\sigma=(2,6)(3,8)(4,5)\in I_{9,3}$, see Figure \ref{figlp1}, then $b(\sigma)=1$, namely, the arc $(3,8)$ is a bridge over the fixed point $7$.

For $a<b\in\mathbb N$ let $S_{[a,b]}$ denote the symmetric group on $\{a,a+1,\ldots,b\}$.\index{$S_{[a,b]}$, Symmetric group on $\{a,a+1,\ldots,b\}$}
For $1\leq a<b\leq n$ and $(i,j)\in\sigma$, we say that $(i,j)\in[a,b]$ if $a\leq i,j\leq b$. Let $\pi_{a,b}(\sigma):=\{(i,j)\in\sigma\cap[a,b]\}$\index{$\pi_{a,b}(\sigma)$, Involution $\{(i,j)\in\sigma\cap[a,b]\}$} and we consider it
as an involution of $S_{[a,b]}$.  Put $R_{[a,b]}(\sigma):=|\pi_{a,b}(\sigma)|$.\index{$R_{[a,b]}(\sigma)$, $|\pi_{a,b}(\sigma)|$} Let $R(\sigma)\in Mat_n(\mathbb Z^+)$ be the corresponding
(strictly upper triangular $n\times n$ matrix)\index{$R(\sigma)$, Strictly upper triangular $n\times n$ matrix}  defined by
$$(R(\sigma))_{i,j}:=\left\{\begin{array}{ll} R_{[i,j]}(\sigma)&{\rm if}\ i<j;\\
0&{\rm otherwise}\\
\end{array}\right.$$
For instance, if $\sigma=(1,4)(2,5)\in I_5$ then
$$R(\sigma)=\left(\begin{array}{lllll}
0&0&0&1&2\\
0&0&0&0&1\\
0&0&0&0&0\\
0&0&0&0&0\\
0&0&0&0&0
\end{array}\right).$$

We define a partial order on $Mat_n(\mathbb Z^+)$ by putting $A\leq B$\index{$A\leq B$, Partial order induces partial order on matrices} for $A,B\in Mat_n(\mathbb Z^+)$ if for any $1\leq i,j\leq n$ one has $A_{i,j}\leq B_{i,j}$. This partial order induces partial order on $I_n$ as follows. Let $\sigma,\upsilon\in I_n$ put $\sigma\geq \upsilon$ if $R(\sigma)\geq R(\upsilon)$.\index{$\sigma\geq \upsilon$, Partial order induces partial order on involutions}

Recall that for a partially ordered set $(S,\leq)$ the cover of $a\in S$ is\index{$Cov(a)$, Cover of $a$}
$$Cov(a):=\{b\in S\ :\ b<a\ {\rm and}\ a\leq c\leq b\, \Rightarrow\, (c=a\ {\rm or}\ c=b)\}.$$
For instance,
$$I_{5,2}=\left\{\begin{array}{c}(1,2)(3,4),\, (1,2)(3,5),\, (1,2)(4,5),\, (1,3)(2,4),\, (1,3)(2,5), (1,3)(4,5),\\
                                 (1,4)(2,3),\, (1,4)(2,5),\, (1,4)(3,5),\, (1,5)(2,3),\, (1,5)(2,4),\,(1,5)(3,4)\\
\end{array}\right\}.$$
The corresponding graph of link patterns is given in Figure \ref{fig5p}  where the first row consists of maximal link patterns and the last row contains the unique minimal link pattern of $I_{5,2}$. Edges between link patterns show the cover of a given link pattern in our relation.
\begin{figure}[htp]
\begin{center}
\begin{pspicture}(0,-.4)(8,6.3)
\put(2,6){\pscircle*(0,0){0.05}\pscircle*(.3,0){0.05}\pscircle*(.6,0){0.05}\pscircle*(.9,0){0.05}\pscircle*(1.2,0){0.05}
\put(-.1,-.3){$1$}\put(.2,-.3){$2$}\put(.5,-.3){$3$}\put(.8,-.3){$4$}\put(1.1,-.3){$5$}
\pscurve(0,0)(.45,.25)(.9,0)\pscurve(.3,0)(.45,.18)(.6,0)\put(-.08,-.6){\scriptsize $b=0,c=0$}}
\put(4,6){\pscircle*(0,0){0.05}\pscircle*(.3,0){0.05}\pscircle*(.6,0){0.05}\pscircle*(.9,0){0.05}\pscircle*(1.2,0){0.05}
\put(-.1,-.3){$1$}\put(.2,-.3){$2$}\put(.5,-.3){$3$}\put(.8,-.3){$4$}\put(1.1,-.3){$5$}
\pscurve(0,0)(.15,.2)(.3,0)\pscurve(.6,0)(.75,.2)(.9,0)\put(-.08,-.6){\scriptsize $b=0,c=0$}}
\put(6,6){\pscircle*(0,0){0.05}\pscircle*(.3,0){0.05}\pscircle*(.6,0){0.05}\pscircle*(.9,0){0.05}\pscircle*(1.2,0){0.05}
\put(-.1,-.3){$1$}\put(.2,-.3){$2$}\put(.5,-.3){$3$}\put(.8,-.3){$4$}\put(1.1,-.3){$5$}
\pscurve(0,0)(.15,.2)(.3,0)\pscurve(.9,0)(1.05,.2)(1.2,0)\put(-.08,-.6){\scriptsize $b=0,c=0$}}
\put(0,4){\pscircle*(0,0){0.05}\pscircle*(.3,0){0.05}\pscircle*(.6,0){0.05}\pscircle*(.9,0){0.05}\pscircle*(1.2,0){0.05}
\put(-.1,-.3){$1$}\put(.2,-.3){$2$}\put(.5,-.3){$3$}\put(.8,-.3){$4$}\put(1.1,-.3){$5$}
\pscurve(0,0)(.6,.25)(1.2,0)\pscurve(.3,0)(.45,.18)(.6,0)\put(-.08,-.6){\scriptsize $b=1,c=0$}}
\put(2,4){\pscircle*(0,0){0.05}\pscircle*(.3,0){0.05}\pscircle*(.6,0){0.05}\pscircle*(.9,0){0.05}\pscircle*(1.2,0){0.05}
\put(-.1,-.3){$1$}\put(.2,-.3){$2$}\put(.5,-.3){$3$}\put(.8,-.3){$4$}\put(1.1,-.3){$5$}
\pscurve(0,0)(.6,.25)(1.2,0)\pscurve(.6,0)(.75,.18)(.9,0)\put(-.08,-.6){\scriptsize $b=1,c=0$}}
\put(4,4){\pscircle*(0,0){0.05}\pscircle*(.3,0){0.05}\pscircle*(.6,0){0.05}\pscircle*(.9,0){0.05}\pscircle*(1.2,0){0.05}
\put(-.1,-.3){$1$}\put(.2,-.3){$2$}\put(.5,-.3){$3$}\put(.8,-.3){$4$}\put(1.1,-.3){$5$}
\pscurve(0,0)(.3,.2)(.6,0)\pscurve(.3,0)(.6,.2)(.9,0)\put(-.08,-.6){\scriptsize $b=0,c=1$}}
\put(6,4){\pscircle*(0,0){0.05}\pscircle*(.3,0){0.05}\pscircle*(.6,0){0.05}\pscircle*(.9,0){0.05}\pscircle*(1.2,0){0.05}
\put(-.1,-.3){$1$}\put(.2,-.3){$2$}\put(.5,-.3){$3$}\put(.8,-.3){$4$}\put(1.1,-.3){$5$}
\pscurve(0,0)(.15,.2)(.3,0)\pscurve(.6,0)(.9,.2)(1.2,0)\put(-.08,-.6){\scriptsize $b=1,c=0$}}
\put(8,4){\pscircle*(0,0){0.05}\pscircle*(.3,0){0.05}\pscircle*(.6,0){0.05}\pscircle*(.9,0){0.05}\pscircle*(1.2,0){0.05}
\put(-.1,-.3){$1$}\put(.2,-.3){$2$}\put(.5,-.3){$3$}\put(.8,-.3){$4$}\put(1.1,-.3){$5$}
\pscurve(0,0)(.3,.2)(.6,0)\pscurve(.9,0)(1.05,.2)(1.2,0)\put(-.08,-.6){\scriptsize $b=1,c=0$}}
\put(2,2){\pscircle*(0,0){0.05}\pscircle*(.3,0){0.05}\pscircle*(.6,0){0.05}\pscircle*(.9,0){0.05}\pscircle*(1.2,0){0.05}
\put(-.1,-.3){$1$}\put(.2,-.3){$2$}\put(.5,-.3){$3$}\put(.8,-.3){$4$}\put(1.1,-.3){$5$}
\pscurve(0,0)(.6,.25)(1.2,0)\pscurve(.3,0)(.6,.18)(.9,0)\put(-.08,-.6){\scriptsize $b=2,c=0$}}
\put(4,2){\pscircle*(0,0){0.05}\pscircle*(.3,0){0.05}\pscircle*(.6,0){0.05}\pscircle*(.9,0){0.05}\pscircle*(1.2,0){0.05}
\put(-.1,-.3){$1$}\put(.2,-.3){$2$}\put(.5,-.3){$3$}\put(.8,-.3){$4$}\put(1.1,-.3){$5$}
\pscurve(0,0)(.45,.22)(.9,0)\pscurve(.6,0)(.9,.22)(1.2,0)\put(-.08,-.6){\scriptsize $b=1,c=1$}}
\put(6,2){\pscircle*(0,0){0.05}\pscircle*(.3,0){0.05}\pscircle*(.6,0){0.05}\pscircle*(.9,0){0.05}\pscircle*(1.2,0){0.05}
\put(-.1,-.3){$1$}\put(.2,-.3){$2$}\put(.5,-.3){$3$}\put(.8,-.3){$4$}\put(1.1,-.3){$5$}
\pscurve(0,0)(.3,.2)(.6,0)\pscurve(.3,0)(.75,.2)(1.2,0)\put(-.08,-.6){\scriptsize $b=1,c=1$}}
\put(4,0){\pscircle*(0,0){0.05}\pscircle*(.3,0){0.05}\pscircle*(.6,0){0.05}\pscircle*(.9,0){0.05}\pscircle*(1.2,0){0.05}
\put(-.1,-.3){$1$}\put(.2,-.3){$2$}\put(.5,-.3){$3$}\put(.8,-.3){$4$}\put(1.1,-.3){$5$}
\pscurve(0,0)(.45,.22)(.9,0)\pscurve(.3,0)(.75,.22)(1.2,0)\put(-.08,-.6){\scriptsize $b=2,c=1$}}
\psline(4.5,.4)(2.6,1.3)\psline(4.5,.4)(4.6,1.3)\psline(4.5,.4)(6.6,1.3)
\psline(2.5,2.4)(.6,3.3)\psline(2.5,2.4)(2.6,3.3)\psline(2.5,2.4)(4.6,3.3)
\put(4,0){\psline(.5,2.4)(-1.4,3.3)\psline(.5,2.4)(2.6,3.3)\psline(.5,2.4)(4.6,3.3)}
\put(4,0){\psline(2.5,2.4)(.6,3.3)\psline(2.5,2.4)(2.6,3.3)\psline(2.5,2.4)(4.6,3.3)}
\psline(.5,4.4)(2.6,5.3)\psline(2.5,4.4)(4.6,5.3)\psline(4.5,4.4)(2.6,5.3)\psline(4.5,4.4)(4.6,5.3)
\psline(6.5,4.4)(4.6,5.3)\psline(6.5,4.4)(6.6,5.3)\psline(8.5,4.4)(6.6,5.3)
\end{pspicture}
\caption{Partial order relation of the link patterns of $I_{5,2}$}\label{fig5p}
\end{center}
\end{figure}
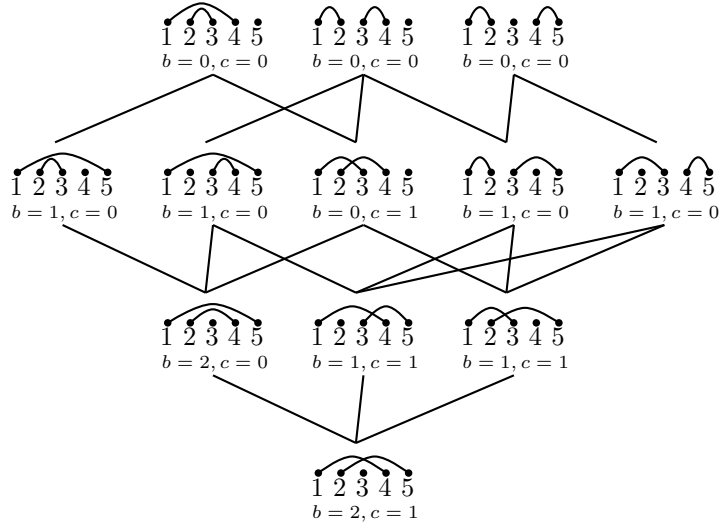

Put $\mathcal F_\sigma:=\overline{\mathcal Z_\sigma}$. By \cite{FM2009} and \cite{M2007} one has:

\begin{theorem}\label{thM007}
Let $x\in End(V)$ with $x^2=0$ and ${\rm Rank}\, x=k$. Then
\begin{itemize}
\item the map $\sigma\mapsto\mathcal{Z}_\sigma$ is a bijection between $I_{n,k}$ and the set of $Z_x$-orbits of the Springer fiber $\mathcal{F}_x$.
\item for  $\sigma\in I_{n,k}$, the dimension of $\mathcal{F}_\sigma$ is given by $\binom{n-k}{2}+\binom{k}{2}-b(\pi)-c(\pi)$. In particular, the irreducible components of $\mathcal{F}_x$ are  $\mathcal{F}_\sigma$ corresponding to $\sigma\in I_{n,k}$  without crossings and fixed points under an arc.
\item for any $\sigma,\upsilon\in I_{n,k}$, $\mathcal{F}_{\sigma}\supset\mathcal{F}_{\upsilon}$ if and only if $\sigma\geq \upsilon$.
\end{itemize}
In particular $\mathcal{Z}_{\omega_o}$,  where $\omega_o=(1,n-k+1)(2,n-k+2)\cdots(k,n)$,  is the unique closed $Z_x$-orbit of $\mathcal{F}_x$
(so that $\mathcal F_{\omega_o}=\mathcal Z_{\omega_o}$) and
$\dim\mathcal{F}_{\omega_o}=\binom{n-2k}{2}+\binom{k}{2}$ and this is the unique element of $I_{n,k}$ satisfying $\omega_o\leq\sigma$ for any $\sigma\in I_{n,k}.$
\end{theorem}
By this theorem and \cite{Tim1995} $\sigma'\in Cov(\sigma)$ iff ${\rm codim}_{\mathcal F_\sigma}\mathcal Z_{\sigma'}=1$.

Define $d_0:=\dim \mathcal F_{\omega_o}=\binom{n-2k}{2}+\binom{k}{2}$.\index{$d_0$}
Note that by the theorem each $\mathcal F_T$ contains the unique dense $Z_x$-orbit $\mathcal Z_\sigma$, so that there is a bijection between the components of $\mathcal F_x$ and involutions $\sigma$ without crossings and fixed points under an arc. For $T=(T_1,T_2)\in Tab_{(2^k,1^{n-2k})}$ let $T_1=(1,a_2,\ldots,a_{n-k})$ denote the first column of $T$ and $T_2=(b_1,\ldots,b_k)$ denote its second column.  $\sigma_T$ such that $\mathcal F_T=\mathcal F_{\sigma_T}$ is constructed as follows. We define\index{$\sigma_T$}
\begin{align}
\sigma_T=(i_1,b_1),\ldots,(i_k,b_k),\label{eqsigmaT}
\end{align}
where $i_1=b_1-1$ and $i_s=\max\{a_m\in T_1\setminus\{i_1,\ldots,i_{s-1}\}\ :\ a_m<b_s\}.$

 For instance, let
$$T=\ytableausetup{mathmode}
 \begin{ytableau}
  1 & 4\\
  2 & 5\\
  3 &9\\
  6&10\\
  7\\
  8\\
 \end{ytableau}.$$
One has $b_1=4$, $b_2=5$, $b_3=9$ and $b_4=10$. Thus, $i_1=b_1-1=3$, $i_2=\max\{a\in\{1,2,6,7,8\}|a<5\}=2$, $i_3=\max\{a\in\{1,6,7,8\}|a<9\}=8$ and $i_4=\max\{a\in\{1,2,6,7\}|a<10\}=7$. Thus, $\sigma_T=(3,4)(2,5)(8,9)(7,10)\in I_{10,4}$.

Given $\sigma=(i_1,j_1)\ldots(i_k,j_k)\in I_{(n,k)}$ where $j_1<j_2<\ldots j_k$ without crossings and fixed points under an arc,  $T\in Tab_{(2^k,1^{n-2k})}$ such that $\sigma_T=\sigma$ is $T=(T_1,T_2)$ where $T_2=(j_1,\ldots,j_k)$ and $T_1=\{i\}_{i=1}^n\setminus\{j_s\}_{s=1}^k$ written in increasing order. We  call $\sigma\in I_{n,k}$ without crossings and fixed point under an arc {\bf maximal} and denote by $I_{n,k}^{\max}$\index{$I_{n,k}^{\max}$, Subset of all maximal $\sigma$ in $I_{n,k}$} the subset of all maximal $\sigma$ in $I_{n,k}$.

For $\sigma\in I_{n,k}^{\max}$ put $\tau^*(\sigma):=\{i\ :\
(i,i+1)\in\sigma\}$.\index{$\tau^*(\sigma)$, Set $\{i\ :\
(i,i+1)\in\sigma\}$} One can easily see that
$\tau^*(\sigma_T)=\tau^*(T).$ Respectively, put
$$\rho(\sigma)=\left\{\begin{array}{ll}|\tau^*(\sigma)|+2&{\rm if}\ 1,n\in\sigma^0;\\
|\tau^*(\sigma)|+1 &{\rm if\ either}\ |\sigma^0\cap\{1,n\}|=1,\ {\rm or}\ (1,n)\in\sigma;\\
|\tau^*(\sigma)|&{\rm otherwise}.\\
\end{array}\right.$$
Again, $\rho(\sigma_T)=\rho(T)$.\index{$\rho(\sigma_T)$, $\rho(T)$}

\subsection{$G_{T}$ graphs and smoothness of $\mathcal{F}_T$}\label{secGG}
Given $\sigma\in I_{(n,k)}$, for $(i_1,j_1),\ldots, (i_s,j_s)\in \sigma$ put $\sigma_{(i_1,j_1),\ldots,(i_s,j_s)}^-\in I_{n,k-s}$ to be a link pattern obtained from $\sigma$ by deleting arcs $(i_1,j_1),\ldots, (i_s,j_s).$\index{$\sigma_{(i_1,j_1),\ldots,(i_s,j_s)}^-$, Link pattern obtained from $\sigma$ by deleting arcs $(i_1,j_1)$, $\ldots$, $(i_s,j_s)$} For $i<j\in\sigma^0$ put $(i,j)\sigma=\sigma(i,j)\in I_{n,k+1}$ to be a link pattern obtained from $\sigma$ by adding arc $(i,j).$\index{$\sigma(i,j)=(i,j)\sigma$, Link pattern obtained from $\sigma$ by adding arc $(i,j)$}

For $(i,j)\in\sigma$ and $f\in \sigma^0$ such that $(i,j)$ is bridge over $f$ we say that $\sigma'\in I_{n,k}$ is obtained from $\sigma$ by an {\bf elementary move forward of first type}\index{Elementary move forward of first type} if either $\sigma'=(i,f)\sigma^-_{(i,j)}$  or $\sigma'=(f,j)\sigma^-_{(i,j)}$. We say that $\sigma$ is obtained by an {\bf elementary move backward of first type}\index{Elementary move backward of first type} from $\sigma'$ in this case.

For $(i,j),(\ell,m)\in \sigma$ satisfying $i<\ell<j<m,$  we say that $\sigma'\in I_{n,k}$ is obtained from $\sigma$ by an {\bf elementary move forward of the second type}\index{Elementary move forward of the second type} if either $\sigma'=(i,\ell)(j,m)\sigma_{(i,j),(\ell,m)}^-$ or $\sigma'=(i,m)(j,\ell)\sigma_{(i,j),(\ell,m)}^-$. Again say that $\sigma$ is obtained by an {\bf elementary move backward of second type}\index{Elementary move backward of second type} from $\sigma'$ in this case.

We call $\sigma'$  a {\bf predecessor}\index{Predecessor} of $\sigma$ if is obtained from $\sigma$ by an elementary move forward. Note that  $\sigma'>\sigma$.

For $\sigma\in I_{n,k}$, graph $G_\sigma$ is a graph on vertices
$$\mathcal V_{G_\sigma}=\{\upsilon\in I_{n,k}\ :\ \upsilon \leq \sigma\}$$\index{$\mathcal V_{G_\sigma}$, Vertices of the graph $G_\sigma$}
and with edges
$$\mathcal E_{G_\sigma}=\{(\upsilon,\upsilon')\ :\ \upsilon,\upsilon'\in \mathcal V_{G_\pi}\ {\rm and\ one\ of \ them\ is\ a\ predecessor\ of\ another}\}.$$\index{$\mathcal E_{G_\sigma}$, Edges of the graph $G_\sigma$}
For $\upsilon\in \mathcal V_{G_\sigma}$, let $a_{\sigma,\upsilon}$ denote the number of edges of vertex $\upsilon$ in $G_\sigma$.\index{$a_{\sigma,\upsilon}$, Number of edges of vertex $\upsilon$ in $G_\sigma$} For instance, Figure \ref{figgg1} presents the graph $G_\sigma$, where $\sigma=(2,3)(4,5)\in I_{6,2}$. Column $i$ of the graph contains $\sigma'$ with ${\rm codim}_{\mathcal F_{(2,3)(4,5)}}\mathcal Z_{\sigma'}=i-1$.
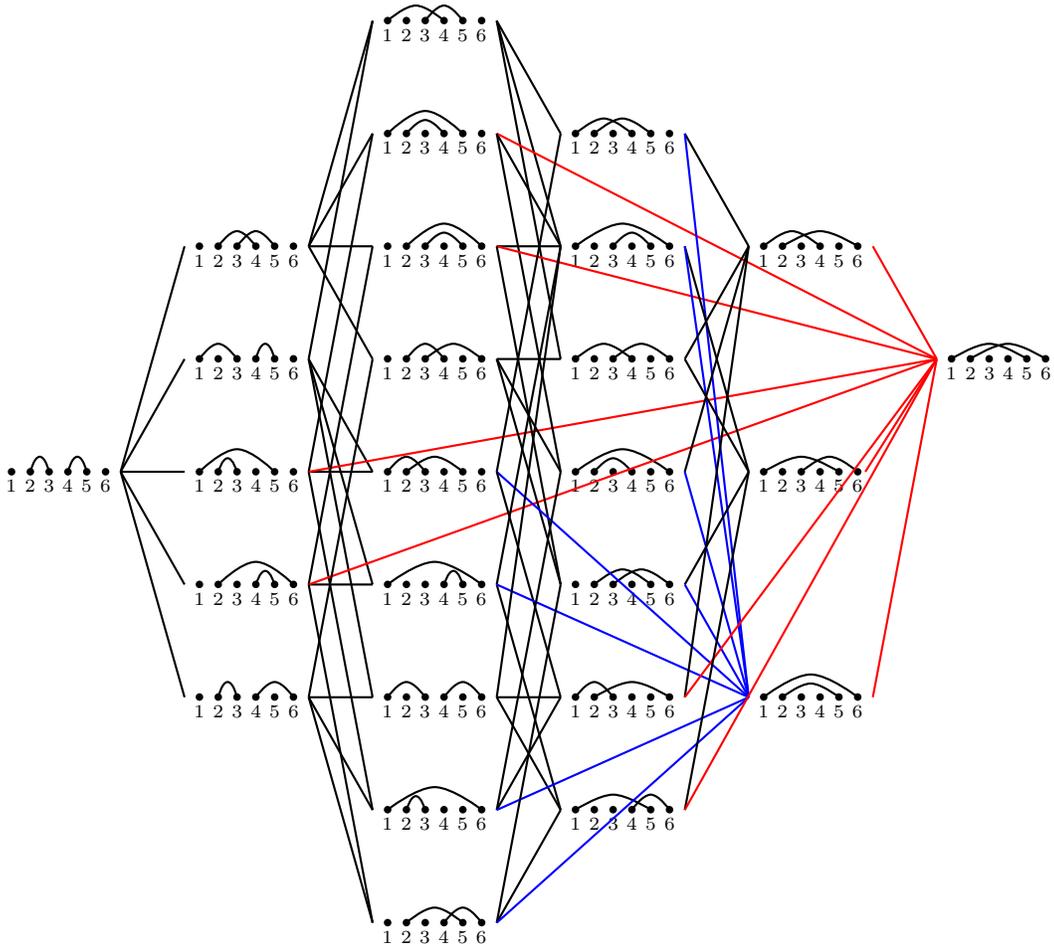
\begin{figure}[h!]
\begin{center}
\begin{pspicture}(0,-6)(13,6)
\put(0,0){\pscircle*(0,0){0.05}\pscircle*(.25,0){0.05}\pscircle*(.5,0){0.05}
\pscircle*(.75,0){0.05}\pscircle*(1,0){0.05}\pscircle*(1.25,0){0.05}
\put(-.07,-.27){\scriptsize$1$}\put(.18,-.27){\scriptsize$2$}\put(.43,-.27){\scriptsize$3$}
\put(.68,-.27){\scriptsize$4$}\put(.93,-.27){\scriptsize$5$}\put(1.18,-.27){\scriptsize$6$}
\pscurve(.25,0)(.375,.2)(.5,0)\pscurve(.75,0)(.875,.2)(1,0)}
\put(2.5,3){\pscircle*(0,0){0.05}\pscircle*(.25,0){0.05}\pscircle*(.5,0){0.05}\pscircle*(.75,0){0.05}\pscircle*(1,0){0.05}\pscircle*(1.25,0){0.05}
\put(-.07,-.27){\scriptsize$1$}\put(.18,-.27){\scriptsize$2$}\put(.43,-.27){\scriptsize$3$}
\put(.68,-.27){\scriptsize$4$}\put(.93,-.27){\scriptsize$5$}\put(1.18,-.27){\scriptsize$6$}
\pscurve(.25,0)(.5,.2)(.75,0)\pscurve(.5,0)(.75,.2)(1,0)}
\put(2.5,1.5){\pscircle*(0,0){0.05}\pscircle*(.25,0){0.05}\pscircle*(.5,0){0.05}\pscircle*(.75,0){0.05}\pscircle*(1,0){0.05}\pscircle*(1.25,0){0.05}
\put(-.07,-.27){\scriptsize$1$}\put(.18,-.27){\scriptsize$2$}\put(.43,-.27){\scriptsize$3$}
\put(.68,-.27){\scriptsize$4$}\put(.93,-.27){\scriptsize$5$}\put(1.18,-.27){\scriptsize$6$}
\pscurve(0,0)(.25,.2)(.5,0)\pscurve(.75,0)(.875,.2)(1,0)}
\put(2.5,0){\pscircle*(0,0){0.05}\pscircle*(.25,0){0.05}\pscircle*(.5,0){0.05}\pscircle*(.75,0){0.05}\pscircle*(1,0){0.05}\pscircle*(1.25,0){0.05}
\put(-.07,-.27){\scriptsize$1$}\put(.18,-.27){\scriptsize$2$}\put(.43,-.27){\scriptsize$3$}
\put(.68,-.27){\scriptsize$4$}\put(.93,-.27){\scriptsize$5$}\put(1.18,-.27){\scriptsize$6$}
\pscurve(0,0)(.5,.3)(1,0)\pscurve(.25,0)(.375,.18)(.5,0)}
\put(2.5,-1.5){\pscircle*(0,0){0.05}\pscircle*(.25,0){0.05}\pscircle*(.5,0){0.05}\pscircle*(.75,0){0.05}\pscircle*(1,0){0.05}\pscircle*(1.25,0){0.05}
\put(-.07,-.27){\scriptsize$1$}\put(.18,-.27){\scriptsize$2$}\put(.43,-.27){\scriptsize$3$}
\put(.68,-.27){\scriptsize$4$}\put(.93,-.27){\scriptsize$5$}\put(1.18,-.27){\scriptsize$6$}
\pscurve(.25,0)(.75,.3)(1.25,0)\pscurve(.75,0)(.875,.18)(1,0)}
\put(2.5,-3){\pscircle*(0,0){0.05}\pscircle*(.25,0){0.05}\pscircle*(.5,0){0.05}\pscircle*(.75,0){0.05}\pscircle*(1,0){0.05}\pscircle*(1.25,0){0.05}
\put(-.07,-.27){\scriptsize$1$}\put(.18,-.27){\scriptsize$2$}\put(.43,-.27){\scriptsize$3$}
\put(.68,-.27){\scriptsize$4$}\put(.93,-.27){\scriptsize$5$}\put(1.18,-.27){\scriptsize$6$}
\pscurve(.25,0)(.375,.2)(.5,0)\pscurve(.75,0)(1,.2)(1.25,0)}
\put(5,6){\pscircle*(0,0){0.05}\pscircle*(.25,0){0.05}\pscircle*(.5,0){0.05}\pscircle*(.75,0){0.05}\pscircle*(1,0){0.05}\pscircle*(1.25,0){0.05}
\put(-.07,-.27){\scriptsize$1$}\put(.18,-.27){\scriptsize$2$}\put(.43,-.27){\scriptsize$3$}
\put(.68,-.27){\scriptsize$4$}\put(.93,-.27){\scriptsize$5$}\put(1.18,-.27){\scriptsize$6$}
\pscurve(0,0)(.375,.2)(.75,0)\pscurve(.5,0)(.75,.2)(1,0)}
\put(5,4.5){\pscircle*(0,0){0.05}\pscircle*(.25,0){0.05}\pscircle*(.5,0){0.05}\pscircle*(.75,0){0.05}\pscircle*(1,0){0.05}\pscircle*(1.25,0){0.05}
\put(-.07,-.27){\scriptsize$1$}\put(.18,-.27){\scriptsize$2$}\put(.43,-.27){\scriptsize$3$}
\put(.68,-.27){\scriptsize$4$}\put(.93,-.27){\scriptsize$5$}\put(1.18,-.27){\scriptsize$6$}
\pscurve(0,0)(.5,.3)(1,0)\pscurve(.25,0)(.5,.18)(.75,0)}
\put(5,3){\pscircle*(0,0){0.05}\pscircle*(.25,0){0.05}\pscircle*(.5,0){0.05}\pscircle*(.75,0){0.05}\pscircle*(1,0){0.05}\pscircle*(1.25,0){0.05}
\put(-.07,-.27){\scriptsize$1$}\put(.18,-.27){\scriptsize$2$}\put(.43,-.27){\scriptsize$3$}
\put(.68,-.27){\scriptsize$4$}\put(.93,-.27){\scriptsize$5$}\put(1.18,-.27){\scriptsize$6$}
\pscurve(.25,0)(.75,.3)(1.25,0)\pscurve(.5,0)(.75,.18)(1,0)}
\put(5,1.5){\pscircle*(0,0){0.05}\pscircle*(.25,0){0.05}\pscircle*(.5,0){0.05}\pscircle*(.75,0){0.05}\pscircle*(1,0){0.05}\pscircle*(1.25,0){0.05}
\put(-.07,-.27){\scriptsize$1$}\put(.18,-.27){\scriptsize$2$}\put(.43,-.27){\scriptsize$3$}
\put(.68,-.27){\scriptsize$4$}\put(.93,-.27){\scriptsize$5$}\put(1.18,-.27){\scriptsize$6$}
\pscurve(.25,0)(.5,.2)(.75,0)\pscurve(.5,0)(.875,.2)(1.25,0)}
\put(5,0){\pscircle*(0,0){0.05}\pscircle*(.25,0){0.05}\pscircle*(.5,0){0.05}\pscircle*(.75,0){0.05}\pscircle*(1,0){0.05}\pscircle*(1.25,0){0.05}
\put(-.07,-.27){\scriptsize$1$}\put(.18,-.27){\scriptsize$2$}\put(.43,-.27){\scriptsize$3$}
\put(.68,-.27){\scriptsize$4$}\put(.93,-.27){\scriptsize$5$}\put(1.18,-.27){\scriptsize$6$}
\pscurve(0,0)(.25,.2)(.5,0)\pscurve(.25,0)(.625,.2)(1,0)}
\put(5,-1.5){\pscircle*(0,0){0.05}\pscircle*(.25,0){0.05}\pscircle*(.5,0){0.05}\pscircle*(.75,0){0.05}\pscircle*(1,0){0.05}\pscircle*(1.25,0){0.05}
\put(-.07,-.27){\scriptsize$1$}\put(.18,-.27){\scriptsize$2$}\put(.43,-.27){\scriptsize$3$}
\put(.68,-.27){\scriptsize$4$}\put(.93,-.27){\scriptsize$5$}\put(1.18,-.27){\scriptsize$6$}
\pscurve(0,0)(.625,.3)(1.25,0)\pscurve(.75,0)(.875,.18)(1,0)}
\put(5,-3){\pscircle*(0,0){0.05}\pscircle*(.25,0){0.05}\pscircle*(.5,0){0.05}\pscircle*(.75,0){0.05}\pscircle*(1,0){0.05}\pscircle*(1.25,0){0.05}
\put(-.07,-.27){\scriptsize$1$}\put(.18,-.27){\scriptsize$2$}\put(.43,-.27){\scriptsize$3$}
\put(.68,-.27){\scriptsize$4$}\put(.93,-.27){\scriptsize$5$}\put(1.18,-.27){\scriptsize$6$}
\pscurve(0,0)(.25,.2)(.5,0)\pscurve(.75,0)(1,.2)(1.25,0)}
\put(5,-4.5){\pscircle*(0,0){0.05}\pscircle*(.25,0){0.05}\pscircle*(.5,0){0.05}\pscircle*(.75,0){0.05}\pscircle*(1,0){0.05}\pscircle*(1.25,0){0.05}
\put(-.07,-.27){\scriptsize$1$}\put(.18,-.27){\scriptsize$2$}\put(.43,-.27){\scriptsize$3$}
\put(.68,-.27){\scriptsize$4$}\put(.93,-.27){\scriptsize$5$}\put(1.18,-.27){\scriptsize$6$}
\pscurve(0,0)(.625,.3)(1.25,0)\pscurve(.25,0)(.375,.18)(.5,0)}
\put(5,-6){\pscircle*(0,0){0.05}\pscircle*(.25,0){0.05}\pscircle*(.5,0){0.05}\pscircle*(.75,0){0.05}\pscircle*(1,0){0.05}\pscircle*(1.25,0){0.05}
\put(-.07,-.27){\scriptsize$1$}\put(.18,-.27){\scriptsize$2$}\put(.43,-.27){\scriptsize$3$}
\put(.68,-.27){\scriptsize$4$}\put(.93,-.27){\scriptsize$5$}\put(1.18,-.27){\scriptsize$6$}
\pscurve(.25,0)(.625,.2)(1,0)\pscurve(.75,0)(1,.2)(1.25,0)}
\put(7.5,4.5){\pscircle*(0,0){0.05}\pscircle*(.25,0){0.05}\pscircle*(.5,0){0.05}\pscircle*(.75,0){0.05}\pscircle*(1,0){0.05}\pscircle*(1.25,0){0.05}
\put(-.07,-.27){\scriptsize$1$}\put(.18,-.27){\scriptsize$2$}\put(.43,-.27){\scriptsize$3$}
\put(.68,-.27){\scriptsize$4$}\put(.93,-.27){\scriptsize$5$}\put(1.18,-.27){\scriptsize$6$}
\pscurve(0,0)(.375,.2)(.75,0)\pscurve(.25,0)(.625,.2)(1,0)}
\put(7.5,3){\pscircle*(0,0){0.05}\pscircle*(.25,0){0.05}\pscircle*(.5,0){0.05}\pscircle*(.75,0){0.05}\pscircle*(1,0){0.05}\pscircle*(1.25,0){0.05}
\put(-.07,-.27){\scriptsize$1$}\put(.18,-.27){\scriptsize$2$}\put(.43,-.27){\scriptsize$3$}
\put(.68,-.27){\scriptsize$4$}\put(.93,-.27){\scriptsize$5$}\put(1.18,-.27){\scriptsize$6$}
\pscurve(0,0)(.625,.3)(1.25,0)\pscurve(.5,0)(.75,.18)(1,0)}
\put(7.5,1.5){\pscircle*(0,0){0.05}\pscircle*(.25,0){0.05}\pscircle*(.5,0){0.05}\pscircle*(.75,0){0.05}\pscircle*(1,0){0.05}\pscircle*(1.25,0){0.05}
\put(-.07,-.27){\scriptsize$1$}\put(.18,-.27){\scriptsize$2$}\put(.43,-.27){\scriptsize$3$}
\put(.68,-.27){\scriptsize$4$}\put(.93,-.27){\scriptsize$5$}\put(1.18,-.27){\scriptsize$6$}
\pscurve(0,0)(.375,.2)(.75,0)\pscurve(.5,0)(.875,.2)(1.25,0)}
\put(7.5,0){\pscircle*(0,0){0.05}\pscircle*(.25,0){0.05}\pscircle*(.5,0){0.05}\pscircle*(.75,0){0.05}\pscircle*(1,0){0.05}\pscircle*(1.25,0){0.05}
\put(-.07,-.27){\scriptsize$1$}\put(.18,-.27){\scriptsize$2$}\put(.43,-.27){\scriptsize$3$}
\put(.68,-.27){\scriptsize$4$}\put(.93,-.27){\scriptsize$5$}\put(1.18,-.27){\scriptsize$6$}
\pscurve(0,0)(.625,.3)(1.25,0)\pscurve(.25,0)(.5,.18)(.75,0)}
\put(7.5,-1.5){\pscircle*(0,0){0.05}\pscircle*(.25,0){0.05}\pscircle*(.5,0){0.05}\pscircle*(.75,0){0.05}\pscircle*(1,0){0.05}\pscircle*(1.25,0){0.05}
\put(-.07,-.27){\scriptsize$1$}\put(.18,-.27){\scriptsize$2$}\put(.43,-.27){\scriptsize$3$}
\put(.68,-.27){\scriptsize$4$}\put(.93,-.27){\scriptsize$5$}\put(1.18,-.27){\scriptsize$6$}
\pscurve(.25,0)(.625,.2)(1,0)\pscurve(.5,0)(.875,.2)(1.25,0)}
\put(7.5,-3){\pscircle*(0,0){0.05}\pscircle*(.25,0){0.05}\pscircle*(.5,0){0.05}\pscircle*(.75,0){0.05}\pscircle*(1,0){0.05}\pscircle*(1.25,0){0.05}
\put(-.07,-.27){\scriptsize$1$}\put(.18,-.27){\scriptsize$2$}\put(.43,-.27){\scriptsize$3$}
\put(.68,-.27){\scriptsize$4$}\put(.93,-.27){\scriptsize$5$}\put(1.18,-.27){\scriptsize$6$}
\pscurve(0,0)(.25,.2)(.5,0)\pscurve(.25,0)(.75,.2)(1.25,0)}
\put(7.5,-4.5){\pscircle*(0,0){0.05}\pscircle*(.25,0){0.05}\pscircle*(.5,0){0.05}\pscircle*(.75,0){0.05}\pscircle*(1,0){0.05}\pscircle*(1.25,0){0.05}
\put(-.07,-.27){\scriptsize$1$}\put(.18,-.27){\scriptsize$2$}\put(.43,-.27){\scriptsize$3$}
\put(.68,-.27){\scriptsize$4$}\put(.93,-.27){\scriptsize$5$}\put(1.18,-.27){\scriptsize$6$}
\pscurve(0,0)(.5,.2)(1,0)\pscurve(.75,0)(1,.2)(1.25,0)}
\put(10,3){\pscircle*(0,0){0.05}\pscircle*(.25,0){0.05}\pscircle*(.5,0){0.05}\pscircle*(.75,0){0.05}\pscircle*(1,0){0.05}\pscircle*(1.25,0){0.05}
\put(-.07,-.27){\scriptsize$1$}\put(.18,-.27){\scriptsize$2$}\put(.43,-.27){\scriptsize$3$}
\put(.68,-.27){\scriptsize$4$}\put(.93,-.27){\scriptsize$5$}\put(1.18,-.27){\scriptsize$6$}
\pscurve(0,0)(.375,.2)(.75,0)\pscurve(.25,0)(.75,.2)(1.25,0)}
\put(10,0){\pscircle*(0,0){0.05}\pscircle*(.25,0){0.05}\pscircle*(.5,0){0.05}\pscircle*(.75,0){0.05}\pscircle*(1,0){0.05}\pscircle*(1.25,0){0.05}
\put(-.07,-.27){\scriptsize$1$}\put(.18,-.27){\scriptsize$2$}\put(.43,-.27){\scriptsize$3$}
\put(.68,-.27){\scriptsize$4$}\put(.93,-.27){\scriptsize$5$}\put(1.18,-.27){\scriptsize$6$}
\pscurve(0,0)(.5,.2)(1,0)\pscurve(.5,0)(.875,.2)(1.25,0)}
\put(10,-3){\pscircle*(0,0){0.05}\pscircle*(.25,0){0.05}\pscircle*(.5,0){0.05}\pscircle*(.75,0){0.05}\pscircle*(1,0){0.05}\pscircle*(1.25,0){0.05}
\put(-.07,-.27){\scriptsize$1$}\put(.18,-.27){\scriptsize$2$}\put(.43,-.27){\scriptsize$3$}
\put(.68,-.27){\scriptsize$4$}\put(.93,-.27){\scriptsize$5$}\put(1.18,-.27){\scriptsize$6$}
\pscurve(0,0)(.625,.3)(1.25,0)\pscurve(.25,0)(.625,.18)(1,0)}
\put(12.5,1.5){\pscircle*(0,0){0.05}\pscircle*(.25,0){0.05}\pscircle*(.5,0){0.05}\pscircle*(.75,0){0.05}\pscircle*(1,0){0.05}\pscircle*(1.25,0){0.05}
\put(-.07,-.27){\scriptsize$1$}\put(.18,-.27){\scriptsize$2$}\put(.43,-.27){\scriptsize$3$}
\put(.68,-.27){\scriptsize$4$}\put(.93,-.27){\scriptsize$5$}\put(1.18,-.27){\scriptsize$6$}
\pscurve(0,0)(.5,.2)(1,0)\pscurve(.25,0)(.75,.2)(1.25,0)}
\psline(1.45,0)(2.3,3)\psline(1.45,0)(2.3,1.5)\psline(1.45,0)(2.3,0)\psline(1.45,0)(2.3,-1.5)\psline(1.45,0)(2.3,-3)
\psline(3.95,3)(4.8,6)\psline(3.95,3)(4.8,4.5)\psline(3.95,3)(4.8,3)\psline(3.95,3)(4.8,1.5)
\psline(3.95,1.5)(4.8,6)\psline(3.95,1.5)(4.8,0)\psline(3.95,1.5)(4.8,-1.5)\psline(3.95,1.5)(4.8,-3)
\psline(3.95,0)(4.8,4.5)\psline(3.95,0)(4.8,0)\psline(3.95,0)(4.8,-4.5)\psline[linecolor=red](3.95,0)(12.3,1.5)
\psline(3.95,-1.5)(4.8,3)\psline(3.95,-1.5)(4.8,-1.5)\psline[linecolor=red](3.95,-1.5)(12.3,1.5)\psline(3.95,-1.5)(4.8,-6)
\psline(3.95,-3)(4.8,1.5)\psline(3.95,-3)(4.8,-3)\psline(3.95,-3)(4.8,-4.5)\psline(3.95,-3)(4.8,-6)
\psline(6.45,6)(7.3,4.5)\psline(6.45,6)(7.3,3)\psline(6.45,6)(7.3,1.5)
\psline(6.45,4.5)(7.3,3)\psline[linecolor=red](6.45,4.5)(12.3,1.5)\psline(6.45,4.5)(7.3,0)
\psline(6.45,3)(7.3,3)\psline[linecolor=red](6.45,3)(12.3,1.5)\psline(6.45,3)(7.3,-1.5)
\psline(6.45,1.5)(7.3,1.5)\psline(6.45,1.5)(7.3,0)\psline(6.45,1.5)(7.3,-1.5)
\psline(6.45,0)(7.3,4.5)\psline[linecolor=blue](6.45,0)(9.8,-3)\psline(6.45,0)(7.3,-3)
\psline(6.45,-1.5)(7.3,3)\psline(6.45,-1.5)(7.3,-4.5)\psline[linecolor=blue](6.45,-1.5)(9.8,-3)
\psline(6.45,-3)(7.3,3)\psline(6.45,-3)(7.3,-3)\psline(6.45,-3)(7.3,-4.5)
\psline(6.45,-4.5)(7.3,0)\psline(6.45,-4.5)(7.3,-3)\psline[linecolor=blue](6.45,-4.5)(9.8,-3)
\psline(6.45,-6)(7.3,-4.5)\psline(6.45,-6)(7.3,-1.5)\psline[linecolor=blue](6.45,-6)(9.8,-3)
\psline(8.95,4.5)(9.8,3)\psline[linecolor=blue](8.95,4.5)(9.8,-3)
\psline(8.95,3)(9.8,0)\psline[linecolor=blue](8.95,3)(9.8,-3)
\psline(8.95,1.5)(9.8,3)\psline(8.95,1.5)(9.8,0)
\psline(8.95,0)(9.8,3)\psline[linecolor=blue](8.95,0)(9.8,-3)
\psline(8.95,-1.5)(9.8,0)\psline[linecolor=blue](8.95,-1.5)(9.8,-3)
\psline(8.95,-3)(9.8,3)\psline[linecolor=red](8.95,-3)(12.3,1.5)
\psline(8.95,-4.5)(9.8,0)\psline[linecolor=red](8.95,-4.5)(12.3,1.5)
\psline[linecolor=red](11.45,3)(12.3,1.5)
\psline[linecolor=red](11.35,.0)(12.3,1.5)
\psline[linecolor=red](11.45,-3)(12.3,1.5)
\end{pspicture}
\caption{The graph $G_\sigma$, where $\sigma=(2,3)(4,5)\in S_6$.}\label{figgg1}
\end{center}
\end{figure}

Here we have $a_{(2,3)(4,5),(2,3)(4,5)}=5$ and $a_{(2,3)(4,5),(1,5)(2,6)}=9$.
In what follows we will write $\upsilon\in G_\sigma$ instead of $\upsilon\in\mathcal{V}_{G_\sigma}$ in order to simplify the notation.
Recall that a graph is called $p$-regular if every vertex has $p$ edges.
For a projective variety $\mathcal V$ and a point $F\in \mathcal V$, let  $\mathcal T_F(\mathcal V)$ denote the tangent space to $\mathcal V$ at point $F$.\index{$\mathcal T_F(\mathcal V)$, Tangent space to $\mathcal V$ at point $F$} By \cite{F2012} one has

\begin{theorem}
Let $x\in End(V)$ with $x^2=0$ and ${\rm Rank}\, x=k$. For $\sigma,\upsilon\in I_{n,k}$ with $\upsilon\in G_\sigma$, let $\mathcal{Z}_\sigma,\mathcal Z_\upsilon\subset\mathcal{F}_x$ be the corresponding $Z_x$-orbits. Let $F_\upsilon\in\mathcal Z_\upsilon$ denote a point of orbit $\mathcal Z_\upsilon$. Define $p_\sigma:=\binom{n-k}{2}-\binom{n-2k}{2}-b(\sigma)-c(\sigma)$.\index{$p_\sigma$}
\begin{itemize}
\item  One has $\dim \mathcal T_{F_\upsilon}(\mathcal F_\sigma)=a_{\sigma,\upsilon}+d_0$. In particular,
$a_{\sigma,\upsilon}\geq dim\mathcal{Z}_\sigma-d_0=a_{\sigma,\sigma}=p_\sigma$.
\item $\mathcal F_\sigma$ is smooth if and only if the graph $G_\sigma$ is $p_\sigma$-regular.
 Otherwise, the singular locus of $\mathcal F_\sigma$ is
$\coprod\limits_{\upsilon\in G_\sigma,\ a_{\sigma,\upsilon}>p_\sigma} \mathcal Z_{\upsilon}$.
\end{itemize}
\end{theorem}
For instance, one can see at Figure \ref{figgg1} that $\mathcal F_{(2,3)(4,5)}$ for $n=6$ is singular and its singular locus is
$\mathcal F_{(1,6)(2,5)}={\mathcal Z}_{(1,6)(2,5)}\sqcup {\mathcal Z}_{(1,5)(2,6)}.$ Indeed, $\dim \mathcal F_{(2,3)(4,5)}=7$ and $$\dim \mathcal T_{F_{(1,6)(2,5)}}
(\mathcal F_{(2,3)(4,5)})=\dim \mathcal T_{F_{(1,5)(2,6)}}(\mathcal F_{(2,3)(4,5)})=11.$$
Note that $\mathcal{F}_{(1,6)(2,5)}=\mathcal{Z}_{(1,6)(2,5)}\sqcup\mathcal{Z}_{(1,5)(2,6)}$ so that $\mathcal{F}_{(1,6)(2,5)}$ is the component of singular locus of $\mathcal{F}_{(2,3)(4,5)}$.

\section{Components of the singular locus for a given $\mathcal F_\sigma$}
\subsection{Admissible pair}\label{3.1}
Let $\sigma\in I_{n,k}$ and $1\leq s<t\leq n$, we say that $(i,j),(i',j')\in \sigma$ is an {\bf admissible pair} at $[s,t]$ if\index{Admissible pair}
\begin{itemize}
\item{} $s<i<j<i'<j'<t$;
\item{} $s,t\in \pi_{s,t}(\sigma)^0$;
\item{} for any $\ell: s<\ell<j$ or $i'<\ell<t$ one has $\ell\not\in \pi_{s,t}(\sigma)^0$, that is additional fixed points can be only at $[j,i']$;
\item{} for any $(r,q)\in \pi_{s,t}(\sigma)$ such that either $r<i<q$ or $r<j'<q$ or $r<j<i'<q$ one has $r<i$ and $j'<q$, that is, there are no arcs crossing $(i,j)$ or $(i',j')$ and every $(r,q)$ over either $(i,j)$ or $(i',j')$ is over both of them;
\item{} for any $(r,q),(r',q')\in \pi_{s,t}(\sigma)$ satisfying $r'<r<i<j'<q$, one has $q<q'$, that is, all the arcs over both $(i,j)$ and $(i',j')$ are concentric;
\item{} let $k$ be the number of fixed points at $[j,i']$; and $r$ the number of arcs over both $(i,j)$ and $(i',j')$, then $kr=0$.
\end{itemize}

For $\sigma\in I_{n,k}$ such that $\mathcal F_\sigma$ is singular let
$$Sing(\sigma):=\{\upsilon\ :\ \mathcal F_\upsilon\ {\rm is\ a\ component\ of\ the\ signular\ locus\ of}\ \mathcal F_\sigma\}.$$\index{$Sing(\sigma)$}

In order to formulate the main theorem we need the notions of maximal completion and concatenation, which we define in the next subsection.

\subsection{Maximal completion, concatenation and the main theorem}\label{maintheorem}
For $\sigma\in I_{n,k}$ let $\sigma_{+a}$ be obtained  by moving points $i \ :\ 1\leq i\leq n$   to $i+a$.\index{$\sigma_{+a}$, Link pattern obtained  by moving points $i \ :\ 1\leq i\leq n$  to $i+a$} That is $(i,j)\in\sigma$ is translated to $(i+a,j+a)\in\sigma_{+a}$. We can consider $\sigma_{+a}$ as an element of $I_{n+a,k}$, which means that we add $a$ fixed points to $\sigma$ on the left. We can consider $\sigma$ as an element of $I_{n+a,k}$ by adding $a$ fixed points on the right simply by writing $\sigma\in I_{n+a,k}$.

\begin{definition} Given $\sigma\in I_{n,k}$ with $2k<n$. Let $m_1=\min\, \sigma^0$ and $m_2=\max\, \sigma^0$. We call
$(1,m_1+1)\sigma_{+1}\in I_{n+1,k+1}$
{\bf $(1,0)$-maximal completion}\index{$(1,0)$-maximal completion} of $\sigma$ and $(m_2,n+1)\sigma\in I_{n+1,k+1}$ {\bf $(0,1)$-maximal completion} of $\sigma.$ If $\widehat\sigma\in I_{n+s_1+s_2,k+s_1+s_2}$ is obtained from $\sigma$ by series of $s_1$ maximal completions on the left and $s_2$ maximal completions on the right we  call it
{\bf $(s_1,s_2)$-maximal completion}.\index{$(s_1,s_2)$-maximal completion}
\end{definition}
Note that the order of completion is not important.
Since $\sigma\in I_{n,k}$ has $n-2k$ fixed points $(s,n-2k-s)$-maximal completions are the largest possible. We will call
$\sigma\in I_{2k,k}$ $k$-{\bf complete} link patterns.

\begin{definition}\label{def_cancatenation} For $\sigma\in I_n$ and $\upsilon\in I_m$ put $\sigma\upsilon_{+n}\in I_{n+m}$ to be their concatenation.\index{$\sigma\upsilon_{+n}$, Concatenation of $\sigma\in I_n$ and $\upsilon\in I_m$} It is obtained moving arcs of $\upsilon$ by $n$ and ``gluing'' the new set of arcs  to $\sigma$.
\end{definition}

Note that each maximal link pattern is a concatenation of complete maximal link patterns and fixed points between them.

\begin{theorem}\label{thm-main}
Consider $\sigma\in I_{n,k}^{\max}$ with $\rho(\sigma)\geq 4$. Let $S=\{(\{(i,j),(i',j')\},[s,t])\}$ be the set of all admissible pairs at all possible $[s,t]$. For  $x\in S$
let $a=\max(\sigma^0\cap[1,s])$ if it is not empty, and $a=0$ otherwise and $b=\min(\sigma^0\cap[t,n])$ if it is not empty and $b=n+1$ otherwise. Let $\upsilon=\pi_{s,t}(\sigma)^-_{(i,j),(i',j')}(i,j')(s,t)$.
\begin{itemize}
\item[(i)] If $a=s$ put $\upsilon_l=\pi_{1,a-1}(\sigma)$ and $l=0$ and $i=a-1$. Otherwise put
$\upsilon_l=\pi_{1,a}(\sigma)$ and $l=s-1-a$ and $i=a$;
\item[(ii)] If $b=t$ put $\upsilon_r=\pi_{b+1,n}(\sigma)$ and $r=0$ and $j=b$. Otherwise put $\upsilon_r=\pi_{b,n}(\sigma)$ and $r=b-t-1$ and $j=b-1$;
\end{itemize}
Let $\widehat\upsilon$ be $(l,r)$-completion of $\upsilon$ and $\upsilon(x)=\upsilon_l\widehat\upsilon_{+i}(\upsilon_r)_{+j}$. Then $Sing(\sigma)=\{\upsilon(x)\}_{x\in S}$. Moreover, $Sing(\sigma)$ contains  the unique element if and only if $\rho(\sigma)=4$.
\end{theorem}

For instance, there are $5$ elements in $Sing(\sigma)$ for $\sigma=(2,3)(4,5)(6,7)\in I_{8,3}^{\max}$, see Figure \ref{c2fig002}.
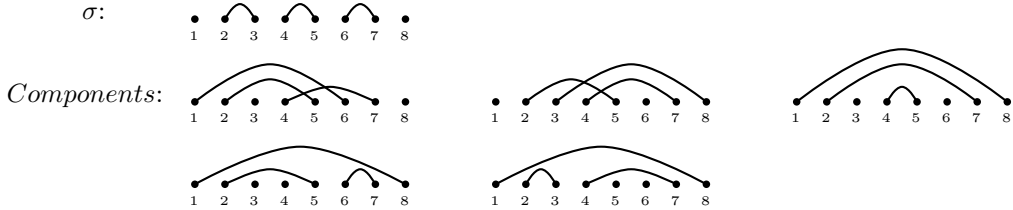
\begin{figure}[htp]
\begin{center}
\begin{pspicture}(0,-1.2)(9,1.6)
\put(0,1.1){\put(-1.5,0){$\sigma$:}
\pscircle*(0,0){0.05}\pscircle*(.4,0){0.05}\pscircle*(.8,0){0.05}
\pscircle*(1.2,0){0.05}\pscircle*(1.6,0){0.05}\pscircle*(2,0){0.05}
\pscircle*(2.4,0){0.05}\pscircle*(2.8,0){0.05}
\put(-.07,-.27){\tiny$1$}\put(.33,-.27){\tiny$2$}\put(.73,-.27){\tiny$3$}
\put(1.13,-.27){\tiny$4$}\put(1.53,-.27){\tiny$5$}\put(1.93,-.27){\tiny$6$}
\put(2.33,-.27){\tiny$7$}\put(2.73,-.27){\tiny$8$}
\pscurve(.4,0)(.6,.2)(.8,0)\pscurve(1.2,0)(1.4,.2)(1.6,0)\pscurve(2,0)(2.2,.2)(2.4,0)}
\put(0,0){\put(-2.5,0){$Components$:}
\pscircle*(0,0){0.05}\pscircle*(.4,0){0.05}\pscircle*(.8,0){0.05}
\pscircle*(1.2,0){0.05}\pscircle*(1.6,0){0.05}\pscircle*(2,0){0.05}
\pscircle*(2.4,0){0.05}\pscircle*(2.8,0){0.05}
\put(-.07,-.27){\tiny$1$}\put(.33,-.27){\tiny$2$}\put(.73,-.27){\tiny$3$}
\put(1.13,-.27){\tiny$4$}\put(1.53,-.27){\tiny$5$}\put(1.93,-.27){\tiny$6$}
\put(2.33,-.27){\tiny$7$}\put(2.73,-.27){\tiny$8$}
\pscurve(0,0)(1,.5)(2,0)\pscurve(.4,0)(1,.3)(1.6,0)\pscurve(1.2,0)(1.8,.2)(2.4,0)}
\put(4,0){\pscircle*(0,0){0.05}\pscircle*(.4,0){0.05}\pscircle*(.8,0){0.05}
\pscircle*(1.2,0){0.05}\pscircle*(1.6,0){0.05}\pscircle*(2,0){0.05}
\pscircle*(2.4,0){0.05}\pscircle*(2.8,0){0.05}
\put(-.07,-.27){\tiny$1$}\put(.33,-.27){\tiny$2$}\put(.73,-.27){\tiny$3$}
\put(1.13,-.27){\tiny$4$}\put(1.53,-.27){\tiny$5$}\put(1.93,-.27){\tiny$6$}
\put(2.33,-.27){\tiny$7$}\put(2.73,-.27){\tiny$8$}
\pscurve(.4,0)(1,.3)(1.6,0)\pscurve(.8,0)(1.8,.5)(2.8,0)\pscurve(1.2,0)(1.8,.3)(2.4,0)}
\put(8,0){\pscircle*(0,0){0.05}\pscircle*(.4,0){0.05}\pscircle*(.8,0){0.05}
\pscircle*(1.2,0){0.05}\pscircle*(1.6,0){0.05}\pscircle*(2,0){0.05}
\pscircle*(2.4,0){0.05}\pscircle*(2.8,0){0.05}
\put(-.07,-.27){\tiny$1$}\put(.33,-.27){\tiny$2$}\put(.73,-.27){\tiny$3$}
\put(1.13,-.27){\tiny$4$}\put(1.53,-.27){\tiny$5$}\put(1.93,-.27){\tiny$6$}
\put(2.33,-.27){\tiny$7$}\put(2.73,-.27){\tiny$8$}
\pscurve(0,0)(1.4,.7)(2.8,0)\pscurve(.4,0)(1.4,.5)(2.4,0)\pscurve(1.2,0)(1.4,.2)(1.6,0)}
\put(0,-1.1){\pscircle*(0,0){0.05}\pscircle*(.4,0){0.05}\pscircle*(.8,0){0.05}
\pscircle*(1.2,0){0.05}\pscircle*(1.6,0){0.05}\pscircle*(2,0){0.05}
\pscircle*(2.4,0){0.05}\pscircle*(2.8,0){0.05}
\put(-.07,-.27){\tiny$1$}\put(.33,-.27){\tiny$2$}\put(.73,-.27){\tiny$3$}
\put(1.13,-.27){\tiny$4$}\put(1.53,-.27){\tiny$5$}\put(1.93,-.27){\tiny$6$}
\put(2.33,-.27){\tiny$7$}\put(2.73,-.27){\tiny$8$}
\pscurve(0,0)(1.4,.5)(2.8,0)\pscurve(.4,0)(1,.2)(1.6,0)\pscurve(2,0)(2.2,.2)(2.4,0)}
\put(4,-1.1){\pscircle*(0,0){0.05}\pscircle*(.4,0){0.05}\pscircle*(.8,0){0.05}
\pscircle*(1.2,0){0.05}\pscircle*(1.6,0){0.05}\pscircle*(2,0){0.05}
\pscircle*(2.4,0){0.05}\pscircle*(2.8,0){0.05}
\put(-.07,-.27){\tiny$1$}\put(.33,-.27){\tiny$2$}\put(.73,-.27){\tiny$3$}
\put(1.13,-.27){\tiny$4$}\put(1.53,-.27){\tiny$5$}\put(1.93,-.27){\tiny$6$}
\put(2.33,-.27){\tiny$7$}\put(2.73,-.27){\tiny$8$}
\pscurve(0,0)(1.4,.5)(2.8,0)\pscurve(.4,0)(.6,.2)(.8,0)\pscurve(1.2,0)(1.8,.2)(2.4,0)}
\end{pspicture}
\caption{Link patterns $\sigma=(2,3)(4,5)(6,7)\in I_{8,3}^{\max}$ and its components}\label{c2fig002}
\end{center}
\end{figure}

As we show in Subsection \ref{sect_remarks} for $\sigma\in I_{n,k}^{\max}$ such that $\rho(\sigma)>4$ the number of elements in $Sing(\sigma)$ depends not only on $\rho(\sigma)$ but on the structure of $\sigma$, and this happens exactly because the same pair can be admissible at different intervals, depending on the structure of $\sigma$. In particular, if $\rho(\sigma)=5$ then the number of elements in $Sing(\sigma)$ is between 3 and 5.

To prove this theorem we  start in Section \ref{basic} with construction of $\upsilon\in Sing(\sigma)$ using an admissible pair on $[1,n]$, then in Section \ref{gencase} we show  that for $\sigma\in I_{2k,k}$ a maximal $(s-1,n-t)$-completion of  $\upsilon\in Sing(\pi_{s,t}(\sigma))$ is in $Sing(\sigma)$ and that for $\sigma\in I_{n,k}$, $\upsilon\in Sing(\sigma)$ and  $\sigma'\in I_{m,k}$ concatenation $\upsilon\sigma'_{+n}\in Sing(\sigma\sigma'_{+n})$ and $\sigma'\upsilon_{+m}\in Sing(\sigma'\sigma_{+m})$. Finally, in Section
\ref{sec2.4} we show that for $\sigma\in I_{n,k}^{\max}$ with $\rho(\sigma)\geq 4$,  $\upsilon\in Sing(\sigma)$ iff it is obtained by the algorithm above.

As an immediate corollary of this result and its proof we get

\begin{corollary} For $\sigma\in I_{n,k}^{\max}$ with $\rho(\sigma)\geq 4$ and for $\upsilon\in Sing(\sigma)$, one has
$${\rm codim}_{\mathcal F_\sigma}\mathcal F_\upsilon\geq 4\quad{and}\quad \dim \mathcal T_{F_\upsilon}(\mathcal F_\sigma)=
{\rm codim}_{\mathcal F_\sigma}\mathcal F_\upsilon+\dim\mathcal F_\sigma.$$
\end{corollary}

\section{Proof of Theorem \ref{thm-main}}
\subsection{Basic case}\label{basic} Let us first show the following basic case

\begin{theorem}\label{thm_basic}
Let $\sigma$ be a link pattern such that $(i,j),(i',j')$ is an admissible pair at $[1,n]$. Put $\upsilon:=(1,n)(i,j')\sigma^-_{(i,j)(i',j')}$. Then $\mathcal F_\upsilon$ is a component of the singular locus of $\mathcal F_\sigma.$
\end{theorem}

Graphically one has (where all other points and arcs are drawn by points)
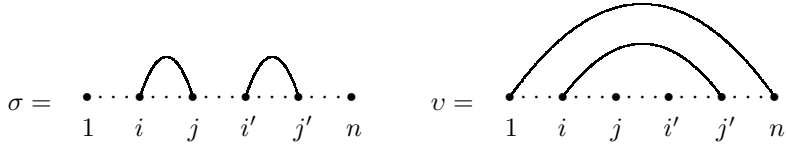
\begin{center}
\begin{picture}(150,60)(80,20)
\put(-40,35){$\sigma=$}
\multiput(-10,40)(20,0){6} {\circle*{3}}
\multiput (-5,40)(5,0){19}{\circle*{1}}
\put(-12,25){$1$}\put(8,25){$i$}\put(28,25){$j$}
\put(48,25){$i'$}\put(68,25){$j'$}\put(88,25){$n$}
\qbezier(10,40)(20,70)(30,40)\qbezier(50,40)(60,70)(70,40)
\put(120,35){$\upsilon=$}
\multiput(150,40)(20,0){6} {\circle*{3}} 
\multiput (155,40)(5,0){19}{\circle*{1}}
\put(148,25){$1$}\put(168,25){$i$}
\put(188,25){$j$}\put(208,25){$i'$}\put(228,25){$j'$}\put(248,25){$n$}
\qbezier(150,40)(200,110)(250,40)\qbezier(170,40)(200,80)(230,40)
\end{picture}
\end{center}
To prove the theorem we show in Proposition \ref{prop1} below that if $(i,j),(i',j')$ is an admissible pair then $F_\upsilon$ is a singular point of  $\mathcal F_\sigma$, and then in Proposition \ref{prop2} we show that for any $\upsilon'\ :\ \upsilon<\upsilon'<\sigma$ one has  $\dim \mathcal T_{F_\upsilon'}(\mathcal F_\sigma)=\dim \mathcal F_\sigma$ so that $\upsilon\in Sing(\sigma)$.

\begin{proposition}\label{prop1}
Let $\sigma\in I_{n,k}^{\max}$ be a link pattern such that $(i,j),(i',j')$ is an admissible pair at $[1,n]$. Put $\upsilon=(1,n)(i,j')\sigma^-_{(i,j)(i',j')}$. Then $F_\upsilon$ is a singular point of $\mathcal{F}_\sigma$. Let
\begin{itemize}
\item{} $k$ be the number of fixed points of $\sigma$ on $[j,i']$
\item{} $R$ be the set of arcs of $\sigma$ over both $(i,j)$ and $(i',j')$ that is  $R:=\{(s,t)\in \sigma\ :\ s<i,\ j'<t\}$, and $r=|R|$.
\end{itemize}
Then
$$\dim\mathcal{T}_{F_\upsilon}(\mathcal F_\sigma)=\dim\mathcal{F}_\sigma+(k+1)(2r+2)+2.$$
\end{proposition}
\begin{proof}
At first, let us show
$${\rm codim}_{\mathcal F_\sigma}\mathcal{F}_\upsilon=4+2(k+r).\eqno{(*)}$$
\begin{itemize}
\item{} For $(s,t)\in\sigma^-_{(i,j)(i',j')}$ one has
$c_\upsilon^\ell(s,t)=c_\sigma^\ell(s,t)$ because there are no arcs crossing $(i,j)$ and $(i',j')$.
\item{} For $(s,t)\in\sigma^-_{(i,j)(i',j')}$ one has
$$b_\upsilon(s,t)=\left\{\begin{array}{ll} b_\sigma(s,t)+2&{\rm if}\ (s,t)\in R;\\
                                       b_\sigma(s,t)&{\rm otherwise};\\
\end{array}\right.$$
\item{} $c_\upsilon^\ell(i,j')=c_\sigma^\ell(i,j)=c_\upsilon^\ell(1,n)=0$;
\item{} $b_\upsilon(i,j')=b_\upsilon(1,n)=k+2$ and $b_\sigma(i,j)=b_\sigma(i',j')=0$.
\end{itemize}
Thus
\begin{align*}
b(\upsilon)+c(\upsilon)&=b_\upsilon(i,j')+b_\upsilon(1,n)+\sum\limits_{(s,t)\in\sigma^-_{(i,j)(i',j')}}b_\upsilon(s,t)
+\sum\limits_{(s,t)\in\sigma^-_{(i,j)(i',j')}}
c_\upsilon^\ell(s,t)\\
&=2(k+2)+b(\sigma)+2r+c(\sigma)=b(\sigma)+c(\sigma)+4+2(k+r),
\end{align*}
which provides $(*)$.

Now, let us compute $\dim \mathcal T_{F_\upsilon}(\mathcal F_\sigma)$. In order to do that, we have to find the number of all $\sigma'$ such that $\upsilon<\sigma'\leq\sigma$, where $\sigma'$ is a predecessor of $\upsilon$. Such $\sigma'$ can be obtained by one of the following ways:
\begin{itemize}
\item[(i)]  $\sigma'=\upsilon_{(s,t)}^-(s,t')$ where $(s,t)\in \{(i,j'),(1,n)\}\sqcup R$ and $t'$ is any of fixed points of $\upsilon$ at $[j,i']$.  The number of such
$\sigma'$ is
$$(k+2)(r+2).$$
\item[(ii)] $\sigma'=\upsilon_{(s,t)}^-(s',t)$ where $(s,t)\in\{(i,j'),(1,n)\}\sqcup R$ and $s'$ is any of fixed points of $\upsilon$ at $[j,i']$. The number of such
$\sigma'$ is $$(k+2)(r+2).$$
\end{itemize}
These are all  predecessors of $\upsilon$ which are smaller than $\sigma$. Recall that $\dim \mathcal T_{F_\upsilon} (\mathcal F_\sigma)$ is equal to $\dim\mathcal{F}_\upsilon$ plus the number of predecessors which are smaller than $\sigma$. So we get
$\dim \mathcal T_{F_\upsilon} (\mathcal F_\sigma)=\dim \mathcal F_\upsilon+(k+2)(2r+4)$.
Hence, taking into account $(*)$ we get
\begin{align*}
\dim \mathcal T_{F_\upsilon}(\mathcal F_\sigma)&=\dim \mathcal F_\upsilon+(k+2)(2r+4)=\dim F_\sigma-(4+2(k+r))+(k+2)(2r+4)\\
&=\dim \mathcal F_\sigma+(k+1)(2r+2)+2.
\end{align*}
Therefore, $F_\upsilon$ is a singular point of $\mathcal F_\sigma$.
\end{proof}

Now let us show

\begin{proposition}\label{prop2} Let $(i,j),(i',j')\in \sigma$ be  admissible at $[1,n]$.
Let $$\upsilon=\sigma_{(i,j),(i',j')}^-(i,j')(1,n).$$
Then for any $\upsilon'$ such that $\sigma\geq\upsilon'>\upsilon$ one has   $\dim {\mathcal T}_{F_{\upsilon'}}(\mathcal F_\sigma)= \dim \mathcal F_\sigma$, so that $\upsilon\in Sing(\sigma).$
\end{proposition}
\begin{proof}
Since for any $\upsilon'\ :\ \sigma\geq\upsilon'>\upsilon$ there exists $\upsilon''\ :\ \upsilon'\geq \upsilon''>\upsilon$ such that $$\dim\mathcal F_{\upsilon''}=\dim\mathcal F_\upsilon+1$$ and
$$\dim \mathcal T_{F_{\upsilon'}}(\mathcal F_\sigma)\leq \dim \mathcal T_{F_{\upsilon''}}(\mathcal F_\sigma)$$
it is enough to check that the proposition is true only for $\upsilon'$ predecessors of $\upsilon$ such that $\sigma>\upsilon'$ and $\dim\mathcal{F}_{\upsilon'}=\dim\mathcal{F}_\upsilon+1$.

Since $\{(i,j),(i',j')\}$ is admissible at $[1,n]$, there are no fixed points in $\upsilon$
outside of $[j,i']$ so that the minimal fixed point of $\upsilon$ is $j$ and the maximal fixed point of $\upsilon$ is $i'$. Thus, according to the proof of Proposition \ref{prop1} the predecessors of $\upsilon$ smaller than $\sigma$ can be only  of types (i) and (ii) namely:
$$\upsilon_1=\upsilon^-_{(s,t)}(s,i'),\quad{\rm or}\quad \upsilon_2=\upsilon^-_{(s,t)}(j,t)$$
where $(s,t)\in R\sqcup\{(i,j'),(1,n)\}.$

Exactly as in the proof of Proposition \ref{prop1} we have to compare ${\rm codim}_{\mathcal F_\sigma}\mathcal F_{\upsilon_m}$  and
$\dim \mathcal T_{F_{\upsilon_m}}(\mathcal F_\sigma)$, where $m=1,2$.
By $(*)$ in the proof of Proposition \ref{prop1}, we have
${\rm codim}_{\mathcal F_\sigma}\mathcal{F}_{\upsilon_m}=3+2(k+r)$.
The computations of $\dim \mathcal T_{F_{\upsilon_m}}(\mathcal F_\sigma)$ are exactly the same for both cases and are very similar to the computations in the proof of Proposition \ref{prop1}. For example, we compute $\dim\mathcal{T}_{\upsilon_1}(\mathcal F_\sigma)$.

We have to compute the number of all $\sigma'\ :\ \sigma\geq\sigma'$,
where $\sigma'$ is a predecessor of $\upsilon_1$. They are obtained as follows:
\begin{itemize}
\item{} $\sigma'=(\upsilon_1)_{(s,i')}^-(s,f)$ for $f$ a fixed point of  $[j,i'-1]$. The number of such $\sigma'$ is $k+1$.
\item{} $\sigma'=(\upsilon_1)_{(s',t')}^-(f,t')$ for $f$ a fixed point of  $[j,i'-1]$ and $$(s',t')\in \{(1,n),(i,j')\}\sqcup R\backslash\{(s,t)\}.$$
The number of such $\sigma'$ is $(k+1)(r+1)$.
\item{} $\sigma'=(\upsilon_1)_{(s',t'),(s,i')}^-(s',i')(s,t')$ for $(s',t')\in \{(1,n),(i,j')\}\sqcup R$ such that $s'>s.$ Let us denote the number of such $\sigma'$
by $a$.
\item{} $\sigma'=(\upsilon_1)_{(s',t')}^-(s',t)$ for $(s',t')\in \{(1,n),(i,j')\}\sqcup R$ such that $t'>t.$ Note that since all the arcs of $R$ are concentric we get that
$(s',t')$ satisfies $s'<s,$ so that the number of such $\sigma'$ is $r+1-a$.
\item{} Finally, note that for $\upsilon'=(\upsilon_1)_{(s,i')}^-(f,i')$ or $\upsilon'=(\upsilon_1)_{(s',t')}^-(s',f)$ for $f\in[j,i'-1]$ and $(s',t')\in \{(1,n),(i,j')\}\sqcup R$ such that $s'\neq s$ one has $\upsilon'\not<\sigma$.
\end{itemize}
All other predecessors also provide us with $\sigma'\not<\sigma.$

All together this shows that
$$\begin{array}{ll}\dim \mathcal T_{F_{\upsilon_1}}(\mathcal F_\sigma)&=\dim\mathcal F_{\upsilon_1}+(k+1)+(k+1)(r+1)+a+r+1-a=\dim\mathcal F_{\upsilon_1}+(k+1)(r+2)+r+1\\
&=\dim\mathcal F_{\upsilon_1}+(r+2)(k+2)-1=\left\{\begin{array}{ll}\dim\mathcal F_{\upsilon_1}+2(r+k)+3&{\rm if}\ rk=0;\\
\dim\mathcal F_{\upsilon_1}+2(r+k)+3+kr&{\rm otherwise};\\
\end{array}\right.
\end{array}$$
So we get that $\upsilon\in Sing(\sigma)$ if and only if $kr=0$ (here we used the last condition of admissibility).
\end{proof}


\subsection{Construction of $\upsilon\in Sing(\sigma)$ from $\upsilon((i,j),(i',j'),[s,t])$}\label{gencase}
For $\sigma\in I_{n,k}$ where $n>2k$ let $\widehat \sigma$ be either  $(1,0)$ or $(0,1)$ maximal completion of $\sigma$.\index{$\widehat \sigma$, Iseither  $(1,0)$ or $(0,1)$ maximal completion of $\sigma$}  We show that  $\upsilon\in Sing(\sigma)$ if and only if its corresponding $(1,0)$ or respectively  $(0,1)$ maximal completion $\widehat \upsilon\in Sing(\widehat \sigma).$
We also show that for $\sigma\in I_{n,k}$, $\upsilon\in Sing(\sigma)$ and $\sigma'\in I_{n',k'}$ one has $\sigma'\upsilon_{+n'}\in Sing(\sigma'\sigma_{+n'})$ and $\upsilon\sigma'_{+n}\in Sing(\sigma\sigma'_{+n}).$ In particular, one has $\upsilon_{+1}\in Sing(\sigma_{+1})$ (resp. $\upsilon$ as an element of $I_{n+1,k}$ is in $Sing(\sigma)$ as an element of  $I_{n+1,k}$).

Let $Gr_d(n)$ denote $d$-Grassmanian variety of $\mathbb{C}^n$\index{$Gr_d(n)$, $d$-Grassmanian variety of $\mathbb{C}^n$} that is the variety of all $d$-dimensional subspaces of $\mathbb{C}^n$.  As a straightforward generalization of \cite[\S 6.3]{FM2009} we get

\begin{proposition}\label{prop_compl} Let $\sigma\in I_{n,k}$ with $2k<n$, $x$ with $\lambda(x)=(2^k,1^{n-2k})$, and $\sigma(m,n+1)\in I_{n+1,k+1}$ be its maximal $(0,1)$ completion. Then a map
$$\phi:\mathcal F_{\widehat\sigma}\rightarrow Gr_n(n+1)\ :\ \phi((V_0,\ldots V_{n+1}))= V_n$$
is the subvariety of $\mathcal H_x:=\{ V\in Gr_n(n+1): V\supset Ker\, x\}\cong {\mathbb P}^k$. Moreover, the map $\phi:\mathcal F_{\widehat \sigma}\rightarrow \mathcal H_x$ is a locally trivial fiber bundle with typical fiber  isomorphic to $\mathcal F_\sigma$ so that $\mathcal F_{\widehat\sigma}$ is an iterated bundle of base $(\mathcal F_\sigma, {\mathbb P}^k)$.

In particular one has that $\mathcal F_\upsilon\subseteq \mathcal F_\sigma$ if and only if  $\mathcal F_{\widehat \upsilon}\subseteq\mathcal F_{\widehat \sigma}$ and ${\rm codim}_{\mathcal F_\sigma}\mathcal F_\upsilon={\rm codim}_{\mathcal F_{\widehat \sigma}}\mathcal F_{\widehat \upsilon}$ and $\dim \mathcal T_{F_\upsilon}(\mathcal F_\sigma)-\dim\mathcal F_\sigma=\dim\mathcal T_{F_{\widehat \upsilon}}(\mathcal F_{\widehat \sigma})-\dim\mathcal F_{\widehat \sigma}$, so that the components of a singular locus of $\mathcal F_{\widehat\sigma}$ are obtained by maximal completion from the components of the singular locus of $\mathcal F_\sigma$ and  singular locus of $\mathcal F_{\widehat\sigma}$ is an iterated bundle of base singular locus of $\mathcal F_\sigma$ and ${\mathbb P}^k$.
\end{proposition}
\begin{proof}
Let $\sigma\in I_{n,k}$ with $2k<n$ be a link pattern and $\widehat{\sigma}\in I_{n+1,k+1}$ its maximal $(0,1)$ completion. Let $\mathcal{H}_x$ be the variety of hyperplanes
$H\subset V=\mathbb C^{n+1}$ such that $\ker x\subset H$.\index{$\mathcal{H}_x$, Variety of hyperplanes $H\subset V=\mathbb C^{n+1}$ such that $\ker x\subset H$} This is a projective variety which is naturally isomorphic to the variety of hyperplanes of the space
$V/\ker x$, hence $\mathcal{H}_x\cong\mathbb{P}^k$.

Recall that $\mathcal F_{\widehat\sigma}=\overline{\mathcal Z}_{\widehat\sigma}$.
Every flag $F=(V_0,\ldots,V_{n+1})\subset \mathcal Z_{\widehat{\sigma}}$ satisfies $\ker x\subset V_n$. Hence, $V_n\in\mathcal{H}$ whenever
$F\in \mathcal Z_{\widehat{\sigma}}$. Thus,
 the fact that the map
\begin{align*}
\Phi:\mathcal F_{\widehat{\sigma}}&\rightarrow \mathcal{H}_x,\\
(V_0,\ldots,V_{n+1})&\mapsto V_n
\end{align*}
is indeed locally trivial fiber bundle of fiber isomorphic to $\mathcal F_\sigma$ follows immediately from the fact that it is locally closed subset of $\mathcal F_{\sigma}$ containing $\mathcal Z_\sigma.$ (or cf. the proof of \cite[Theorem 2.1]{FM2010}.)
\end{proof}

By Proposition \ref{prop_compl} if
$\pi_{s,t}(\sigma)$ is singular and $\sigma$ is the maximal
$(s-1,n-t)$-completion of $\pi_{s,t}(\sigma)$ then $\sigma$ is
singular. However, it is not true that if
$\mathcal{F}_{\pi_{s,t}(\sigma)}$ is singular then
$\mathcal{F}_\sigma$ must be singular, as it is shown by the small example,
not connected to admissible pairs. Consider
$\sigma=(1,7)(2,6)(4,5)\in I_{7,3}$ and its projection
$\pi_{1,6}(\sigma)=(2,6)(4,5)\in I_{6,2}$.
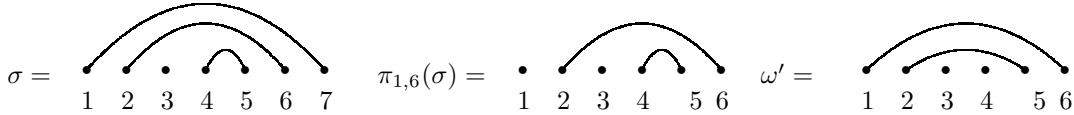
\begin{center}
\begin{picture}(230,60)(50,15)
\put(-40,35){$\sigma=$} \multiput(-10,40)(15,0){7} {\circle*{3}}
\put(-12,25){$1$} \put(3,25){$2$} \put(18,25){$3$}
\put(33,25){$4$} \put(48,25){$5$} \put(63,25){$6$}
\put(78,25){$7$} \qbezier(-10,40)(35,90)(80,40)
\qbezier(5,40)(35,75)(65,40) \qbezier(35,40)(42.5,55)(50,40)
\put(100,35){$\pi_{1,6}(\sigma)=$} \multiput(155,40)(15,0){6}
{\circle*{3}} \put(153,25){$1$} \put(168,25){$2$}
\put(183,25){$3$} \put(198,25){$4$} \put(218,25){$5$}
\put(228,25){$6$} \qbezier(170,40)(200,75)(230,40)
\qbezier(200,40)(207.5,55)(215,40)
\put(130,0){\put(115,35){$\omega'=$} \multiput(155,40)(15,0){6}
{\circle*{3}} \put(153,25){$1$} \put(168,25){$2$}
\put(183,25){$3$} \put(198,25){$4$} \put(218,25){$5$}
\put(228,25){$6$} \qbezier(155,40)(192.5,75)(230,40)
\qbezier(170,40)(192.5,55)(215,40)}
\end{picture}
\end{center}

One can easily check that  $\mathcal F_{(2,6)(4,5)}$ is singular,
but $\mathcal F_{\sigma}$ is smooth. Indeed,
$\mathcal{Z}_{\omega'}$ is of codimension $3$ in
$\mathcal{F}_{\pi_{1,6}(\sigma)}$ and has $4$ predecessors
in $G_{\pi_{1,6}(\sigma)}$. On the other hand, by \cite{Fr2009}, if $\mathcal{F}_\sigma$ is
singular then $\mathcal Z_{\omega_0}$ must be singular in
it (where $\omega_0=(1,5)(2,6)(3,7)$ is from Theorem \ref{thM007}).  However, the straightforward computation shows that there are 4 predecessors of $\omega_0$ in $G_\sigma$ and that ${\rm codim}_{\mathcal F_\sigma}\mathcal Z_{\omega_0}=4,$ thus $\mathcal F_\sigma$ is smooth.

For $\sigma\in I_{n,k}^{\max}$ one can use projections and
concatenations in order to find elements of $Sing(\sigma):$ their
number is greater or equal to the number of admissible
pairs (at different intervals) of $\sigma$ obtained by repeating
the procedures of adding fixed points and of maximal completions
may be a few times. Note that concatenation can be represented as
series of additions of fixed points and maximal completions.

Let $\mathcal B_n=\{e_i\}_{i=1}^n$ be the standard basis of
$\mathbb C^n$,\index{$\mathcal B_n=\{e_i\}_{i=1}^n$, Standard basis of $\mathbb C^n$,} $\mathcal B'_{n'}=\{e'_j\}_{j=1}^{n'}$ be the
standard basis of $\mathbb C^{n'}$ and $\mathcal B_n\sqcup
\mathcal B'_{n'}$ the standard basis of $\mathbb C^{n+n'}$.\index{$\mathcal B_n\sqcup
\mathcal B'_{n'}$, Standard basis of $\mathbb C^{n+n'}$} For
$x\in End (\mathbb C^n), y\in End(\mathbb C^{n'})$ of nilpotency
order 2 of ranks $k,k'$ respectively let $x*y\in End(\mathbb
C^{n+n'})$ be defined by $x*y(e_i)=x(e_i)$\index{$x*y(e_i)$} for any $1\leq i\leq n$
and $x*y(e'_j)=y(e'_j)$ for any $1\leq j\leq n'$. Obviously, $x*y$
is of nilpotency order 2,  ${\rm Rank}\, x*y= k+k'$ and for any
$\sigma\in I_{n,k},\sigma'\in I_{n',k'}$  one has $\mathcal
Z_{\sigma\sigma'_{+n}}$ is $Z_{x*y}$ orbit. We get

\begin{proposition}\label{prop3.17}
Let $\sigma,\upsilon\in I_{n,k}$ be  link patterns such that $\sigma>\upsilon$ and let $\sigma'\in I_{n',k'}.$  Then  $\upsilon\sigma'_{+n}\in Sing(\sigma\sigma'_{+n})$ iff
$\upsilon\in Sing(\sigma).$

In particular, if $\widetilde \upsilon,\widetilde \sigma$ are  $\upsilon,\sigma$ considered as elements of $I_{n+1,k}$ one has $\widetilde \upsilon\in Sing(\widetilde \sigma)$ iff
$\upsilon\in Sing(\sigma).$
\end{proposition}
\begin{proof}
Indeed, note that ${\rm codim}_{\mathcal F_{\sigma\sigma'_{+n}}}\mathcal F_{\upsilon\sigma'_{+n}}=c(\upsilon)+b(\upsilon)-(c(\sigma)+b(\sigma))={\rm codim}_{\mathcal F_\sigma}\mathcal F_\upsilon$.

One also has that $\sigma\sigma'_{+n}\geq\upsilon'>\upsilon\sigma'_{+n}$  iff $\upsilon'=\omega\sigma'_{+n}$ where $\sigma\geq\omega>\upsilon$.  Indeed, if
 $\upsilon'=\omega\sigma'_{+n}$ where $\sigma\geq\omega>\upsilon$ then obviously, $\sigma\sigma'_{+n}\geq\upsilon'>\upsilon\sigma'_{+n}$. On the other hand, if   $\sigma\sigma'_{+n}\geq\upsilon'>\upsilon\sigma'_{+n}$, then in particular,
$\sigma'_{+n}=\pi_{n+1,n+n'}(\sigma\sigma'_{+n})
\leq \pi_{n+1,n+n'}(\upsilon')\leq \pi_{n+1,n+n'}(\upsilon\sigma'_{+n})=\sigma'_{+n}.$
Thus, $\pi_{n+1,n+n'}(\upsilon')=\sigma'_{+n}$.
One also has
$\pi_{1,n}(\sigma\sigma'_{+n})=\sigma\geq\pi_{1,n}(\upsilon')\geq \pi_{1,n}(\upsilon\sigma'_{+n})=\upsilon.$ Taking into account that the number of arcs in all link patterns are
$k+k'$ we get  that $\upsilon'=\omega\sigma'_{+n}$ where $\omega=\pi_{1,n}(\upsilon')$ satisfies $\sigma\geq \omega>\upsilon.$

In particular $\upsilon'\leq\sigma\sigma'_{+n}$ is a predecessor of $\upsilon\sigma'_{+n}$ iff $\omega\leq \sigma$ is a predecessor of $\upsilon$.
Thus,  $$\dim\mathcal T_{F_{\upsilon\sigma'_{+n}}}(\mathcal F_{\sigma\sigma'_{+n}})-\dim\mathcal F_{\sigma\sigma'_{+n}}=\dim\mathcal T_{F_\upsilon}(\mathcal F_\sigma)-\dim\mathcal F_\sigma.$$
\end{proof}

\subsection{The components of the singular locus of a $\mathcal F_\sigma$ for $\sigma\in I_{n,k}^{\max}$}\label{sec2.4}
The algorithm for elements of $Sing(\sigma)$ for $\sigma\in
I^{\max}_{n,k}$ with $\rho(\sigma)\geq 4$ is obtained as follows:
\begin{itemize}
\item[a)] completion -- determining $\widehat\sigma$:
\begin{itemize}
\item If $1,n\in \sigma^0$ put $\widehat\sigma=\sigma$ and
$s=1,t=n;$
\item If $|\{1,n\}\cap\sigma^0|=1$ then if
$1\in\sigma^0$ put $s=1$, $t=\max \tau^*(\sigma)$ and
$\widehat\sigma=\pi_{1,t}(\sigma)$, otherwise put $t=n$,
$s=\min\tau^*(\sigma)+1$ and
$\widehat\sigma=(\pi_{s,n}(\sigma))_{+(1-s)}$; \item  If
$\{1,n\}\cap\sigma^0=\emptyset$ and $(1,n)\in\sigma$ put $s=1$,
$t=\max \tau^*(\sigma)$ and $\widehat\sigma=\pi_{1,t}(\sigma)$;
otherwise put $s=\min\tau^*(\sigma)+1$, $t=\max\tau^*(\sigma)$ and
$\widehat\sigma=(\pi_{s,t}(\sigma))_{+(1-s)}$.
\end{itemize}
In step (b) we  find $Sing(\widehat\sigma)$ and then get
$Sing(\sigma)$ from it by maximal $(s-1,n-t)$-completion.

\item[b)] For any $\{i,j\}\in\tau^*(\widehat\sigma)$ let
$(a,b)\in\widehat\sigma$ be maximal arc over $(i,i+1)$ such that
$b<j$ and $\tau^*(\widehat\sigma)\cap [a,i-1]=\emptyset$. Let
$(c,d)$ be maximal arc over $(j,j+1)$ such that $i<c$ and
$\tau^*(\widehat\sigma)\cap[j+1,d]=\emptyset$. Note that these
arcs are well defined and $b<c$ since $\widehat\sigma$ is maximal.
Now for $(a,b),(c,d)$ find all possible intervals, where they are an admissible pair. To do this, put
$m_1=\max\{f\in{\widehat\sigma}^0\ :\ f<a\}$ and
$m_2=\min\{f\in{\widehat\sigma}^0\ :\ f>d\}$ and put
$\upsilon_1=\pi_{1,m_1-1}(\widehat\sigma)$,
$\upsilon_2=\pi_{m_2+1,n}(\widehat\sigma)$and
$\upsilon=\pi_{m_1,m_2}(\widehat\sigma)$.

\begin{itemize}
\item If
$\tau^*(\widehat\sigma)\cap([m_1,a]\cup[d,m_2])=\emptyset$ then
$[m_1,m_2]$ is the only possible interval for $(a,b),(c,d)$ to be
admissible at and
$\omega=\upsilon_1(\upsilon_{(a,b)(c,d)}^-(a,d)(m_1,m_2))\upsilon_2$.
\item Otherwise,\\
 If there is no $(s,t)\in\upsilon$ such that $s<a<b<t<j$ then for $l\in\tau^*(\widehat\sigma)\cap([m_1,a])$ let $(r_l,q_l)$ be maximal arc of $\upsilon$ over $(l,l+1)$ such that $q_l<a$ and let $I=\{m_1\}\cup\{q_l\}_{l\in\tau^*(\widehat\sigma)\cap([m_1,a])}$.\\
If there exists $(s,t)\in\upsilon$ such that $s<a<b<t<j$ then let
$(s,t)$ be minimal such arc. Note that
$\tau^*(\widehat\sigma)\cap[s,a] \ne\emptyset$ and for
$l\in\tau^*(\widehat\sigma)\cap[s,a]$ let $(r_l,q_l)$ be maximal
arc of $\upsilon$ over $(l,l+1)$ such that $q_l<a$ and let
$I=\{q_l\}_{l\in\tau^*(\widehat\sigma)\cap[s,a]}$.

Exactly in the same way,\\
If there is no $(s,t)\in\upsilon$ such that $i<s<c<d<t$ then for $l\in\tau^*(\widehat\sigma)\cap([d,m_2])$ let $(r_l,q_l)$ be maximal arc of $\upsilon$ over $(l,l+1)$ such that $d<r_l$ and let $J=\{m_2\}\cup\{r_l\}_{l\in\tau^*(\widehat\sigma)\cap([d,m_2])}$.\\
If there exists $(s,t)\in\upsilon$ such that $i<s<c<d<t$ then let
$(s,t)$ be minimal such arc. Note that
$\tau^*(\widehat\sigma)\cap[d,t] \ne\emptyset$ and for
$l\in\tau^*(\widehat\sigma)\cap[d,t]$ let $(r_l,q_l)$ be maximal
arc of $\upsilon$ over $(l,l+1)$ such that $r_l>d$ and let
$J=\{r_l\}_{l\in\tau^*(\widehat\sigma)\cap[d,t]}$.

For any interval $[p,q]\ : r\in I,\ q\in J$ the pair $(a,b),(c,d)$ is admissible at $[r,q]$ and
$\omega([r,q])=\upsilon_1\widehat\omega\upsilon_2$ where \\
if $r>m_1$ and $q<m_2$ then $\widehat\omega$ is  $(r-2-m_1,m_2-q-1)$-completion of $$\pi_{r,q}(\upsilon)_{(a,b),(c,d)}^-(a,d)(r,q).$$
This completion belongs to $[m_1+1,m_2-1]$ so that $m_1,m_2$ are fixed points;\\
if $r=m_1$ and $q<m_2$ then $\widehat\omega$ is $(0,m_2-q-1)$-completion of\\ $\pi_{m_1,q}(\upsilon)_{(a,b),(c,d)}^-(a,d)(m_1,q)$ with one added fixed point on the right;\\
if $r>m_1$ and $q=m_2$ then $\widehat\omega$ is $(r-2-m_1,0)$-completion of\\ $\pi_{r,m_2}(\upsilon)_{(a,b),(c,d)}^-(a,d)(r,m_2)$ with one added fixed point on the left;\\
Finally, if $(r,q)=(m_1,m_2)$ then $\widehat\omega=\upsilon_{(a,b)(c,d)}^-(a,d)(m_1,m_2).$
\end{itemize}
\end{itemize}

\begin{example}
Consider the following example, demonstrating different cases of our algorithm. Let $\sigma=(2,9)(3,6)(4,5)(7,8)(10,11)\in I_{12,5}$
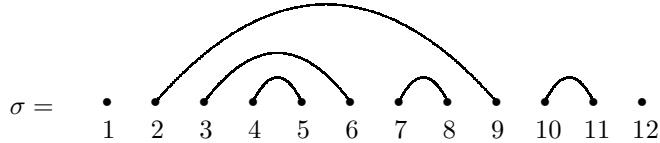
\begin{center}
\begin{picture}(200,60)(20,-10)
\setlength{\unitlength}{.92pt}
\put(-40,-5){$\sigma=$}\multiput(0,0)(20,0){12} {\circle*{3}}
\put(-2,-15){$1$}\put(18,-15){$2$}\put(38,-15){$3$}
\put(58,-15){$4$}\put(78,-15){$5$}\put(98,-15){$6$}
\put(118,-15){$7$}\put(138,-15){$8$}\put(158,-15){$9$}
\put(176,-15){$10$}\put(196,-15){$11$}\put(216,-15){$12$}
\qbezier(20,0)(90,80)(160,0)\qbezier(40,0)(70,40)(100,0)
\qbezier(60,0)(70,20)(80,0)\qbezier(120,0)(130,20)(140,0)
\qbezier(180,0)(190,20)(200,0)
\end{picture}
\end{center}
$\tau^*(\sigma)=\{4,7,10\}$. For $\{4,7\}$ the admissible pair is
$(3,6),(7,8)$ and the intervals are $[1,10]$ and $[1,12]$ with
respective $\omega_1,\omega_2\in Sing(\sigma):$
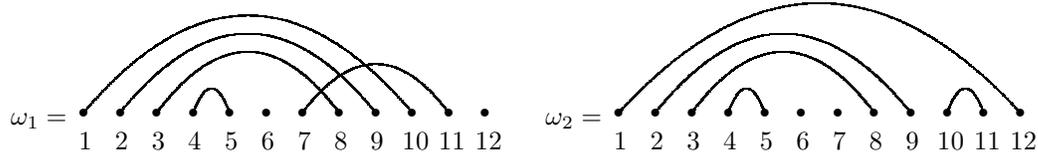
\begin{center}
\begin{picture}(200,60)(75,-10)
\setlength{\unitlength}{.92pt}
\put(-30,-5){$\omega_1=$}
\multiput(0,0)(15,0){12} {\circle*{3}}
\put(-2,-15){$1$}\put(13,-15){$2$}\put(28,-15){$3$}\put(43,-15){$4$}
\put(58,-15){$5$}\put(73,-15){$6$}\put(88,-15){$7$}\put(103,-15){$8$}
\put(118,-15){$9$}\put(131,-15){$10$}\put(146,-15){$11$}\put(161,-15){$12$}
\qbezier(0,0)(67.5,80)(135,0)\qbezier(15,0)(67.5,65)(120,0)\qbezier(30,0)(67.5,50)(105,0)
\qbezier(45,0)(52.5,20)(60,0)\qbezier(90,0)(120,40)(150,0)
\put(190,-5){$\omega_2=$}
\multiput(220,0)(15,0){12} {\circle*{3}}
\put(218,-15){$1$}\put(233,-15){$2$}\put(248,-15){$3$}\put(263,-15){$4$}
\put(278,-15){$5$}\put(293,-15){$6$}\put(308,-15){$7$}\put(323,-15){$8$}
\put(338,-15){$9$}\put(351,-15){$10$}\put(366,-15){$11$}\put(381,-15){$12$}
\qbezier(220,0)(302.5,90)(385,0)\qbezier(235,0)(287.5,65)(340,0)\qbezier(250,0)(287.5,50)(325,0)
\qbezier(265,0)(272.5,20)(280,0)\qbezier(355,0)(362.5,20)(370,0)
\end{picture}
\end{center}
For $\{4,10\}$ the admissible pair is $(2,9),(10,11)$ and the only interval is $[1,12]$. Respectively,
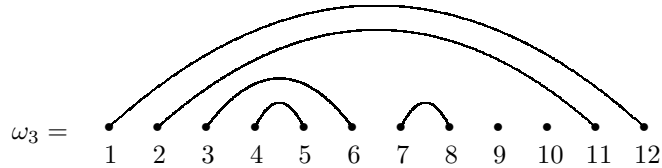
\begin{center}
\begin{picture}(200,70)(20,-10)
\setlength{\unitlength}{.92pt}
\put(-40,-5){$\omega_3=$}
\multiput(0,0)(20,0){12} {\circle*{3}}
\put(-2,-15){$1$}\put(18,-15){$2$}\put(38,-15){$3$}
\put(58,-15){$4$}\put(78,-15){$5$}\put(98,-15){$6$}\put(118,-15){$7$}\put(138,-15){$8$}
\put(158,-15){$9$}\put(176,-15){$10$}\put(196,-15){$11$}\put(216,-15){$12$}
\qbezier(20,0)(110,80)(200,0)\qbezier(40,0)(70,40)(100,0)\qbezier(60,0)(70,20)(80,0)
\qbezier(120,0)(130,20)(140,0)\qbezier(0,0)(110,100)(220,0)
\end{picture}
\end{center}
Finally, for $\{7,10\}$ the admissible pair is $(7,8),(10,11)$ and the only interval $[6,12]$ so that
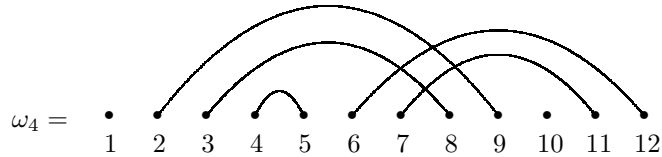
\begin{center}
\begin{picture}(200,70)(20,-10)
\setlength{\unitlength}{.92pt}
\put(-40,-5){$\omega_4=$}\multiput(0,0)(20,0){12} {\circle*{3}}
\put(-2,-15){$1$}\put(18,-15){$2$}\put(38,-15){$3$}\put(58,-15){$4$}
\put(78,-15){$5$}\put(98,-15){$6$}\put(118,-15){$7$}\put(138,-15){$8$}
\put(158,-15){$9$}\put(176,-15){$10$}\put(196,-15){$11$}\put(216,-15){$12$}
\qbezier(20,0)(90,90)(160,0)\qbezier(40,0)(90,60)(140,0)
\qbezier(60,0)(70,20)(80,0) \qbezier(120,0)(160,50)(200,0)\qbezier(100,0)(160,70)(220,0)
\end{picture}
\end{center}
So, according to the algorithm there are 4 elements in $Sing(\sigma).$
\end{example}

We are going to show in this subsection that for $\sigma\in I_{n,k}$ with $\rho(\sigma)\geq 4$,  $Sing(\sigma)$ is obtained by this algorithm.

\begin{theorem}\label{components}
Let $\sigma\in I_{n,k}^{\max}$ be such that $\rho(\sigma)\geq 4$ and let $\omega\in Sing(\sigma)$ then $\omega$ is obtained by our algorithm.
\end{theorem}

To prove the theorem we need a few technical lemmas. Then we prove it in Propositions \ref{prop3.24}, \ref{prop3.25}.

Recall a notion of a predecessor of  $\sigma\in I_{n,k}$ from \S2.2.

\begin{lemma} Given $\sigma\in I_{n,k}^{\max}$ with $b(\sigma)+c(\sigma)=k$ then the number of predecessors of $\sigma$ is $2k$.
\end{lemma}
\begin{proof}
Indeed for each fixed point $f$ and $(i,j)\in\sigma$ a bridge over it one has $\sigma'=(i,f)\sigma^-_{(i,j)}$ and $\sigma''=(f,j)\sigma^-_{(i,j)}$ are predecessors connected to this point and this bridge.

For $(i,j)(s,t)\in\sigma$ such that $i<s<j<t$ one has $$\sigma'=(i,s)(t,j)\sigma_{(i,j),(s,t)}^-\mbox{ and }\sigma''=(i,t)(s,j)\sigma_{(i,j),(s,t)}^-$$ are predecessors connected to these two crossing arcs.

Thus we get that each crossing and each bridge over a fixed point provides us exactly 2 predecessors.
\end{proof}

Recall that for $\sigma,\omega\in I_{n,k}$ such that $\sigma>\omega$ and for their $(0,1)$-maximal completions $\sigma(a,n+1),\ \omega(b,n+1)$ one has $\sigma(a,n+1)>\omega(b,n+1)$ by Proposition \ref{prop_compl}.
Note also that the elementary moves on $\sigma$ and its maximal completion correspond, namely

\begin{lemma}\label{lemmatex2} Let $\sigma\in I_{n,k}$ and let $f=\max \sigma^0$. Let $\widehat\sigma=\sigma(f,n+1)$ be its (0,1)-maximal completion. Then for any predecessor
$\sigma'$ of $\sigma$ there exists $\widehat\sigma'$ predecessor of $\widehat\sigma$ obtained by completion (not necessarily maximal).
\end{lemma}
\begin{proof}
If $\sigma'$ is a predecessor of $\sigma$ obtained by any elementary move forward, not connected to $f$ then respectively $\sigma'(f,n+1)$ is the corresponding predecessor
of $\widehat\sigma$ (note that it must not be the maximal completion of $\sigma'$).

If there exist $(a,b)\in\sigma$ such that $a<f<b$ then for
predecessor $\sigma'=\sigma_{(a,b)}^-(a,f)$ one has $\widehat
\sigma'=\sigma_{(a,b)}^-(a,f)(b,n+1)$ is the corresponding
predecessor of $\widehat\sigma$ and for predecessor
$\sigma''=\sigma_{(a,b)}^-(f,b)$ one has
$\widehat\sigma''=\sigma_{(a,b)}^-(a,n+1)(f,b)$ is the
corresponding predecessor of $\widehat \sigma$ (again in this case
$\widehat\sigma''$ is not necessarily maximal completion of
$\sigma''$).
\end{proof}

As a corollary we get:
\begin{proposition}\label{prop3.20} Given  $\sigma\in I_{n,k}^{\max}$ and $\omega\in G_\sigma\setminus\{\sigma\}$  then for every pair $\{\omega',\omega''\}$ of predecessors of $\omega$ at least one is in $G_\sigma$.
\end{proposition}
\begin{proof}
This claim is true only for $\sigma\in I_{n,k}^{\max}$, so its
proof is rather long and technical. We simply make all possible cases one by one. For $\sigma\in I_{n,k}^{\max}$, $\omega\in
G_\sigma$ and its predecessor $\omega'$ satisfying
$\omega'\leq\sigma$ one has by Proposition \ref{prop_compl} and Lemma \ref{lemmatex2}
that for (0,1)-maximal completions
$\widehat\sigma,\widehat\omega$ there exists completion $\widehat\omega'$ satisfying $\widehat\omega'$ is a predecessor of $\widehat\omega$ and $\widehat\omega'\leq\widehat\sigma$. Thus, using these two claims $n-2k$ times it is enough to
 show the claim only for $\sigma,\omega\in I_{2k,k}$, that is $k-$complete link patterns  where all the predecessors are obtained by ``uncrossing'' a given cross. We will show the claim by induction on the number of crosses in $\omega.$

If $c(\omega)=1$ then obviously $\sigma$ is one of the only two predecessors of $\omega$ so the claim is trivially true.

Now assume that this is true for $\omega'$ with at most $j-1$ crosses and show for $\omega$ with $j$ crosses.

By the construction of $G_\sigma$ there exists $\upsilon\in G_\sigma$ such that $\omega=(a,c)(b,d)\phi$ where $a<b<c<d$ and
$\upsilon=(a,b)(c,d)\phi$ or $\upsilon=(a,d)(b,c)\phi$ and in both cases $cr(\upsilon)=j-1$. Before we proceed, note that since $codim_{\mathcal{F}_\upsilon}\mathcal{Z}_\omega=1$:
\begin{itemize}
\item[(1)] if $\upsilon=(a,b)(c,d)\phi$ then there are no
$(i,j)\in\phi$ such that $i<a,b<j<c$ or $b<i<c,d<j$;
\item[(2)] if $\upsilon=(a,d)(b,c)\phi$ then there are no $(i,j)\in\phi$ such that
$a<i<b$ and $c<j<d$.
\end{itemize}

There are 3 categories of predecessors of $\omega$:
\begin{itemize}
\item[(i)] connected to $(p,r)(q,s)\in\phi$ such that $p<q<r<s$;
\item[(ii)] connected either to $(a,c)$ and $(i,j)\in\phi$ crossing $(a,c)$ that is such that either $a<i<c<j$ or $i<a<j<c$; or respectively connected to $(b,d)$ and $(i,j)\in\phi$ crossing $(b,d)$ that is
such that either $i<b<j<d$ or $b<i<d<j$.
\item[(iii)] $(a,b)(c,d)\phi$ and $(a,d)(b,c)\phi$.
\end{itemize}
By our construction at least one of predecessors in (iii) satisfies the claim. Let us study the predecessors of types (i) and (ii)

(i) For any $(p,r)(q,s)\in\phi$ such that $p<q<r<s$ and for $\upsilon=(a,b)(c,d)\phi$ or $(a,d)(b,c)\phi$ one has by induction assumption that either
predecessor $\upsilon'=\upsilon^-_{(p,r)(q,s)}(p,q)(r,s)<\sigma$ or
 or  predecessor $\upsilon'=\upsilon^-_{(p,r)(q,s)}(p,s)(q,r)<\sigma$. In both cases
$$\omega'=\left\{\begin{array}{ll}(\upsilon')^-_{(a,b)(c,d)}(a,c)(b,d)&{\rm if}\ \upsilon=(a,b)(c,d)\phi;\\
(\upsilon')^-_{(a,d)(b,c)}(a,c)(b,d)&{\rm if}\ \upsilon=(a,d)(b,c)\phi;\\
\end{array}\right.$$
is a predecessor of $\omega$ satisfying $\omega'<\upsilon'$ so that $\omega'<\sigma$ and we are done.

(ii) Obviously, case $(i,j)$ intersecting $(a,c)$ and case $(i,j)$ intersecting $(b,d)$ are obtained one from another by Sch\"utzenberger transformation of $\sigma,\omega$, so it is enough to consider only the case $(i,j)$ intersecting $(a,c).$  Here we consider two types of $\upsilon$ separately.

First  assume $\upsilon=(a,b)(c,d)\phi$.
\begin{itemize}
\item{}
Let $(i,j)\in\phi$ be such that $i<a<j<c$. The corresponding predecessors of $\omega$ are $\omega'=(i,a)(j,c)(b,d)\phi_{(i,j)}^-$ and
$\omega''=(i,c)(a,j)(b,d)\phi_{(i,j)}^-$. By (1)  one has $i<a<j<b$ so that $\upsilon$ has two predecessors $\upsilon'=(i,a)(j,b)(c,d)\phi_{(i,j)}^-$
and $\upsilon''=(i,b)(a,j)(c,d)\phi_{(i,j)}^-$ connected to crossing arcs $(i,j),(a,b)$. By induction hypothesis either $\upsilon'\leq\sigma$ or
$\upsilon''\leq\sigma.$ Note that $\upsilon'$ is a predecessor of $\omega'$ and
$\upsilon''$ is a predecessor of $\omega''$ so that at least one of $\omega',\omega''$ is smaller than $\sigma.$

\item{} Let $(i,j)\in\phi$ be such that $a<i<c<j$. The corresponding predecessors of $\omega$ are
$\omega'=(a,i)(c,j)(b,d)\phi_{(i,j)}^-$ and
$\omega''=(a,j)(i,c)(b,d)\phi_{(i,j)}^-$. By (1) one has either $b<i<c<j<d$ or $i<b$.

If $b<i<c<j<d$ then
\begin{align*}
\upsilon'&=(a,b)(i,d)(c,j)\phi_{(i,j)}^-,\\
\upsilon''&=(a,b)(i,c)(j,d)\phi_{(i,j)}^-
\end{align*}
and are predecessors of $\upsilon$ connected to arcs $(i,j)$ and $(c,d)$. By induction hypothesis at least one of them is in $G_\sigma$. Further note that $\upsilon'$ is a predecessor of $\omega'$ and
$\upsilon''$ is a predecessor of $\omega''$ so that at least one of $\omega',\omega''$ is smaller than $\sigma.$

If $i<b$ then $\upsilon'=(a,i)(b,j)(c,d)\phi_{(i,j)}^-$ and
$\upsilon''=(a,j)(i,b)(c,d)\phi^-_{(i,j)}$ are the predecessors of
$\upsilon$ connected to arcs $(a,b),(i,j)$ and by induction hypothesis
at least one of them is less or equal to $\sigma.$
\begin{itemize}
\item{} If  in addition $j>d$ then  $\upsilon'$ is a predecessor of $\omega'$ and $\upsilon''$ is a predecessor of $\omega''$ so that at least one of $\omega',\omega''$ is less than $\sigma$.
\item{} The case $j<d$  is more subtle and we  draw the corresponding link patterns in order to see the picture. Since all the arcs but connected to $a,b,c,d,i,j$ are the same in all our link patterns we can ignore them and draw only those connected to these 6 points. One has:
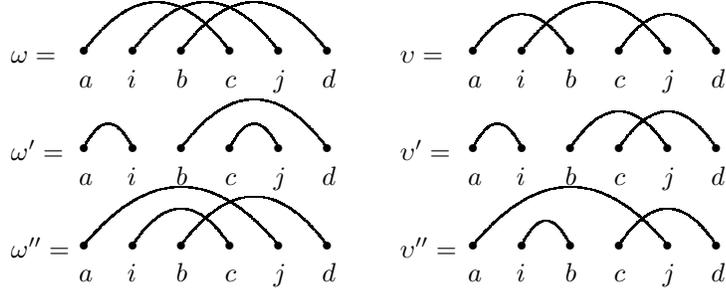
\begin{center}
\begin{picture}(130,140)(80,-60)
\setlength{\unitlength}{.92pt}
\put(-40,35){$\omega=$}\multiput(-10,40)(20,0){6} {\circle*{3}}
\put(-12,25){$a$}\put(8,25){$i$}\put(28,25){$b$}\put(48,25){$c$}
\put(68,25){$j$}\put(88,25){$d$}\qbezier(-10,40)(20,80)(50,40)\qbezier(10,40)(40,80)(70,40)
\qbezier(30,40)(60,80)(90,40)
\put(120,35){$\upsilon=$}\multiput(150,40)(20,0){6} {\circle*{3}}
\put(148,25){$a$}\put(168,25){$i$}\put(188,25){$b$}\put(208,25){$c$}
\put(228,25){$j$}\put(248,25){$d$}\qbezier(150,40)(170,70)(190,40)\qbezier(170,40)(200,80)(230,40)
\qbezier(210,40)(230,70)(250,40)
\put(-40,-5){$\omega'=$}\multiput(-10,0)(20,0){6} {\circle*{3}}
\put(-12,-15){$a$}\put(8,-15){$i$}\put(28,-15){$b$}\put(48,-15){$c$}
\put(68,-15){$j$}\put(88,-15){$d$}\qbezier(-10,0)(0,20)(10,0)\qbezier(30,0)(60,40)(90,0)
\qbezier(50,0)(60,20)(70,0)
\put(120,-5){$\upsilon'=$}\multiput(150,0)(20,0){6} {\circle*{3}}
\put(148,-15){$a$}
\put(168,-15){$i$}\put(188,-15){$b$}\put(208,-15){$c$}\put(228,-15){$j$}\put(248,-15){$d$}
\qbezier(150,0)(160,20)(170,0)\qbezier(190,0)(210,30)(230,0)\qbezier(210,0)(230,30)(250,0)
\put(-40,-45){$\omega''=$}
\multiput(-10,-40)(20,0){6} {\circle*{3}}
\put(-12,-55){$a$}
\put(8,-55){$i$}\put(28,-55){$b$}\put(48,-55){$c$}\put(68,-55){$j$}\put(88,-55){$d$}
\qbezier(-10,-40)(30,8)(70,-40)\qbezier(10,-40)(30,-10)(50,-40)\qbezier(30,-40)(60,0)(90,-40)
\put(120,-45){$\upsilon''=$}\multiput(150,-40)(20,0){6} {\circle*{3}}
\put(148,-55){$a$} \put(168,-55){$i$} \put(188,-55){$b$} \put(208,-55){$c$}\put(228,-55){$j$}\put(248,-55){$d$}
\qbezier(150,-40)(190,8)(230,-40)\qbezier(170,-40)(180,-20)(190,-40)
\qbezier(210,-40)(230,-10)(250,-40)
\end{picture}
\end{center}
If $\upsilon''<\sigma$ then since $\upsilon''$ is a predecessor of
$\omega''$ we are done. If $\upsilon'<\sigma$ then consider its
predecessors connected to arcs $(b,j),(c,d)$. They are
$\upsilon'_1=(a,i)(b,j)(c,d)\phi^-_{(i,j)}$ and
$\upsilon'_2=(a,i)(b,c)(j,d)\phi^-_{(i,j)}.$ By induction hypothesis
applied to $\upsilon'$ either $\upsilon'_1\leq \sigma$ or $\upsilon'_2\leq
\sigma$. Note that $\upsilon'_1=\omega'$ thus in this case we are
done.
 If $\upsilon'_2\leq\sigma$ note that $\upsilon'_2$ is a predecessor of $(a,j)(i,d)(b,c)\phi^-_{(i,j)}$ which is in turn is a predecessor of
$\omega''$ so that in this case again $\omega''<\sigma$.
Thus, in this case also at least one of $\omega',\omega''$ less than $\sigma.$
\end{itemize}
\end{itemize}

Consider case $\upsilon=(a,d)(b,c)\phi$.
\begin{itemize}
\item{}
Let $(i,j)\in\phi$ be such that $i<a<j<c$. The corresponding predecessors of $\omega$ are $\omega'=(i,a)(j,c)(b,d)\phi_{(i,j)}^-$ and
$\omega''=(i,c)(a,j)(b,d)\phi_{(i,j)}^-$.
\begin{itemize}
\item{} If in addition $j<b$ then $(i,j)$ crosses $(a,d)$ in $\upsilon$ and does not cross $(b,c)$. The predecessors of $\upsilon$ are
$\upsilon'=(i,a)(j,d)(b,c)\phi^-_{(i,j)}$ and $\upsilon''=(i,d)(a,j)(b,c)\phi^-_{(i,j)}$. By induction hypothesis at least one of them is less or equal to $\sigma$ and as one can see immediately $\upsilon'$ is a predecessor of $\omega'$ and $\upsilon''$ is a predecessor of $\omega''$ so that at least one of $\omega', \omega''$ less than $\sigma.$
\item{} The case $j>b$  is more subtle and and we  draw the corresponding link patterns in order to see the picture. We again draw only arcs connected to
$a,b,c,d,i,j$.  One has:
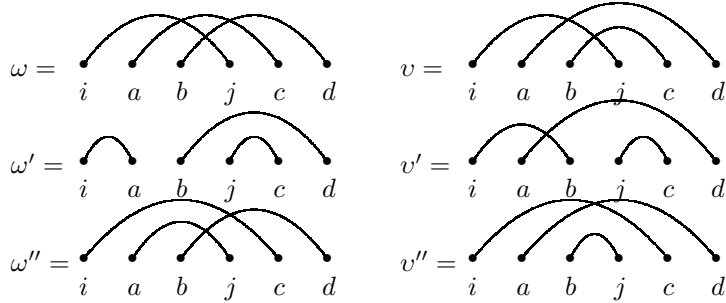
\begin{center}
\begin{picture}(130,140)(80,-60)
\setlength{\unitlength}{.92pt}
\put(-40,35){$\omega=$}
\multiput(-10,40)(20,0){6} {\circle*{3}} 
\put(-12,25){$i$}\put(8,25){$a$}\put(28,25){$b$}\put(48,25){$j$}
\put(68,25){$c$}\put(88,25){$d$}
\qbezier(-10,40)(20,80)(50,40)\qbezier(10,40)(40,80)(70,40)\qbezier(30,40)(60,80)(90,40)
\put(120,35){$\upsilon=$}
\multiput(150,40)(20,0){6} {\circle*{3}} 
\put(148,25){$i$}
\put(168,25){$a$}\put(188,25){$b$}\put(208,25){$j$}\put(228,25){$c$}\put(248,25){$d$}
\qbezier(150,40)(180,80)(210,40)\qbezier(170,40)(210,90)(250,40)\qbezier(190,40)(210,70)(230,40)
\put(-40,-5){$\omega'=$}
\multiput(-10,0)(20,0){6} {\circle*{3}} 
\put(-12,-15){$i$}
\put(8,-15){$a$}\put(28,-15){$b$}\put(48,-15){$j$}\put(68,-15){$c$}\put(88,-15){$d$}
\qbezier(-10,0)(0,20)(10,0)\qbezier(30,0)(60,40)(90,0)\qbezier(50,0)(60,20)(70,0)
\put(120,-5){$\upsilon'=$}
\multiput(150,0)(20,0){6} {\circle*{3}} 
\put(148,-15){$i$}
\put(168,-15){$a$}\put(188,-15){$b$}\put(208,-15){$j$}\put(228,-15){$c$} \put(248,-15){$d$}
\qbezier(150,0)(170,30)(190,0)\qbezier(170,0)(210,50)(250,0)\qbezier(210,0)(220,20)(230,0)
\put(-40,-45){$\omega''=$}\multiput(-10,-40)(20,0){6} {\circle*{3}} 
\put(-12,-55){$i$}
\put(8,-55){$a$}\put(28,-55){$b$}\put(48,-55){$j$}\put(68,-55){$c$}\put(88,-55){$d$}
\qbezier(-10,-40)(30,8)(70,-40)\qbezier(10,-40)(30,-10)(50,-40)\qbezier(30,-40)(60,0)(90,-40)
\put(120,-45){$\upsilon''=$}\multiput(150,-40)(20,0){6} {\circle*{3}} 
\put(148,-55){$i$}
\put(168,-55){$a$}\put(188,-55){$b$}\put(208,-55){$j$}\put(228,-55){$c$}\put(248,-55){$d$}
\qbezier(150,-40)(190,8)(230,-40)\qbezier(170,-40)(210,8)(250,-40)\qbezier(190,-40)(200,-20)(210,-40)
\end{picture}
\end{center}
and by induction hypothesis $\upsilon'<\sigma$ or $\upsilon''<\sigma$. Note that $\upsilon''$ is a predecessor of $\omega''$ so that if $\upsilon''<\sigma$ we are done.
If $\upsilon'<\sigma$ then by induction hypothesis at least one of predecessors connected to arcs $(i,b),(a,d)$ is less or equal to $\sigma.$ These predecessors are
$\upsilon'_1=(i,a)(b,d)(j,c)\phi^-_{(i,j)}=\omega'$ and $\upsilon'_2=(i,d)(a,b)(j,c)\phi^-_{(i,j)}$. One has $\upsilon'_2$ is a predecessor of $(i,c)(a,b)(j,d)\phi^-_{(i,j)}$ which is a predecessor of
$\omega''$, so that in this case also either $\omega'<\sigma$ or $\omega''<\sigma.$
\end{itemize}
\item{} Let $(i,j)\in\phi$ be such that $a<i<c<j$. The corresponding predecessors of $\omega$ are $\omega'=(a,i)(b,d)(c,j)\phi^-_{(i,j)}$ and
$\omega''=(a,j)(i,c)(b,d)\phi^-_{(i,j)}$. By (2) one has either $i>b$ or $i<b$ and $j>d$.

If $i<b$ and $j>d$ then $\upsilon=(a,d)(i,j)(b,c)\phi^-_{(i,j)}$ has  predecessors $\upsilon'=(a,i)(b,c)(d,j)\phi^-_{(i,j)}$ and
$\upsilon''=(a,j)(i,d)(b,c)\phi^-_{(i,j)}$ connected to arcs $(a,d),(i,j)$. By induction hypothesis at least one of them is less or equal to $\sigma$ and since $\upsilon'$ is a predecessor of $\omega'$ and $\upsilon''$ is a predecessor of $\omega''$ we are done.

If $i>b$ and $j<d$ then $\upsilon=(a,d)(b,c)(i,j)\phi^-_{(i,j)}$ has predecessors
$\upsilon'=(a,d)(b,i)(c,j)\phi^-_{(i,j)}$ and $\upsilon''=(a,d)(b,j)(i,c)\phi^-_{(i,j)}$
connected to arcs $(b,c),(i,j)$ and since $\upsilon'$ is a predecessor of $\omega'$ and $\upsilon''$ is a predecessor of $\omega''$ by induction hypothesis applied to $\upsilon$ we are done.

Finally, the case $i>b$ and $j>d$  is more subtle so we again use a picture. One has:
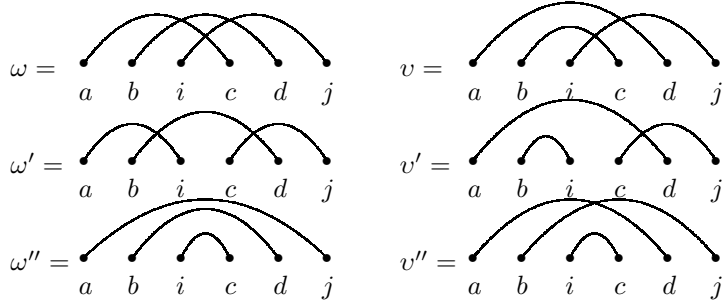
\begin{center}
\begin{picture}(130,140)(80,-60)
\setlength{\unitlength}{.92pt}
\put(-40,35){$\omega=$}
 \multiput(-10,40)(20,0){6} {\circle*{3}} 
 \put(-12,25){$a$} \put(8,25){$b$} \put(28,25){$i$} \put(48,25){$c$} \put(68,25){$d$} \put(88,25){$j$} \qbezier(-10,40)(20,80)(50,40) \qbezier(10,40)(40,80)(70,40) \qbezier(30,40)(60,80)(90,40)
\put(120,35){$\upsilon=$}
 \multiput(150,40)(20,0){6} {\circle*{3}} 
 \put(148,25){$a$} \put(168,25){$b$} \put(188,25){$i$} \put(208,25){$c$} \put(228,25){$d$} \put(248,25){$j$} \qbezier(150,40)(190,90)(230,40) \qbezier(170,40)(190,70)(210,40) \qbezier(190,40)(220,80)(250,40)
\put(-40,-5){$\omega'=$}
 \multiput(-10,0)(20,0){6} {\circle*{3}} 
 \put(-12,-15){$a$} \put(8,-15){$b$} \put(28,-15){$i$} \put(48,-15){$c$} \put(68,-15){$d$} \put(88,-15){$j$} \qbezier(-10,0)(10,30)(30,0) \qbezier(10,0)(40,40)(70,0) \qbezier(50,0)(70,30)(90,0)
\put(120,-5){$\upsilon'=$}
 \multiput(150,0)(20,0){6} {\circle*{3}} 
 \put(148,-15){$a$} \put(168,-15){$b$} \put(188,-15){$i$} \put(208,-15){$c$} \put(228,-15){$d$} \put(248,-15){$j$} \qbezier(150,0)(190,50)(230,0) \qbezier(170,0)(180,20)(190,0) \qbezier(210,0)(230,30)(250,0)
\put(-40,-45){$\omega''=$}
 \multiput(-10,-40)(20,0){6} {\circle*{3}} 
 \put(-12,-55){$a$} \put(8,-55){$b$} \put(28,-55){$i$} \put(48,-55){$c$} \put(68,-55){$d$} \put(88,-55){$j$} \qbezier(-10,-40)(40,8)(90,-40) \qbezier(10,-40)(40,0)(70,-40) \qbezier(30,-40)(40,-20)(50,-40)
\put(120,-45){$\upsilon''=$}
 \multiput(150,-40)(20,0){6} {\circle*{3}} 
 \put(148,-55){$a$} \put(168,-55){$b$} \put(188,-55){$i$} \put(208,-55){$c$} \put(228,-55){$d$} \put(248,-55){$j$} \qbezier(150,-40)(190,8)(230,-40) \qbezier(170,-40)(210,8)(250,-40) \qbezier(190,-40)(200,-20)(210,-40)
\end{picture}
\end{center}
If $\upsilon'<\sigma$ then since it is a predecessor of $\omega'$ we get $\omega'<\sigma.$ If $\upsilon''<\sigma$ then one of its predecessors
$\upsilon''_1=(a,j)(b,d)(i,c)\phi^-_{(i,j)}=\omega''$ or $\upsilon''_2=(a,b)(i,c)(d,j)\phi^-_{(i,j)}$ connected to arcs $(a,d),(b,j)$ is less or equal to $\sigma$. Thus it is enough to note that
$\upsilon''_2$ is a predecessor of $(a,i)(b,c)(d,j)\phi^-_{(i,j)}$ which is in turn is a predecessor of $\omega'$, thus again at least one of $\omega',\omega''$ is less or equal to $\sigma.$
\end{itemize}
\end{proof}
As a straightforward corollary we get
\begin{corollary}\label{main} Given  $\sigma\in I_{n,k}^{\max}$ and $\omega\in G_\sigma\setminus\{\sigma\}$, one has $\mathcal F_\omega$ is in the singular locus of $\mathcal F_\sigma$ if and only if either there exists $(a,b)\in\omega$ and fixed point $i\ :\ a<i<b$ such that both $\omega_{(a,b)}^-(a,i),\omega_{(a,b)}^-(i,b)<\sigma$ or there exists crossing pair $(a,b),(c,d)\in \omega$ where $a<c<b<d$ such that both $\omega^-_{(a,b)(c,d)}(a,c)(b,d),\omega^-_{(a,b)(c,d)}(a,d)(c,b)<\sigma$.
\end{corollary}
For $\sigma=(i_1,j_1)\ldots(i_k,j_k)\in I_{n,k}$ put $${\widetilde\sigma}^{(i_s,j_s)}=(i_1',j_1')\ldots(i_{s-1}',j_{s-1}')(i_{s+1}',j_{s+1}')\ldots(i_k',j_k')$$ to be a link pattern in $I_{n-1,k-1}$ obtained by
$$d_t'=\left\{\begin{array}{ll}d_t&{\rm if}\ d_t<i_s;\\
                                d_t-1&{\rm if}\ i_s<d_t<j_s;\\
                                                                d_t-2&{\rm if}\ d_t>j_s;\\
                                                                \end{array}\right.$$
                                                                for any $d_t\in\{i_r,j_r\}_{r=1,r\ne s}^k$.
                                                                In other words we extract $(i_s,j_s)$ together with the points $i_s,j_s.$
\begin{lemma}\label{lemma22} Let $\sigma\in I_{n,k}^{\max}$ and $\upsilon\in I_{n,k}$ be such that $(i,i+1)\in \sigma,\upsilon$.
Then ${\widetilde{\upsilon}}^{(i,i+1)}\in G_{{\widetilde\sigma}^{(i,i+1)}}$ and  $F_{{\widetilde\upsilon}^{(i,i+1)}}$ is a singular point of
$\mathcal F_{{\widetilde\sigma}^{(i,i+1)}}$ if and only if $\upsilon\in G_{\sigma}$ and respectively $F_\upsilon$ is a singular point of $\mathcal F_\sigma$.
In particular, $\upsilon\in Sing(\sigma)$ iff ${\widetilde{\upsilon}}^{(i,i+1)}\in Sing({\widetilde\sigma}^{(i,i+1)}).$
\end{lemma}
\begin{proof}
Both parts of the Lemma are immediate. For the first part it is enough to note that for $\omega$ such that $(i,i+1)\in\omega$ one has

$$R_{s,t}(\omega)=\left\{\begin{array}{ll}R_{s,t}({\widetilde{\omega}}^{(i,i+1)})&{\rm if}\ t\leq i;\\
                                          R_{s-2,t-2}({\widetilde{\omega}}^{(i,i+1)})&{\rm if}\ s\geq i+1;\\
                                          R_{s,t-2}({\widetilde{\omega}}^{(i,i+1)})+1&{\rm if}\ s<i+1,\ t>i;\\
                                                                                     \end{array}\right.$$
Thus $R_{s,t}(\sigma)\geq R_{s,t}(\upsilon)$ for all $1\leq s<t\leq n$ if and only if $R_{s,t}({\widetilde\sigma}^{(i,i+1)})\geq    R_{s,t}({\widetilde\upsilon}^{(i,i+1)})$for any $1\leq s<t\leq n-2.$ Moreover since  for any predecessor $\upsilon'$ of $\upsilon$ one has $(i,i+1)\in\upsilon'$ we get exactly the same result for $\upsilon'$.
\end{proof}
In our next lemma we concentrate on $\omega\in Sing(\sigma)$.
\begin{lemma}\label{l3.23} Let $\sigma\in I_{n,k}^{\max}$ with $\rho(\sigma)\geq 4$ and let $\omega\in Sing(\sigma)$ such that there exists $f\in\omega^0$ and $(i,j)\in\omega$
such that $i<f<j$ and both $\omega_{(i,j)}^-(i,f),\,\omega_{(i,j)}^-(f,j)\in G_\sigma.$ Let $(a,b)\in\omega$ be minimal such arc, that is for any $(i,j)\in\omega$ satisfying the condition, one has $(i,j)\not\in[a,b]$. Put $k_1=|\{(i,j)\in\sigma\ :\ i<f<j\}|$. Let $S=\{(i,j)\in\omega\ :\ i<f<j\}$. One has
\begin{itemize}
\item[(i)] $|S|\geq k_1+2;$
\item[(ii)] All arcs of $S$ are concentric (that is for any $(i,j),(i',j')\in S$ such that $i<i'$ one has $i<i'<j'<j$ and $(a,b)$ is the minimal arc of $S$, that is if $(i,j)\in S\setminus\{(a,b)\}$ then $i<a<b<j$. For any $(i,j)\in S$ and any $f'\in\omega^0\cap[a,b]$ one has $\omega_{(i,j)}^-(i,f'),\,\omega_{(i,j)}^-(f',j)\in G_\sigma$, and there are no fixed points at $[i,j]\backslash[a,b]$ for any $(i,j)\in S$.
\item[(iii)] For any $(i,j)\in\omega\setminus S$ there are no fixed points under $(i,j)$ and any $(i,j),(i',j')$ in $\omega\setminus S$ do not intersect, and if $(i,j)\in\omega\backslash S$ intersects any $(c,d)\in S$ then $(i,j)$ intersects all the arcs of $S$.
\end{itemize}
\end{lemma}
\begin{proof}
(i) Let $k_2=|\{(i,j)\in \sigma\ :\ j\leq f\}|$ and $k_3=|\{(i,j)\in\sigma\ :\ i\geq f\}|$ then $k=k_1+k_2+k_3.$ Exactly in the same manner for $\upsilon\in G_\sigma$ let $m_1(\upsilon)=|\{(i,j)\in\upsilon\ :\ i<f<j\}|$,
$m_2(\upsilon)=|\{(i,j)\in\omega\ :\ j\leq f\}|$ and $m_3(\upsilon)=|\{(i,j)\in\upsilon\ :\ i\geq f\}|$. Again one has $m_1(\upsilon)+m_2(\upsilon)+m_3(\upsilon)=k$ and since $\upsilon\leq \sigma$ we get $m_2(\upsilon)\leq k_2$ and $m_3(\upsilon)\leq k_3.$
One also has $m_2(\omega_{(a,b)}^-(a,f))=m_2(\omega)+1\leq k_2$ and $m_3(\omega_{(a,b}^-(f,b)=m_3(\omega)+1\leq k_3.$ Thus $|S|=k-m_2(\omega)-m_3(\omega)\geq k_1+2.$

(ii) We have to consider the following three cases:
\begin{itemize}
\item[(a)] There exists $(c,d)\in\omega$ such that $a<c<f<d<b$. We show in this case that there exists $\omega'>\omega$ such that $\omega'$ is singular in $G_\sigma$.

\item[(b)] There exists $(c,d)\in S$ crossing $(a,b)$. Again, we show in this case that there exists $\omega'>\omega$ such that $\omega'$ is singular in $G_\sigma$.

\item[(c)] There exists $(c,d)\in S$ and exists $p\in \omega^0\backslash[a,b]$ with $c<p<d$. Again, we show in this case that there exists $\omega'>\omega$ such that $\omega'$ is singular in $G_\sigma$.
\end{itemize}
Since all three cases are shown in the same way, we will consider in detail only case (a).
Assume first that there exists $(c,d)\in\omega$ such that $a<c<f<d<b$. By Proposition \ref{prop3.20} and minimality of $(a,b)$ one has that exactly one out of
$\omega_{(c,d)}^-(c,f),\, \omega_{(c,d)}^-(f,d)$ is in $G_\sigma.$ Without loss of generality we can assume that this is $\upsilon=\omega_{(c,d)}^-(c,f)$.
One has
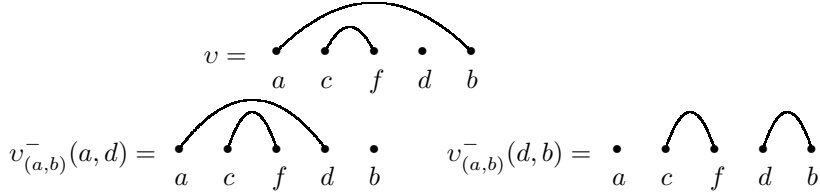
\begin{center}
\begin{picture}(130,85)(80,-20)
\setlength{\unitlength}{.92pt}
\put(20,35){$\upsilon=$}
\multiput(50,40)(20,0){5} {\circle*{3}} 
\put(48,25){$a$}\put(68,25){$c$}\put(88,25){$f$}\put(108,25){$d$}\put(128,25){$b$}
\qbezier(50,40)(90,80)(130,40)\qbezier(70,40)(80,60)(90,40)
\put(-60,-5){$\upsilon_{(a,b)}^-(a,d)=$}
\multiput(10,0)(20,0){5} {\circle*{3}} 
\put(8,-15){$a$}\put(28,-15){$c$}\put(48,-15){$f$}\put(68,-15){$d$}\put(88,-15){$b$}
\qbezier(10,0)(40,40)(70,0)\qbezier(30,0)(40,30)(50,0)
\put(120,-5){$\upsilon_{(a,b)}^-(d,b)=$}
\multiput(190,0)(20,0){5} {\circle*{3}} 
\put(188,-15){$a$}\put(208,-15){$c$}\put(228,-15){$f$}\put(248,-15){$d$}\put(268,-15){$b$}
\qbezier(210,0)(220,30)(230,0)\qbezier(250,0)(260,30)(270,0)
\end{picture}
\end{center}
Again, one of $\upsilon_{(a,b)}^-(a,d),\, \upsilon_{(a,b)}^-(d,b)$ is in $G_\sigma.$
\begin{itemize}
\item If $\upsilon_{(a,b)}^-(a,d)\in G_\sigma$ recall that $\omega'=\omega_{(a,b)}^-(f,b)<\sigma$ so that
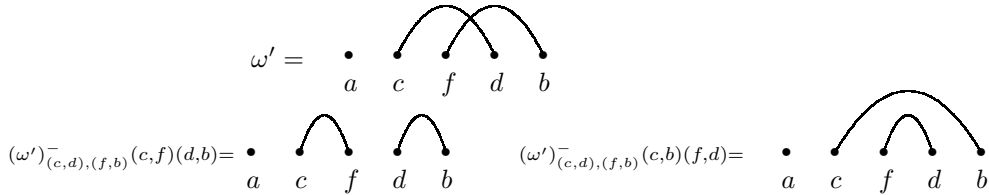
\begin{center}
\begin{picture}(130,100)(80,-20)
\setlength{\unitlength}{.92pt}
\put(10,35){$\omega'=$}
\multiput(50,40)(20,0){5} {\circle*{3}} 
\put(48,25){$a$}\put(68,25){$c$}\put(88,25){$f$}\put(108,25){$d$}\put(128,25){$b$}
\qbezier(70,40)(90,80)(110,40)\qbezier(90,40)(110,80)(130,40)
\put(-90,0){${}_{(\omega')_{(c,d),(f,b)}^-(c,f)(d,b)=}$}
\multiput(10,0)(20,0){5} {\circle*{3}} 
\put(8,-15){$a$}\put(28,-15){$c$}\put(48,-15){$f$}\put(68,-15){$d$}\put(88,-15){$b$}
\qbezier(30,0)(40,30)(50,0)\qbezier(70,0)(80,30)(90,0)
\put(120,0){${}_{(\omega')_{(c,d),(f,b)}^-(c,b)(f,d)=}$}
\multiput(230,0)(20,0){5} {\circle*{3}} 
\put(228,-15){$a$}\put(248,-15){$c$}\put(268,-15){$f$}\put(288,-15){$d$}\put(308,-15){$b$}
\qbezier(250,0)(280,50)(310,0)\qbezier(270,0)(280,30)(290,0)
\end{picture}
\end{center}
If $(\omega')_{(c,d),(f,b)}^-(c,f)(d,b)<\sigma$ then we get that $\upsilon$ satisfies $\upsilon>\omega$ and both predecessors $\upsilon_{(a,b)}^-(a,d),\,
\upsilon_{(a,b)}^-(d,b)<\sigma$ so that $\omega\not\in Sing(\sigma).$ \\
If $(\omega')_{(c,d),(f,b)}^-(c,b)(f,d)<\sigma$ then since $\omega_{(c,d)}^-(f,d)<(\omega')_{(c,d),(f,b)}^-(c,b)(f,d)$ we get that both predecessors of $\omega$ connected to $(c,d)$ and $f$ are in the graph, in contradiction with the assumption.
\item If $\upsilon_{(a,b)}^-(d,b)<\sigma$ recall that $\omega'=\omega_{(a,b)}^-(a,f)<\sigma$ so that
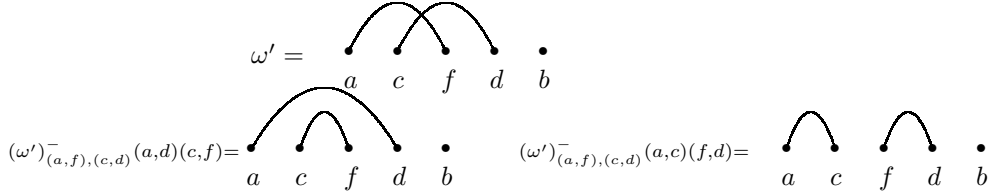
\begin{center}
\begin{picture}(130,100)(80,-20)
\setlength{\unitlength}{.92pt}
\put(10,35){$\omega'=$}
\multiput(50,40)(20,0){5} {\circle*{3}} 
\put(48,25){$a$}\put(68,25){$c$}\put(88,25){$f$}\put(108,25){$d$}\put(128,25){$b$}
\qbezier(70,40)(90,80)(110,40)\qbezier(50,40)(70,80)(90,40)
\put(-90,0){${}_{(\omega')_{(a,f),(c,d)}^-(a,d)(c,f)=}$}
\multiput(10,0)(20,0){5} {\circle*{3}} 
\put(8,-15){$a$}\put(28,-15){$c$}\put(48,-15){$f$}\put(68,-15){$d$}\put(88,-15){$b$}
\qbezier(30,0)(40,30)(50,0)\qbezier(10,0)(40,50)(70,0)
\put(120,0){${}_{(\omega')_{(a,f),(c,d)}^-(a,c)(f,d)=}$}
\multiput(230,0)(20,0){5} {\circle*{3}} 
\put(228,-15){$a$}\put(248,-15){$c$}\put(268,-15){$f$}\put(288,-15){$d$}\put(308,-15){$b$}
\qbezier(230,0)(240,30)(250,0)\qbezier(270,0)(280,30)(290,0)
\end{picture}
\end{center}
Exactly as in the first case,\\
if $(\omega')_{(a,f),(c,d)}^-(a,d)(c,f)$ then $\omega<\upsilon$ satisfying $\upsilon_{(a,b)}^-(a,d),\, \upsilon_{(a,b)}^-(d,b)<\sigma$
so that $\omega\not\in Sing(\sigma);$\\
if $(\omega')_{(a,f),(c,d)}^-(a,c)(f,d)<\sigma$ then both predecessors of $\omega$ connected to $(c,d)$ and $f$ are in the graph, in contradiction with the assumption.
\end{itemize}
So there cannot be $(c,d)\in\omega$ such that $a<c<f<d<b$.

Now, let us show that for $(c,d)\in S$ one has $\omega_{(c,d)}^-(c,f),\,\omega_{(c,d)}^-(f,d)<\sigma$. Indeed, since $\omega_{(a,b)}^-(a,f)<\sigma$  one of its predecessors $\omega_{(a,b),(c,d)}^-(a,f)(c,b),\, \omega_{(a,b),(c,d)}^-(a,f)(b,d)$ is in $G_\sigma$. Exactly in the same way, since $\omega_{(a,b)}^-(f,b)<\sigma$ one of its predecessors $\omega_{(a,b),(c,d)}^-(f,b)(c,a),\, \omega_{(a,b),(c,d)}^-(f,b)(a,d)$
is in $G_\sigma.$ Note that
\begin{itemize}
\item[(i)]  $\omega_{(c,d)}^-(c,f)<\omega_{(a,b),(c,d)}^-(a,f)(c,b)$ so that if $\omega_{(a,b),(c,d)}^-(a,f)(c,b)\leq \sigma$ then $\omega_{(c,d)}^-(c,f)< \sigma$;
\item[(ii)]  $\omega_{(c,d)}^-(f,d)<\omega_{(a,b),(c,d)}^-(a,f)(b,d)$ so that if $\omega_{(a,b),(c,d)}^-(a,f)(b,d)\leq \sigma$ then $\omega_{(c,d)}^-(f,d)<\sigma$;
\item[(iii)] $\omega_{(c,d)}^-(c,f)<\omega_{(a,b),(c,d)}^-(c,a)(f,b)$ so that if $\omega_{(a,b),(c,d)}^-(c,d)(f,b)\leq \sigma$ then $\omega_{(c,d)}^-(c,f)<\sigma$;
\item[(iv)] $\omega_{(c,d)}^-(f,d)<\omega_{(a,b),(c,d)}^-(a,d)(f,b)$ so that if $\omega_{(a,b),(c,d)}^-(a,d)(f,b)\leq\sigma$ then $\omega_{(c,d)}^-(f,d)<\sigma$;
\end{itemize}
Thus if either $\omega_{(a,b),(c,d)}^-(a,f)(c,b),\, \omega_{(a,b),(c,d)}^-(a,d)(f,b)\in G_\sigma$ or\\  $\omega_{(a,b),(c,d)}^-(a,f)(b,d),\, \omega_{(a,b),(c,d)}^-(c,a)(f,b)\in G_\sigma$ then both $\omega_{(c,d)}^-(c,f),\,\omega_{(c,d)}^-(f,d)<\sigma$. Let us show that it cannot occur that
either $$\omega_{(a,b),(c,d)}^-(a,f)(c,b),\, \omega_{(a,b),(c,d)}^-(c,a)(f,b)\leq \sigma$$ or $$\omega_{(a,b),(c,d)}^-(a,f)(b,d),\, \omega_{(a,b),(c,d)}^-(a,d)(f,b)\leq \sigma.$$
Indeed, if
 $$\omega_{(a,b),(c,d)}^-(a,f)(c,b),\, \omega_{(a,b),(c,d)}^-(c,a)(f,b)\leq \sigma$$
then $\omega_{(c,d)}^-(c,f)\in G_\sigma$ and both
 $$\omega_{(a,b),(c,d)}^-(a,f)(c,b),\, \omega_{(a,b),(c,d)}^-(c,a)(f,b)$$
are its predecessors, obtained from ``uncrossing'' $(c,f),(a,b)$ so that $\mathcal F_{\omega_{(a,b),(c,d)}^-(c,f)(a,b)}$ is singular in $\mathcal F_\sigma$ but $\omega_{(c,d)}^-(c,f)$ is a predecessor of $\omega$ , which contradicts to $\omega\in Sing(\sigma)$.
Exactly in the same way if $\omega_{(a,b),(c,d)}^-(a,f)(b,d),\, \omega_{(a,b),(c,d)}^-(a,d)(f,b)\leq \sigma$ then $\omega_{(c,d)}^-(f,d)\in G_\sigma$ and both $\omega_{(a,b),(c,d)}^-(a,f)(b,d),\, \omega_{(a,b),(c,d)}^-(a,d)(f,b)$ are its predecessors obtained from ``uncrossing'' $(a,b),(f,d)$ which again contradicts to
$\omega\in Sing(\sigma)$.

Note that in our proof that $(c,d)\in S$ cannot cross $(a,b)$ we used only fact that both $$\omega_{(a,b)}^-(a,f),\omega_{(a,b)}^-(f,b)<\sigma$$
but not its minimality, so that this is true for any $(i,j),(i',j')\in S$ and all of them are concentric.

Consider $f\ne p\in\omega^0\cap[a,b]$, let us show that $\omega_{(a,b)}^-(a,p)$, $\omega_{(a,b)}^-(p,b)<\sigma$ and $\omega_{(a,b)}^-(p,f)\not<\sigma$.
We can assume that $p<f$ then since $\omega_{(a,b)}^-(p,b)<\omega_{(a,b)}^-(f,b)$ we get that  $\omega_{(a,b)}^-(p,b)<\sigma.$ Now since $a<p<f$ one has either
$\omega_{(a,b)}^-(a,p)<\sigma$ or $\omega_{(a,b)}^-(p,f)<\sigma$. If $\omega_{(a,b)}^-(p,f)<\sigma$ then for $\omega'=\omega_{(a,b)}^-(p,b)$ both predecessors
$(\omega')_{(p,b)}-(p,f)$, $(\omega')_{(p,b)}-(f,b)<\sigma$ so that $\mathcal F_{\omega'}$ is singular in $\mathcal F_\sigma$ in contradiction with $\omega\in Sing(\sigma)$.

To finish with (ii), let us show that if there exists $(c,d)\in S$ and $p\in\omega^0$ such that $p\in [c,d]\backslash[a,b]$, then $\omega\not\in Sing(\sigma)$. Without loss of generality, we can assume that $p\in(c,a)$. Then by Proposition \ref{prop3.20} either $\omega^-_{(c,d)}(c,p)$ or $\omega^-_{(c,d)}(p,d)$ is in $G_\sigma$.
\begin{itemize}
\item[(a)] If $\omega^-_{(c,d)}(c,p)\in G_\upsilon$, then noting that both $\omega^-_{(a,b)}(a,f)$ and $\omega^-_{(a,b)}(f,b)$ are in $G_\sigma$, we get that $\sigma>\omega^-_{(c,d)(a,b)}(c,p)(a,f),\omega^-_{(c,d)(a,b)}(c,p)(f,b)$ so that $\omega^-_{(c,d)}(c,p)$ is singular in $G_\upsilon$ and $\omega>\omega^-_{(c,d)}(c,p)$ is not in $Sing(\sigma)$.
\item[(b)] If $\omega^-_{(c,d)}(p,d)\in G_\upsilon$, then noting that on one hand  $\omega^-_{(c,d)}(f,d)=(\omega^-_{(c,d)}(p,d))^-_{(p,d)}(f,d)\in G_\sigma$. On the other hand, one of the predecessors of $\omega^-_{(c,d)}(c,f)$ connected to $p$ must be in $G_\sigma$. Note that $(\omega^-_{(c,d)}(c,f))^-_{(c,f)}(c,p)=\omega^-_{(c,d)}(c,p)$ so we return to case (a), or $(\omega^-_{(c,d)}(c,f))^-_{(c,f)}(p,f)=(\omega^-_{(c,d)}(p,d))^-_{(p,d)}(p,f)$ so that both predecessors of $\omega^-_{(c,d)}(p,d)$ connected to $f$ are in $G_\sigma$, which implies $\omega\not\in Sing(\sigma)$, again.
\end{itemize}

(iii) Note first of all that if $(i,j)\not\in S$, then $(i,j)$ cannot be a bridge over $p\in\omega^0\cap[a,b]$, by (ii). Now, again, we have to consider the following cases:
\begin{itemize}
\item[(a)] there is $p\in\omega^0\backslash[a,b]$ and there is $(i,j)\in\omega$ such that $i<p<j$. In this case, we show that there exists $\omega'>\omega$ such that $\omega'$ is singular in $G_\sigma$.

\item[(b)] there exist $(i,j),(i',j')\in\omega\backslash S$ such that $i<i'<j<j'$. Again, in this case, we show that there exists $\omega'>\omega$ such that $\omega'$ is singular in $G_\sigma$.

\item[(c)] there is $(i,j)\not\in S$ and $(c,d),(a,b)\in S$ with $c<a<b<d$ such that either $(i,j)$ crosses $(c,d)$ but not $(a,b)$ or $(i,j)$ crosses $(a,b)$ but not $(c,d)$. In this case, we show that there exists $\omega'>\omega$ such that $\omega'$ is singular in $G_\sigma$.
\end{itemize}
Exactly as in (ii) all the cases are analyzed in the same way, so we show only case (a).
Assume that there is a fixed point $p$ under $(i,j)\in\omega\backslash S$. Since any fixed point at $[a,b]$ has the same role as $f$, and there are no fixed points under $(c,d)\in S$ such that they are not at $[a,b]$, by (ii) we get that $p$ is not under any arc of $S$. Thus, without loss of generality, we can assume that $i<p<c$ for any $(c,d)\in S$.

As in the proof of (ii), we see that at least one of
$$\omega^-_{(i,j)}(i,p),\omega^-_{(i,j)}(p,j)$$
is smaller than $\sigma$. We have either $i<p<j<a$ or $i<p<a<j<f$, where $f$ is the most left fixed point at $[a,b]$. So let us consider these two cases:
\begin{itemize}
\item[(a1)] Case $i<p<j<a$. We have
$$\omega^-_{(a,b)}(a,f),\omega^-_{(a,b)}(f,b)\in G_\sigma$$
and either $\omega^-_{(i,j)}(i,p)$ or $\omega^-_{(i,j)}(p,j)$ in $G_\sigma$. Note that
$$(R_\sigma)_{i,p}\geq (R_\omega)_{i,p}+1\mbox{ or }(R_\sigma)_{p,j}\geq (R_\omega)_{p,j}+1$$
and
$$(R_\sigma)_{a,f}\geq (R_\omega)_{a,f}+1\mbox{ and }(R_\sigma)_{f,b}\geq (R_\omega)_{f,b}+1$$
and $[i,p]\cap[a,f]=\emptyset$, $[p,j]\cap[a,f]=\emptyset$, respectively. Thus
$$\upsilon=\left\{
\begin{array}{l}
\omega_0(i,p)(a,f),\omega_0(i,p)(f,b)\in G_\sigma,\\
\qquad\qquad\qquad\mbox{if }\omega^-{(i,j)}(i,p)\in G_\sigma,\\
\omega_0(p,j)(a,f),\omega_0(p,j)(f,b)\in G_\sigma,\\
\qquad\qquad\qquad\mbox{if }\omega^-_{(i,j)}(p,j)\in G_\sigma.
\end{array}\right.$$
So $\omega^-_{(i,j)}(i,p)$ or $\omega^-_{(i,j)}(p,j)$ is singular in $G_\sigma$, respectively. Hence, $\omega\not\in Sing(\sigma)$.

\item[(a2)] Case $i<p<a<j<f$. If $\omega^-_{(i,j)}(i,p)\in G_\sigma$, then exactly as in Case (a1), we get that $\omega^-_{(i,j)}(i,p)$ is singular in $G_\sigma$, so that $\omega\not\in Sing(\sigma)$.
    Now, let $\omega^-_{(i,j)}(p,j)\in G_\sigma$:
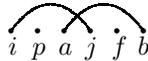
\begin{center}
\begin{picture}(60,15)
\put(0,0){\circle*{2}}\put(10,0){\circle*{2}}\put(20,0){\circle*{2}}\put(30,0){\circle*{2}}
\put(40,0){\circle*{2}}\put(50,0){\circle*{2}}
\qbezier(0,0)(15,20)(30,0)
\qbezier(20,0)(35,20)(50,0)
\put(-1,-9){$i$}\put(8,-9){$p$}\put(18,-9){$a$}\put(28,-9){$j$}\put(38,-9){$f$}\put(48,-9){$b$}
\end{picture}
\end{center}
Note that $R_\omega=R_0+R_{(i,j)(a,b)}$ and the corresponding part of $R_{(i,j)(a,b)}$ at $[i,b]$ (not taking into account other entries) is

\begin{align*}
R_1&=\tilde{R}_{(i,j)(a,f)}=
\begin{array}{lllllll}
 &i&p&a&j&f&b\\
i&0&0&0&1&2&2\\
p& &0&0&0&1&1\\
a& & &0&0&1&1\\
j& & & &0&0&0\\
f& & & & &0&0\\
b& & & & & &0
\end{array}
\qquad
R_3=\tilde{R}_{(p,j)(a,b)}=
\begin{array}{lllllll}
&i&p&a&j&f&b\\
i&0&0&0&1&1&2\\
p&&0&0&1&1&2\\
a&&&0&0&0&1\\
j&&&&0&0&0\\
f&&&&&0&0\\
b&&&&&&0
\end{array}\\[10pt]
R_2&=\tilde{R}_{(i,j)(f,b)}=
\begin{array}{lllllll}
&i&p&a&j&f&b\\
i&0&0&0&1&1&2\\
p&&0&0&0&0&1\\
a&&&0&0&0&1\\
j&&&&0&0&1\\
f&&&&&0&1\\
b&&&&&&0
\end{array}
\end{align*}
Now, let $(M)_{s,t}=\max\{(R_1)_{s,t},(R_3)_{s,t}\}$, we get
\begin{align*}
M&:=\tilde{R}_{(i,j)(f,b)}=
\begin{array}{lllllll}
&i&p&a&j&f&b\\
i&0&0&0&1&2&2\\
p&&0&0&1&1&2\\
a&&&0&0&1&1\\
j&&&&0&0&0\\
f&&&&&0&0\\
b&&&&&&0
\end{array}
\end{align*}
is not a rank matrix but since $M_{p,j}=1$ and $M_{a,f}=1$, we get that $\omega_0(p,j)(a,f)\in G_\sigma$.

Now, let $(N)_{s,t}=\max\{(R_2)_{s,t},(R_3)_{s,t}\}$, we get
\begin{align*}
N&:=\tilde{R}_{(i,j)(f,b)}=
\begin{array}{lllllll}
&i&p&a&j&f&b\\
i&0&0&0&1&1&2\\
p&&0&0&1&1&2\\
a&&&0&0&0&1\\
j&&&&0&0&1\\
f&&&&&0&1\\
b&&&&&&0
\end{array}
\end{align*}
where $N$ is a rank matrix of $\omega_0(f,b)(p,j)$ so it is in $G_\sigma$. So, $\omega^-_{(i,j)}(p,j)$ is in $G_\sigma$ and $\omega\not\in Sing(\sigma)$.
\end{itemize}
This completes the proof.
\end{proof}

As a corollary we get the following proposition
\begin{proposition}\label{prop3.24}
Let $\sigma\in I_{n,k}^{\max}$ be such that $1,n\in \sigma^0$ and
$|\tau^*(\sigma)|\geq 2$. Let $\omega\in Sing(\sigma)$. If there
exists $f\in \omega^0$ and $(i,j)\in\omega$ such that
$\omega_{(i,j)}^-(i,f),\omega_{(i,j)}^-(f,j)<\sigma$ then $\omega$
is obtained by our algorithm.
\end{proposition}
\begin{proof}
We prove it by induction on $n$ starting with $n=6$ and $\sigma=(2,3)(4,5)$, which is known to be true.

Let $(a,b)\in\omega$ be the minimal bridge over $f$.
By Lemma \ref{l3.23} $\pi_{a+1,b-1}(\omega)$ either equal to $\pi_{a+1,b-1}(\sigma)$ or is obtained from $\pi_{a+1,b-1}(\sigma)$ by deleting some external arcs.
Let us show, that if some external arc of $\pi_{a+1,b-1}(\sigma)$ is deleted then $\omega\not\in Sing(\sigma).$ First of all note that $f\in(\pi_{a+1,b-1}\sigma)^0$.
Indeed, if it is not a fixed point then either there exists $(i,f)\in\sigma$ where $i\geq a+1$ or $(f,j)\in\sigma$ where $j\leq b-1.$ We can assume $(i,f)\in\sigma$, but then $i$ is a fixed point of $\omega$ and then $\omega_{(a,b)}^-(i,b)$ satisfies $\omega_{(a,b)}^-(i,f),\omega_{(a,b)}^-(f,b)<\sigma$ in contradiction to $\omega\in Sing(\sigma).$ Thus, $f\in(\pi_{a+1,b-1}(\sigma))^0.$ Now, if $\pi_{a+1,b-1}(\omega)\ne \pi_{a+1,b-1}(\sigma)$ then there exists $(i,j)\in \pi_{a+1,b-1}(\sigma)$ such that $(i,j)\not\in \pi_{a+1,b-1}(\omega)$, and by maximality of  $\pi_{a+1,b-1}(\sigma)$ either $j<f$ or $i>f$. We can assume that $j<f$ and then we get again $\omega_{(a,b)}^-(i,b)$ satisfies $\omega_{(a,b)}^-(i,f),\omega_{(a,b)}^-(f,b)<\sigma$ in contradiction to $\omega\in Sing(\sigma).$ Thus, $\pi_{a+1,b-1}(\omega)=\pi_{a+1,b-1}(\sigma)$. Now consider $\pi_{a,b}(\sigma).$ Since $f$ is a fixed point of  $\pi_{a+1,b-1}(\sigma)$ one has $(a,b)\not\in\pi_{a,b}(\sigma)$.
Moreover, since $\pi_{a+1,f-1}(\sigma)=\pi_{a+1,f-1}(\omega)$ and $R_{a,f}(\omega)=R_{a+1,f-1}(\omega)\leq R_{a,f}(\sigma)-1$ one has by maximality of $\sigma$
that $(a,i)\in\sigma$ where $i=\min(\pi_{a+1,b-1}(\sigma))^0$ and obviously $i\leq f$. Exactly in the same way $(j,b)\in\sigma$ where $j=\max(\pi_{a+1,b-1}(\sigma))^0$
and $j\geq \{f,i+1\}$. Then one has $\pi_{a,b}(\omega)=(\pi_{a,b}(\sigma))^-_{(a,i),(j,b)}(a,b)$.

Moreover if $\pi_{a+1,b-1}(\sigma)\ne\emptyset$ then there exists $(i,i+1)\in\sigma,\omega$ and then by Lemma \ref{lemma22} and by induction $\omega$ is obtained from $\sigma$ by our algorithm.

Thus we are left with the case $\pi_{a,b}(\sigma)=(a,a+1)(b-1,b)$ and $\pi_{a,b}(\omega)=(a,b).$ Further,
\begin{itemize}
\item either $b>a+3$ then if either $(a-1,a+2)\in\sigma$  or $(b-2,b+1)\in\sigma$ we can assume $(a-1,a+2)\in\sigma$ then $\pi_{a-1,b}(\omega)=(a-1,a+1)(a,b)$ and let
$\upsilon=\omega_{(a-1,a+1),(a,b)}^-(a-1,b)(a,a+1)>\omega$ and $\upsilon_{(a-1,b)}^-(a-1,f),\upsilon_{(a-1,b)}(f,b)<\sigma$, so that $\omega\not\in Sing(\sigma)$. Thus, if $b>a+3$ we get that $a-1,b+1\in (\pi_{a-1,b+1}(\sigma))^0$.
If $(a-1,b+1)\in\omega$ then $R_{a-1,b+1}(\omega)=2=R_{a-1,b+1}(\sigma)$ and then $\omega$ is obtained by our algorithm, either by induction since at least one of
$1,n$ is a fixed point of $\omega$ or $a-1=1,\ b+1=n$ and then straightforwardly.
\item or $b=a+3$. If $\sigma=(2,n-1)(3,n-2)\ldots(a-1,a+4)(a,a+1)(a+2,a+3)$ then $\omega=(1,n)\ldots(a-1,a+4)(a,a+3)$ by straightforward computation. Otherwise, there exists $(i,i+1)\in\sigma$ where $i<a$ or $i>b$. We can assume $i<a$. Again a straightforward computation shows either $(i,i+1)\in\omega$ or $1\in\omega^0.$ In both cases we get the result by induction.
\end{itemize}
\end{proof}
 In order to finish the proof we have to show
\begin{proposition}\label{prop3.25}
Let $\sigma\in I_{n,k}^{\max}$ be such that $1,n\in\sigma^0$ and let $\omega\in Sing(\sigma)$ be such that there is no $f\in\omega^0$ and $(a,b)\in\omega$ such that
$$a<f<b\mbox{ and }\omega_{(a,b)}^-(a,f),\omega_{(a,b)}^-(f,b)<\sigma$$ then $b(\omega)=0$. In particular, $1,n\in \omega^0$ and $\pi_{2,n-1}(\omega)\in Sing(\pi_{2,n-1}(\sigma))$ so that $\omega$ is obtained by our algorithm.
\end{proposition}
\begin{proof}
Since $\omega\in Sing(\sigma)$ and there is no $f\in\omega^0$ and $(a,b)\in\omega$ such that
$a<f<b$ and $\omega_{(a,b)}^-(a,f),$ $\omega_{(a,b)}^-(f,b)<\sigma$, there exist $(a,c),(b,d)\in\omega$ where $a<b<c<d$ such that $\omega_{(a,c),(b,d)}^-(a,b)(c,d),$ $
\omega_{(a,c),(b,d)}^-(a,d)(b,c)\leq \sigma$ and for both of them $\mathcal F_{\omega_{(a,c),(b,d)}^-(a,b)(c,d)}$, $\mathcal F_{\omega_{(a,c),(b,d)}^-(a,d)(b,c)}$
are smooth in $\mathcal F_\sigma.$ Assume there exists $f\in \omega^0$ and $(i,j)\ne(a,c),(b,d)$ such that $i<f<j$ then
the same is true for  $\omega_{(a,c),(b,d)}^-(a,b)(c,d),\,
\omega_{(a,c),(b,d)}^-(a,d)(b,c)$. Thus exactly one out of $$\omega_{(a,c),(b,d),(i,j)}^-(a,b)(c,d))(i,f),\,\omega_{(a,c),(b,d),(i,j)}^-(a,b)(c,d))(f,j)$$ is in $G_\sigma$ and exactly one out of $$\omega_{(a,c),(b,d),(i,j)}^-(a,d)(b,c))(i,f),\, \omega_{(a,c),(b,d),(i,j)}^-(a,d)(b,c))(f,j)$$ is in $G_\sigma$. If it is the same one, for example $$\omega_{(a,c),(b,d),(i,j)}^-(a,b)(c,d))(i,f),\, \omega_{(a,c),(b,d),(i,j)}^-(a,d)(b,c))(i,f)$$ then
$\omega_{(i,j)}^-(i,f)$ has 2 predecessors in $G_\sigma$ connected to ``uncrossing'' $(a,c),(b,d)$ in contradiction to $\omega\in Sing(\sigma)$
If they are different, for example
$$\omega_{(a,c),(b,d),(i,j)}^-(a,b)(c,d)(i,f),\, \omega_{(a,c),(b,d),(i,j)}^-(a,d)(b,c)(f,j)$$
is in $G_\sigma$ then
\begin{align*}
\omega_{(i,j)}^-(i,f)&<\omega_{(a,c),(b,d),(i,j)}^-(a,b)(c,d))(i,f),\\
\omega_{(i,j)}^-(f,j)&<\omega_{(a,c),(b,d),(i,j)}^-(a,d)(b,c))(f,j)
\end{align*}
so that both predecessors connected to $(i,j)$ and $f$ are in $G_\sigma$ in contradiction with the conditions.

Now assume $a<f<b$, put $\omega'=\omega_{(a,c),(b,d)}^-$ then one of $\omega'(f,b)(c,d), \, \omega'(a,f)(c,d)<\sigma$ and one of
$\omega'(f,d)(b,c),\, \omega'(a,f)(b,c)$ is in $G_\sigma.$
\begin{itemize}
\item If $\omega'(a,f)(c,d),\, \omega'(a,f)(b,c)<\sigma$ then they are both predecessors of $\omega'(a,f)$ $(b,d)>\omega$ in contradiction to $\omega\in Sing(\sigma);$
\item If $\omega'(a,f)(c,d),\, \omega'(f,d)(b,c)<\sigma$ then by maximality of $\sigma$ exactly as in the proof of Lemma \ref{l3.23} $\omega'(f,c)(b,d)$ has 2 predecessors in $G_\sigma$ connected to ``uncrossing'' $(f,c)(b,d)$  in contradiction to $\omega\in Sing(\sigma);$
\item If $\omega'(f,b)(c,d), \omega'(a,f)(b,c)<\sigma$ then again, exactly as in the proof of of Lemma \ref{l3.23} by maximality of $\sigma$ one has $\omega'(f,c)(b,d)$ has two predecessors connected to
``uncrossing'' $(f,c),(b,d)$ in $G_\sigma$  in contradiction to $\omega\in Sing(\sigma);$
\item If $\omega'(f,b)(c,d), \omega'(f,d)(b,c)<\sigma$ then $\omega'(f,c)(b,d)$ has two predecessors connected to
``uncrossing'' $(f,c),(b,d)$ in $G_\sigma$  in contradiction to $\omega\in Sing(\sigma);$
\end{itemize}
Thus there cannot be fixed points at $[a,b],[c,d]$. The last case is there exists $f\in\omega^0\cap[b,c]$. In this case by symmetry it is enough to to consider case
$\omega'(b,f)(a,d)<\sigma.$ Then for $\omega'(a,f)(b,d)$ one has $\omega'(a,b)(f,d)<\omega'(a,b)(c,d)<\sigma$ and $\omega'(a,d)(b,f)<\sigma$ and $\omega'(a,b)(f,d)>\omega$ in contradiction to $\omega\in Sing(\sigma).$

Thus there are no fixed points under arcs of $\omega$. If $1\not\in\omega^0$ then let $i=\min \omega^0$. One has $i$ is odd and $(R_{\omega})_{1,i-1}=0.5(i-1)$.  Since $1$ is fixed point of $\sigma$ one has $(R_{\sigma})_{1,i-1}=(R_{\sigma})_{2,i-1}<0.5(i-1)$ which contradicts $\omega<\sigma$. Thus both $1,n$ are fixed points of $\omega$.
\end{proof}
As an immediate corollary we get
\begin{corollary} For all $\sigma\in I_{n,k}^{\max}$ with $\rho(\sigma)\geq 4$, and $\upsilon<\sigma$
such that ${\rm codim}_{\mathcal F_\sigma}\mathcal Z_\upsilon\leq 3$
one has $\mathcal F_\upsilon$  is smooth in $\mathcal{F}_\sigma$ and
$\upsilon\in Sing(\sigma)$ satisfies
$$\dim \mathcal{T}_{\mathcal{F}_\sigma}(\mathcal{F}_\upsilon)=\dim\mathcal{F}_\sigma+{\rm codim}_{\mathcal{F}_\sigma}(\mathcal{F}_\upsilon)$$
\end{corollary}

{\bf Remark}: There exists non-maximal $\sigma\in I_{n,k}$ and $\upsilon\in Sing(\sigma)$ such that ${\rm codim}_{\mathcal F_\sigma}\mathcal F_\upsilon=3$, but a straightforward computation shows that for any $\sigma\in I_{n,k}$ and $\upsilon\in G_\sigma$ such that $codim_{\mathcal{F}_\sigma}\mathcal{F}_\upsilon\leq 2$, $\mathcal{F}_\upsilon$ is smooth in $\mathcal{F}_\sigma$. This is known to be true for $\upsilon$ of $codim_{\mathcal{F}_\sigma}\mathcal{F}_\upsilon=1$, see \cite{PS2012}. For
$codim_{\mathcal{F}_\sigma}\mathcal{F}_\upsilon=2$, one can compute this combinatorially.

\section{Remarks on singular locus of $\mathcal F_\sigma$ for $\sigma\in I_{n,k}^{\max}$}\label{sect_remarks}
Note that the number of components depends heavily on the structure of $\sigma$ even for $\sigma\in I^{\max}_{n,k}$ where our algorithm provides all  $\upsilon\in Sing(\sigma).$ We can only note that for  $\sigma\in I_{n,k}^{\max}$ with $\rho(\sigma)\geq 4$ the number of components is greater or equal to $\binom{\rho(\sigma)-2}{2}$.

Indeed, we can assume that $\sigma$ is a link pattern with
$1,n\in\sigma^0$ and $|\tau^*(\sigma)|\geq 2.$  Then, by our
algorithm in the beginning of \S \ref{sec2.4} each pair
$i,j\in\tau^*(\sigma)$   gives rise to the pair of admissible arcs at least at one interval, and for $\{i,j\}\neq\{i',j'\}$ we get different pairs of admissible arcs.

In particular we get
\begin{proposition} For $\sigma\in I_{n,k}^{\max}$ with $\rho(\sigma)>4$ the number of elements in $Sing(\sigma)$ is greater or equal to 3.
\end{proposition}

In order to describe the picture in the case $\sigma\in I_{n,k}^{\max}$ with $\rho(\sigma)=4$  recall notion of $\sigma_{+a}$, complete link pattern and concatenation from Subsection \ref{maintheorem}.
Let $con(k)=(1,2k)(2,2k-1)\cdots(k,k+1)$ denote a link pattern with $2k$ points and $k$ concentric arcs.\index{$con(k)=(1,2k)(2,2k-1)\cdots(k,k+1)$, Link pattern with $2k$ points and $k$ concentric arcs} Recall that each maximal link pattern can be written as a concatenation of  maximal complete link patterns and fixed points. For instance, the link pattern $(2,3)(5,8)(6,7)(11,16)(12,13)(14,15)\in I_{18,6}$ is maximal and can be written as $(con(1)_{+1})
(con(2)_{+4})(\sigma'_{+10})$, where $con(1)$, $con(2)$ and $\sigma'=(1,6)(2,3)(4,5)$ are maximal complete link patterns.

Let $\sigma=(con(k_1)_{+a_1})(con(k_2)_{+a_2})\cdots(con(k_m)_{+a_m})\in I_{n,s}$ be a maximal link pattern where $a_1\geq1$, $a_i\geq a_{i-1}+2k_{i-1}+1$, $s=\sum_{j=1}^mk_j$ and $n\geq a_m+2k_m+1$, that is the union of sets of concentric arcs where each set is separated from the next one by nonempty set of fixed points and $1,n$ are fixed points. Then the number of components of $\sigma$ is exactly $\binom{m}{2}$. In particular, we get
\begin{proposition} Let $\sigma\in I_{n,k}^{\max}$ be such that $\rho(\sigma)=4$. Then $Sing(\sigma)$ contains the unique element.
\end{proposition}
 \begin{proof} Note that $\sigma$ is a completion of $\sigma'$ where there are only two  possibilities for $\sigma'$:
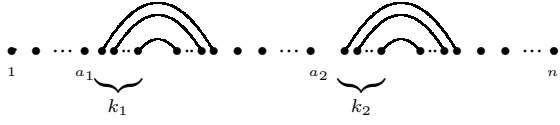
\begin{itemize}
\item $\sigma'=(con(k_1)_{+a_1})(con(k_2)_{+a_2})\in I_{n,k_1+k_2}$ with $a_1\geq 1$, $a_2\geq 2k_1+a_1+1$ and $n\geq a_2+2k_2+1$:

\begin{figure}[htp]
\begin{center}
\begin{picture}(-140,-30)(150,20)
\setlength{\unitlength}{.92pt}
\multiput(0,0)(10,0){2}{\circle*{3}}\multiput(18,0)(3,0){3}{\circle*{.5}}\put(30,0){\circle*{3}}.
\qbezier(34,0)(57,40)(80,0)\put(34,0){\circle*{3}}\put(39,0){\circle*{3}}
\qbezier(39,0)(57,30)(75,0)\multiput(43,0)(2,0){2}{\circle*{.25}}\put(49,0){\circle*{3}}
\qbezier(49,0)(57,10)(65,0)\multiput(69,0)(2,0){2}{\circle*{.25}}\put(65,0){\circle*{3}}
\put(75,0){\circle*{3}}\put(80,0){\circle*{3}}\put(-5,-10){\tiny$1$}\put(23,-10){\tiny$a_1$}
\put(31,-10){$\underbrace{\,\,\,}_{k_1}$}
\put(90,0){\multiput(0,0)(10,0){2}{\circle*{3}}\multiput(18,0)(3,0){3}{\circle*{.5}}\put(30,0){\circle*{3}}}
\put(100,0){\qbezier(34,0)(57,40)(80,0)\put(34,0){\circle*{3}}\put(39,0){\circle*{3}}
\qbezier(39,0)(57,30)(75,0)\multiput(43,0)(2,0){2}{\circle*{.25}}\put(49,0){\circle*{3}}
\qbezier(49,0)(57,10)(65,0)\multiput(69,0)(2,0){2}{\circle*{.25}}\put(65,0){\circle*{3}}
\put(75,0){\circle*{3}}\put(80,0){\circle*{3}}\put(19,-10){\tiny$a_2$}
\put(31,-10){$\underbrace{\,\,\,}_{k_2}$}}
\put(190,0){\multiput(0,0)(10,0){2}{\circle*{3}}\multiput(18,0)(3,0){3}{\circle*{.5}}\put(30,0){\circle*{3}}}
\put(217,-10){\tiny$n$}
\end{picture}
\end{center}\bigskip\bigskip\bigskip
\end{figure}

In this case the unique admissible pair is $(a_1,2k_1+a_1-1),(a_2,2k_2+a_2-1)$ and it is admissible at the unique interval $[a_1,2k_2+a_2+1].$

\item $\sigma'=(i+1,2(k_1+k_2+k_3)+i)(i+2,2(k_1+k_2+k_3+i-1))\cdots(i+k_3,i+k_3+2(k_1+k_2)+1)(con(k_1)_{+(i+k_3)})(con(k_2)_{+(i+k_3+2k_1)})\in I_{n,k_1+k_2+k_3}$ such that $i\geq 1,$ $k_1,k_2\geq 1$, $k_3\geq 0$ and $n>i+2(k_1+k_2+k_3).$

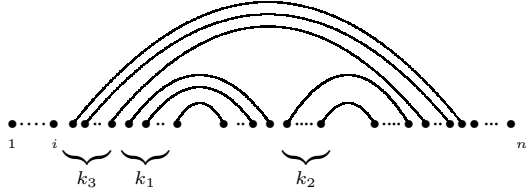
\begin{figure}[htp]
\begin{center}
\begin{picture}(-140,-60)(120,40)
\setlength{\unitlength}{.92pt}
\put(-25,0){\circle*{3}}\multiput(-21,0)(3,0){4}{\circle*{.25}}\put(-8,0){\circle*{3}}
\multiput(165,0)(15,0){2}{\circle*{3}}\multiput(170,0)(2,0){3}{\circle*{.25}}
\qbezier(0,0)(80,100)(160,0)\put(0,0){\circle*{3}}\put(160,0){\circle*{3}}
\qbezier(5,0)(80,90)(155,0)\multiput(9,0)(2,0){2}{\circle*{.25}}\multiput(149,0)(2,0){2}{\circle*{.25}}
\put(5,0){\circle*{3}}\put(155,0){\circle*{3}}
\qbezier(15,0)(80,80)(145,0)\put(16,0){\circle*{3}}\put(145,0){\circle*{3}}

\multiput(23,0)(7,0){2}{\circle*{3}}\put(43,0){\circle*{3}}\multiput(35,0)(2,0){2}{\circle*{.25}}
\put(62,0){\circle*{3}}\multiput(68,0)(2,0){2}{\circle*{.25}}
\multiput(74,0)(7,0){3}{\circle*{3}}\multiput(92,0)(2,0){4}{\circle*{.25}}\multiput(102,0)(22,0){2}{\circle*{3}}
\multiput(128,0)(2,0){4}{\circle*{.25}}\put(138,0){\circle*{3}}
\qbezier(23,0)(52,40)(81,0)
\qbezier(30,0)(52,30)(74,0)
\qbezier(42,0)(52,17)(62,0)
\qbezier(88,0)(113,40)(138,0)
\qbezier(102,0)(113,17)(124,0)
\put(-9,-10){\tiny$i$}
\put(-27,-10){\tiny$1$}
\put(182,-10){\tiny$n$}
\put(20,-10){$\underbrace{\,\,\,\,\,}_{k_1}$}
\put(-4,-10){$\underbrace{\,\,\,}_{k_3}$}
\put(86,-10){$\underbrace{\,\,}_{k_2}$}
\end{picture}
\end{center}\bigskip\bigskip\bigskip\bigskip\bigskip
\end{figure}

In this case the only admissible pair is $(i+k_3+1,i+k_3+2k_1),\, (i+k_3+2k_1+1,i+k_3+2k_1+2k_2)$ and it is admissible at the unique the interval $[i, i+2(k_1+k_2+k_3)+1]$.
\end{itemize}
\end{proof}

If $\sigma\in I_{n,k}$ is non maximal, the picture is more complex.
We would like to finish the paper with providing an example of a singular $Z_x$-orbit inside a smooth component.

\begin{example} Consider an example in $s\ell_8(\mathbb{C})$: Let $\mathcal{Z}_{\omega}$ be a $Z-$orbit with $\omega=(1,4)(2,7)(3,6)(5,8)$, and it is the component of singularity of $\mathcal{F}_T$ where
$$T=\ytableausetup{mathmode}
 \begin{ytableau}
  1 & 2\\
  3 & 4\\
  5 & 6\\
  7 & 8\\
 \end{ytableau}$$
with the  corresponding link pattern $\sigma_T=(1,2)(3,4)(5,6)(7,8)$. The corresponding admissible pair is $(3,4)(5,6)$ and interval $[2,7]$. Note that $\mathcal{Z}_{\omega}$ is also a $Z-$orbit in $\mathcal{F}_S$, where
$$S=\ytableausetup{mathmode}
 \begin{ytableau}
  1 & 4\\
  2 & 6\\
  3 & 7\\
  5 & 8\\
 \end{ytableau}$$
with the corresponding link pattern $\sigma_S=(1,8)(2,7)(3,4)(5,6)$. Since $\mathcal{F}_S$ is smooth, so $\mathcal{Z}_{\omega}$ is a smooth $\mathcal{Z}$-orbit in it (which is not surprising since the singularity of $Z-$orbit $\mathcal Z'$ in the closure of another $Z-$orbit $\mathcal Z$ depends on both $\mathcal Z$ and $\mathcal Z'$).
However, $\mathcal Z_{\omega}$ is the component of singularity in $\overline{\mathcal{Z}}_\upsilon$ with $\upsilon=(1,7)(2,8)(3,4)(5,6)$. One see at once that (1) $\mathcal{Z}_\upsilon\subset\mathcal{F}_S$ (of $codim_{\mathcal{F}_S}\mathcal{F}_\upsilon=1$), (2) $\mathcal{Z}_\omega\subset\bar{\mathcal{Z}}_\upsilon$, (3) $codim_{\overline{\mathcal{Z}}_\upsilon}\mathcal{Z}_\omega=3$. As for $\dim {\mathcal T}_{F_\omega}(\mathcal F_\upsilon)$ one can see that there are
4 predecessors of $\omega$ in $G_{\upsilon}$, see Figure \ref{figL1L2}.
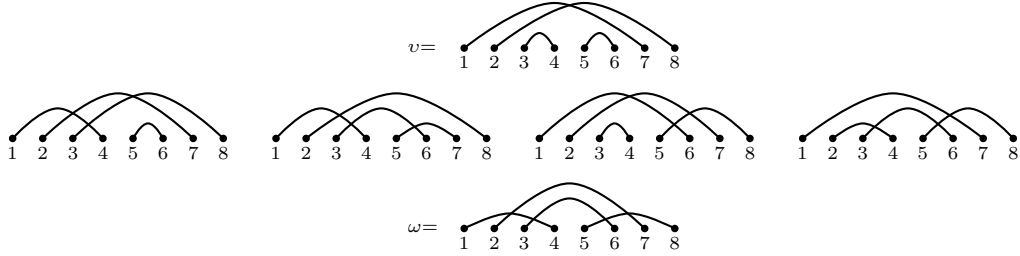
\begin{figure}[htp]
\begin{center}
\begin{pspicture}(0,-2.5)(14,1.2)
\put(6,0){\pscircle*(0,0){0.05}\pscircle*(.4,0){0.05}\pscircle*(.8,0){0.05}
\pscircle*(1.2,0){0.05}\pscircle*(1.6,0){0.05}\pscircle*(2,0){0.05}\pscircle*(2.4,0){0.05}\pscircle*(2.8,0){0.05}
\put(-.07,-.27){\scriptsize$1$}\put(.33,-.27){\scriptsize$2$}\put(.73,-.27){\scriptsize$3$}
\put(1.13,-.27){\scriptsize$4$}\put(1.53,-.27){\scriptsize$5$}\put(1.93,-.27){\scriptsize$6$}
\put(2.33,-.27){\scriptsize$7$}\put(2.73,-.27){\scriptsize$8$}
\pscurve(0,0)(1.2,.6)(2.4,0)\pscurve(.4,0)(1.6,.6)(2.8,0)
\pscurve(.8,0)(1,.2)(1.2,0)\pscurve(1.6,0)(1.8,.2)(2,0)\put(-.75,0){${}_{\upsilon=}$}}
\put(0,-1.2){\pscircle*(0,0){0.05}\pscircle*(.4,0){0.05}\pscircle*(.8,0){0.05}
\pscircle*(1.2,0){0.05}\pscircle*(1.6,0){0.05}\pscircle*(2,0){0.05}\pscircle*(2.4,0){0.05}\pscircle*(2.8,0){0.05}
\put(-.07,-.27){\scriptsize$1$}\put(.33,-.27){\scriptsize$2$}\put(.73,-.27){\scriptsize$3$}
\put(1.13,-.27){\scriptsize$4$}\put(1.53,-.27){\scriptsize$5$}\put(1.93,-.27){\scriptsize$6$}
\put(2.33,-.27){\scriptsize$7$}\put(2.73,-.27){\scriptsize$8$}
\pscurve(0,0)(.6,.4)(1.2,0)\pscurve(.4,0)(1.4,.6)(2.4,0)
\pscurve(.8,0)(1.8,.6)(2.8,0)\pscurve(1.6,0)(1.8,.2)(2,0)}
\put(3.5,-1.2){\pscircle*(0,0){0.05}\pscircle*(.4,0){0.05}\pscircle*(.8,0){0.05}
\pscircle*(1.2,0){0.05}\pscircle*(1.6,0){0.05}\pscircle*(2,0){0.05}\pscircle*(2.4,0){0.05}\pscircle*(2.8,0){0.05}
\put(-.07,-.27){\scriptsize$1$}\put(.33,-.27){\scriptsize$2$}\put(.73,-.27){\scriptsize$3$}
\put(1.13,-.27){\scriptsize$4$}\put(1.53,-.27){\scriptsize$5$}\put(1.93,-.27){\scriptsize$6$}
\put(2.33,-.27){\scriptsize$7$}\put(2.73,-.27){\scriptsize$8$}
\pscurve(0,0)(.6,.4)(1.2,0)\pscurve(.4,0)(1.6,.6)(2.8,0)
\pscurve(.8,0)(1.4,.4)(2,0)\pscurve(1.6,0)(2,.2)(2.4,0)}
\put(7,-1.2){\pscircle*(0,0){0.05}\pscircle*(.4,0){0.05}\pscircle*(.8,0){0.05}
\pscircle*(1.2,0){0.05}\pscircle*(1.6,0){0.05}\pscircle*(2,0){0.05}\pscircle*(2.4,0){0.05}\pscircle*(2.8,0){0.05}
\put(-.07,-.27){\scriptsize$1$}\put(.33,-.27){\scriptsize$2$}\put(.73,-.27){\scriptsize$3$}
\put(1.13,-.27){\scriptsize$4$}\put(1.53,-.27){\scriptsize$5$}\put(1.93,-.27){\scriptsize$6$}
\put(2.33,-.27){\scriptsize$7$}\put(2.73,-.27){\scriptsize$8$}
\pscurve(0,0)(1,.6)(2,0)\pscurve(.4,0)(1.4,.6)(2.4,0)
\pscurve(.8,0)(1,.2)(1.2,0)\pscurve(1.6,0)(2.2,.4)(2.8,0)}
\put(10.5,-1.2){\pscircle*(0,0){0.05}\pscircle*(.4,0){0.05}\pscircle*(.8,0){0.05}
\pscircle*(1.2,0){0.05}\pscircle*(1.6,0){0.05}\pscircle*(2,0){0.05}\pscircle*(2.4,0){0.05}\pscircle*(2.8,0){0.05}
\put(-.07,-.27){\scriptsize$1$}\put(.33,-.27){\scriptsize$2$}\put(.73,-.27){\scriptsize$3$}
\put(1.13,-.27){\scriptsize$4$}\put(1.53,-.27){\scriptsize$5$}\put(1.93,-.27){\scriptsize$6$}
\put(2.33,-.27){\scriptsize$7$}\put(2.73,-.27){\scriptsize$8$}
\pscurve(0,0)(1.2,.6)(2.4,0)\pscurve(.4,0)(.8,.2)(1.2,0)
\pscurve(.8,0)(1.4,.4)(2,0)\pscurve(1.6,0)(2.2,.4)(2.8,0)}
\put(6,-2.4){\pscircle*(0,0){0.05}\pscircle*(.4,0){0.05}\pscircle*(.8,0){0.05}
\pscircle*(1.2,0){0.05}\pscircle*(1.6,0){0.05}\pscircle*(2,0){0.05}\pscircle*(2.4,0){0.05}\pscircle*(2.8,0){0.05}
\put(-.07,-.27){\scriptsize$1$}\put(.33,-.27){\scriptsize$2$}\put(.73,-.27){\scriptsize$3$}
\put(1.13,-.27){\scriptsize$4$}\put(1.53,-.27){\scriptsize$5$}\put(1.93,-.27){\scriptsize$6$}
\put(2.33,-.27){\scriptsize$7$}\put(2.73,-.27){\scriptsize$8$}
\pscurve(0,0)(.6,.2)(1.2,0)\pscurve(.4,0)(1.4,.6)(2.4,0)
\pscurve(.8,0)(1.4,.4)(2,0)\pscurve(1.6,0)(2.2,.2)(2.8,0)\put(-.75,0){${}_{\omega=}$}}
\end{pspicture}
\caption{The link patterns $\omega$ and $\upsilon$, and the predecessors of $\omega$ in $G_{L_S}$.}\label{figL1L2}
\end{center}
\end{figure}

One can also note at once that these are the predecessors of $\omega$ in $G_{\sigma_S}.$
Thus $\dim\,\mathcal{T}_{F_\omega}(\mathcal F_\upsilon)=\dim\mathcal{T}_{F_\omega}(\mathcal{F}_S)=\dim \mathcal{F}_S=\dim \mathcal{Z}_\upsilon+1$.

This example provides us with three facts:
\begin{itemize}
\item A singular component in $\mathcal{Z}$ can be of codimension $3$.
\item $dim\,\mathcal{T}_{\overline{\mathcal{Z}}}(\mathcal{Z'})-\dim\mathcal Z$ can be smaller than ${\rm codim}_{\overline{\mathcal{Z}}}\mathcal{Z}'$
(indeed here it is $\dim\mathcal{Z}+1$).
\item In a smooth component there are singular $\mathcal{Z}$-orbit closures.
\end{itemize}
\end{example}
\section*{Notation}
\begin{tabular}{lp{14cm}}
\S1& $\mathbb{K}$, $V$, $\mathcal{F}$\\
\S1.1& $x$, $\mathcal{F}_x$\\
\S1.2& $\lambda$, $\lambda\vdash n$, $Y(x)=Y_\lambda$, $\lambda^*$, $Tab_\lambda$, $T_{\{i\}}$, $Y_i(T)$, $x\mid_{V_i}$, $\mathcal{V}_x^T$, $\mathcal{F}_T=\overline{\mathcal{V}_x^T}$, $\{\mathcal{F}_T\}_{T\in Tab_\lambda}$\\
\S1.3& $col(i)$, $\tau^*(T)$, $T_i$, $\tau^*(T)$, $|S|$\\
\S1.4& $Z_x$, $S_n$, ${\rm Rank}\,x$, $I_n$, $I_{n,k}$, $\sigma$-basis, $\sigma$-flag, $\sigma_T$, $\mathcal Z_\sigma$, $G_T$\\
\S2.1& $P_\sigma$, $\sigma^0$, $\sigma^\ell$, $\sigma^r$, $c_{\sigma}^r((i,j))$, $c_{\sigma}^l((i,j))$, $c(\sigma))$, $[a,b]$, $b_\sigma((i,j))$, $b(\sigma)$, $b_\sigma(f)$, $S_{[a,b]}$, $\pi_{a,b}(\sigma)$, $R(\sigma)$, $A\leq B$, $\sigma\geq \upsilon$, $Cov(a)$, 
$d_0$, $I_{n,k}^{\max}$\\
\S2.2& $\sigma_{(i_1,j_1),\ldots,(i_s,j_s)}^-$, $\sigma(i,j)=(i,j)\sigma$, $\mathcal V_{G_\sigma}$, $\mathcal E_{G_\sigma}$, $a_{\sigma,\upsilon}$, $\mathcal T_F(\mathcal V)$\\
\S3.1& $Sing(\sigma)$\\
\S3.2& $\sigma_{+a}$, $(s_1,s_2)$-maximal completion, $\sigma\upsilon_{+n}$\\
\S4.2& $\widehat \sigma$, $Gr_d(n)$, $\mathcal{H}_x$, $\mathcal B_n=\{e_i\}_{i=1}^n$, $x*y(e_i)$\\
\S5& $con(k)$
\end{tabular}

\end{document}